\documentclass{article}
\usepackage{amsmath,amsthm,amssymb,amscd}
\usepackage[all,cmtip,color]{xy}
\usepackage[german,english]{babel}
\usepackage{amsmath}
\usepackage{amssymb}
\usepackage[backref]{hyperref}
\usepackage{mathrsfs} 
\usepackage{amssymb}
\usepackage{todonotes}
\usepackage{amssymb}
\usepackage[all,cmtip]{xy}
\input diagxy
\usepackage[backref]{hyperref}
\usepackage[margin = 4.5 cm]{geometry}
\usepackage{bm}
\usepackage{bbm}
\usepackage{comment}	

\usepackage{enumerate,colonequals}

\usepackage{afterpage}

\newcommand\blankpage{%
    \null
    \thispagestyle{empty}%
    \addtocounter{page}{-1}%
    \newpage}

\newcommand{\BA}{{\mathbb{A}}}

\newcommand{\BC}{{\mathbb{C}}}

\newcommand{\BF}{{\mathbb{F}}}
\newcommand{\BG}{{\mathbb{G}}}

\newcommand{\BL}{{\mathbb{L}}}

\newcommand{\BN}{{\mathbb{N}}}

\newcommand{\BP}{{\mathbb{P}}}
\newcommand{\BQ}{{\mathbb{Q}}}
\newcommand{\BR}{{\mathbb{R}}}

\newcommand{\BV}{{\mathbb{V}}}

\newcommand{\BZ}{{\mathbb{Z}}}

\newcommand{\FA}{{\mathcal{A}}}
\newcommand{\FB}{{\mathcal{B}}}
\newcommand{\FC}{{\mathcal{C}}}

\newcommand{\FF}{{\mathcal{F}}}

\newcommand{\FI}{{\mathcal{I}}}
\newcommand{\FJ}{{\mathcal{J}}}

\newcommand{\FM}{{\mathcal{M}}}

\newcommand{\FO}{{\mathcal{O}}}
\newcommand{\FP}{{\mathcal{P}}}

\newcommand{\FV}{{\mathcal{V}}}

\newcommand{\bfc}{{\mathbf{c}}}
\newcommand{\bfv}{{\mathbf{v}}}
\newcommand{\bfw}{{\mathbf{w}}}
\newcommand{\bfl}{{\bm{\lambda}}}
\newcommand{\Hom}{{\text{Hom}}}
\newcommand{\g}{{\mathfrak{g}}}
\newcommand{\ft}{{\mathfrak{t}}}
\newcommand{\fb}{{\mathfrak{b}}}
\newcommand{\m}{{\mathfrak{m}}}
\newcommand{\h}{{\mathscr{H}}}
\newcommand{\p}{{\mathscr{P}}}
\newcommand{\sslash}{\mathbin{/\mkern-6mu/}}

\newcommand{\unt}{\mathop{\mathbbm{1}}\nolimits}

\newcommand{\proj}{{\text{Proj}}}

\newcommand{\kvar}{\mathop{\rm KVar}\nolimits}
\newcommand{\kexp}{\mathop{\rm KExpVar}\nolimits}
\newcommand{\spec}{\mathop{\rm Spec}\nolimits}
\newcommand{\expm}{\mathscr{E}xp\mathscr{M}}
\newcommand{\lla}{\left\langle }
\newcommand{\rra}{\right\rangle}
\newcommand{\n}{{\text{nil}}}
\newcommand{\re}{{\text{reg}}}
\newcommand{\im}{{\text{Im}}}
\newcommand{\Ad}{{\text{Ad}}}
\newcommand{\tr}{{\text{tr}}}
\newcommand{\Id}{{\text{Id}}}
\newcommand{\Char}{{\text{Char}}}
\newcommand{\diag}{{\text{diag}}}
\newcommand{\res}{{\text{res}}}
\newcommand{\Res}{\mathop{\rm Res}\nolimits}
\newcommand{\irr}{{\text{irr}}}
\newcommand{\ol}{\overline}
\newcommand{\End}{{\text{End}}}
\newcommand{\Hilb}{{\text{Hilb}}}

\newcommand{\GL}{\mathop{\rm GL}\nolimits}
\renewcommand{\phi}{\varphi}

\newcommand{\spn}{\mathop{\rm span}\nolimits}
\newcommand{\codim}{\mathop{\rm codim}\nolimits}
\newcommand{\rk}{\mathop{\rm rk}\nolimits}
\newcommand{\crk}{\mathop{\rm crk}\nolimits}
\newcommand{\Lie}{\mathop{\rm Lie}\nolimits}
\newcommand{\PGL}{\mathop{\rm PGL}\nolimits}
\newcommand{\Stab}{\mathop{\rm Stab}\nolimits}
\newcommand{\Rep}{\mathop{\rm Rep}\nolimits}
\newcommand{\Aut}{\mathop{\rm Aut}\nolimits}

\theoremstyle{plain}
\newtheorem{thm}{Theorem}[section]

\newtheorem{lemma}[thm]{Lemma}
\newtheorem{prop}[thm]{Proposition}
\newtheorem{defprop}[thm]{Definition/Proposition}
\newtheorem{cor}[thm]{Corollary}

\newtheorem{ass}[thm]{Assumption}

\newtheorem{conj}[thm]{Conjecture}

\theoremstyle{definition}
\newtheorem{defn}[thm]{Definition}
\newtheorem{expl}[thm]{Example}
\newtheorem{rem}[thm]{Remark}

\begin{document}

\pagenumbering{gobble}
\title{Motivic and $p$-adic Localization Phenomena}
\author{\\\\ PhD Thesis \\ Dimitri Wyss \\\\ Adviser: \\ Tam\'as Hausel }
\maketitle 
%{\\\\ Master thesis \\Dimitri Wyss\\\\Supervisors:\\ Prof. Dr. Brent Doran\\ETH Z\"urich\\\\Prof. Dr. Aravind Asok\\USC Los Angeles
%\abstract{Awesome abstract}
 \afterpage{\blankpage}
\newpage

\pagenumbering{roman}

\tableofcontents
 \afterpage{\blankpage}
\newpage

\section*{Abstract}
\addcontentsline{toc}{section}{Abstract}

In this thesis we compute motivic classes of hypertoric varieties, Nakajima quiver varieties and open de Rham spaces in a certain localization of the Grothendieck ring of varieties. Furthermore we study the $p$-adic pushforward of the Haar measure under a hypertoric moment map $\mu$. This leads to an explicit formula for the Igusa zeta function $\FI_\mu(s)$ of $\mu$, and in particular to a small set of candidate poles for $\FI_\mu(s)$. We also study various properties of the residue at the largest pole of $\FI_\mu(s)$. Finally, if $\mu$ is constructed out of a quiver $\Gamma$ we give a conjectural description of this residue in terms of indecomposable representations of $\Gamma$ over finite depth rings.

The connections between these different results is the method of proof. At the heart of each theorem lies a motivic or $p$-adic volume computation, which is only possible due to some surprising cancellations. These cancellations are reminiscent of a result in classical symplectic geometry by Duistermaat and Heckman on the localization of the Liouville measure, hence the title of the thesis. \newline

\textbf{Keywords:} Hypertoric Varieties, Nakajima Quiver Varieties, Open de Rham Spaces, Igusa Zeta Functions, Motivic Fourier Transform, Duistermaat-Heckman Theorem.

\newpage

\section*{Zusammenfassung}
\addcontentsline{toc}{section}{Zusammenfassung}
In dieser Dissertation berechnen wir die motivischen Klassen von hypertorischen Variet\"aten, Nakajima K\"ocher Variet\"aten und offenen de Rahm R\"aumen. Ausserdem studieren wir das $p$-adische Haar Mass unter dem Vorschieben einer hypertorischen Momentenabbildung $\mu$. Dies f\"uhrt zu einer expliziten Formel f\"ur die Igusa Zeta Funktion $\FI_\mu(s)$ von $\mu$ und insbesondere zu einer kleinen Menge von m\"oglichen Polen von $\FI_\mu(s)$. Wir untersuchen ausserdem verschiedene Eigenschaften des Residuums am gr\"ossten Pol von $\FI_\mu(s)$. Im Fall, wenn $\mu$ aus einem K\"ocher $\Gamma$ konstruiert ist, erkl\"aren wir schliesslich eine Vermutung, welche dieses Residuum mit der Anzahl nicht-aufteilbarer Darstellungen von $\Gamma$ \"uber Ringe endlicher Tiefe in Verbindung bringt.

Diese verschiedenen Resultate sind alle durch eine \"ahnliche Beweismethode verbunden. Im Zentrum jedes Theorems steht eine motivische oder $p$-adische Volumenberechnung, welche nur durch \"uberraschende Vereinfachungen m\"oglich ist. Diese Vereinfachungen erinnern an ein Resultat in klassischer symplektischer Geometrie von Duistermatt und Heckman \"uber die Lokalisierung des Liouville Masses, was den Title der Dissertation erkl\"art.\newline

\textbf{Stichworte:} Hypertorische Variet\"aten, Nakajima K\"ocher Variet\"aten, Offene de Rham R\"aume, Igusa Zeta Funktionen, Motivische Fourier Transformation, Duistermaat-Heckman Theorem.

\newpage

\section*{Acknowledgements}
\addcontentsline{toc}{section}{Acknowledgements}

First and foremost I would like to thank my adviser Tam\'as Hausel, who guided me during 4 years through this difficult area of mathematics and helped me in countless aspects of this thesis and my career in general. 

Chapter \ref{ch3} would not have been possible without my coauthor and friend Michael Lennox Wong, who stayed patient despite my many silly questions. 

In a similar way I'm thankful to my two mathematical brothers Daniele Boccalini and Riccardo Grandi, who were an excellent company in and outside the office. And even though not mathematically related, I would also count Alexandre Peyrot and Martina Rovelli to my closest academic family.

The members of the Hausel geometry group all have contributed in some way or an other to my thesis, mathematically and morally. I would therefore like to express my gratitude to  Martin Mereb, Zongbin Chen, Mario Garcia Fernandez, Ben Davison, Szil\'ard Szab\'o, Alexander Noll, Yohan Brunebarbe, Anton Mellit, Penghui Li, Iordan Ganev and Andras Sandor.

It cannot be stressed enough how big of a help Pierrette Paulou was in many situations during my PhD, she was truly a blessing for our geometry group at EPFL.

At many points during the past 4 years I was lucky enough to get advice, inspiration or precise answers to problems I could not solve from Bal\'{a}zs Szendr\"{o}i, Fran\c{c}ois Loeser, Andrew Morrison, Johannes Nicaise, Michael McBreen, Nicholas Proudfoot and Fernando Rodriguez Villegas. 

Finally I would like to thank my two host institutions \'Ecole Polytechnique F\'ed\'erale de Lausanne and the Institute of Science and Technology Austria for the excellent working conditions. This work was supported by the the EuropeanResearch Council [No. 320593]; and the Swiss National Science Foundation [No. 15362].

\newpage

\section*{Introduction}
\addcontentsline{toc}{section}{Introduction}

A classical theorem in symplectic geometry due to Duistermaat and Heckman \cite{DU82} states, that the volume of a symplectic reduction $M \sslash G$, where $M$ is a compact symplectic manifold and $G$ a compact Lie group, is determined by local data around the fixed points of the $G$-action on $M$. This thesis grew out of the question whether a similar localization theorem could hold in an algebraic context, were manifolds are replaced by algebraic varieties and the symplectic volume by a motivic or $p$-adic volume.

By the motivic volume of a variety we mean here simply its class in the Grothendieck ring of varieties, and the $p$-adic volume of a variety over some local field $F$ comes from the natural Haar measure on $F$. As these two volumes have very little in common with the real symplectic volume, coming from integrating the Liouville form, there is no way to apply Duistermaat's and Heckman's ideas directly. Instead we will use their theorem as a guiding principle, which we apply in each of the chapters \ref{ch2}, \ref{ch3} and \ref{localpf} to different situations.

This approach turns out to be quite fruitful and we obtain in each case explicit formulas for the respective volumes. The (sometimes conjectural) interpretations of these volumes connect our computations with the geometry of various moduli spaces, non-abelian Hodge theory, representations of quivers and even singularity theory.

The thesis is divided in 4 chapters. Chapter \ref{general} contains background material for the later chapters and no original work apart from Definition \ref{motcon} and Proposition \ref{fouconv}. In Section \ref{symg} we explain in more detail the motivation coming from symplectic geometry and the Duistermaat-Heckman theorem. In Section \ref{GCM} we recall the definition of the Grothendieck ring of varieties $\kvar$ and introduce the Grothendieck ring with exponential $\kexp$ \cite{CL, hk09}. The latter enables us to talk about motivic character sums, which leads to the notion of motivic Fourier transform. This Fourier transform is the key to exploit localization phenomena for our computations in the Chapters \ref{ch2} and \ref{ch3}. Section \ref{realize} then explains how to extract geometrical and topological information about a variety from its motivic volume. Finally, in Section \ref{hyarr} we discuss some basic combinatorics of hyperplane arrangements.

We now summarize the original results of this thesis, which are contained in the Chapters \ref{ch2}, \ref{ch3} and \ref{localpf}. More details on the individual sections are given at the beginning of each chapter. 
\subsubsection*{Motivic Classes of Symplectic Reductions of Vector Spaces}

In this chapter we study symplectic reductions of vector spaces and in particular the two prominent examples of hypertoric varieties and Nakajima quiver varieties. In Section \ref{genred} we construct these symplectic reductions algebraically using geometric invariant theory and explain how to compute in principle their motivic class (or volume) as an element in $\kvar$, or more precisely in the localization $\mathscr{M} = \kvar[\BL^{-1};(1-\BL^n)^{-1},n\geq 1]$, where $\BL=[\BA^1]$ denotes the class of $\BA^1$ in $\kvar$. The main ingredient here is Proposition \ref{prop1}, a Fourier transform argument inspired by the finite field computations in \cite{Hau1}, which exhibits some sort of localization phenomenon as explained in Remark \ref{motdh}.

Section \ref{htv} is about hypertoric varieties. In the generic case they are smooth algebraic varieties $M(\FA)$ constructed out of a hyperplane arrangement $\FA$ and their geometry is closely related to the combinatorics of $\FA$ \cite{BD00, HS02}. Using our motivic localization formalism  we prove a formula for the motivic class of $\FM(\FA)$ in terms of the intersection lattice of flats $L(\FA)$ and the M\"obius function $\nu_\FA$ (see Section \ref{hyarr} for details). 

\begin{thm}\label{them1} The motivic class of the generic hypertoric variety $\FM(\FA)$ in $\mathscr{M}$ is given by the formula
\begin{equation*} [\FM(\FA)] = \frac{\BL^{n-m}}{(\BL-1)^m}\sum_{F\in L(\FA)} \nu_\FA(F,\infty) \BL^{|F|},\end{equation*}
where $m$ denotes the rank of $\FA$ and $n=|\FA|$ the number of hyperplanes in $\FA$.
\end{thm}

In Section \ref{nqv} we repeat the same kind of computations for Nakajima quiver varieties. Here the input data is a quiver $\Gamma = (I,E,s,t)$, that is a finite vertex set $I$, a set of arrows $E\subset I\times I$ and maps $s,t:E\rightarrow I$ sending an arrow to its source and target. In \cite{nak94}\cite{nak98} Nakajima associates to $\Gamma$ and two dimension vectors $\bfv,\bfw \in \BN^I$ a smooth algebraic variety $\FM(\bfv,\bfw)$ called Nakajima quiver variety. Our main result here is again a formula for the class of $\FM(\bfv,\bfw)$ in  $\mathscr{M}$.	Explicitly let $\FP$ be the set of partitions. For $\lambda \in \FP$ we write $|\lambda|$ for its size and $m_k(\lambda)$ for the multiplicity of $k\in \BN$ in $\lambda$. Given any two partitions $\lambda,\lambda' \in \FP$ we define their inner product as $\lla \lambda,\lambda' \rra = \sum_{i,j\in \BN} \min(i,j)m_i(\lambda)m_j(\lambda').$ Then we prove
	
\begin{thm}\label{them2} For a fixed dimension vector $\bfw \in \BN^I$  the motivic classes of the Nakajima quiver varieties $\FM(\bfv,\bfw)$ in $\mathscr{M}$ are given by the generating function
	\begin{eqnarray*} \sum_{\bfv \in \BN^I} [\FM(\bfv,\bfw)] \BL^{d_{\bfv,\bfw}}T^\bfv = \frac{\sum_{\bfl \in \FP^I} \frac{\prod_{e\in E} \BL^{\lla \lambda_{s(e)},\lambda_{t(e)}\rra} \prod_{i\in I} \BL^{\lla 1^{w_i},\lambda_i \rra} }{\prod_{i\in I} \BL^{\lla \lambda_i,\lambda_i \rra} \prod_k \prod_{j=1}^{m_k(\lambda_i)} (1-\BL^{-j})}T^{|\bfl|}}{\sum_{\bfl \in \FP^I} \frac{\prod_{e\in E} \BL^{\lla \lambda_{s(e)},\lambda_{t(e)}\rra}}{\prod_{i\in I} \BL^{\lla \lambda_i,\lambda_i \rra} \prod_k \prod_{j=1}^{m_k(\lambda_i)} (1-\BL^{-j})}T^{|\bfl|}},\end{eqnarray*}
where $d_{\bfv,\bfw}$ denotes half the dimension of $\FM(\bfv,\bfw)$.
\end{thm}

We should mention that both Theorem \ref{them1} and \ref{them2} generalize known formulas for the number of points of the respective varieties over a finite field $\BF_q$, when we replace everywhere $\BL$ by $q$ \cite{Hau1,Hau2,PW07}. The real insight comes from the idea of using Grothendieck rings with exponentials in order to perform the Fourier transform computations motivically.

\subsubsection*{Open de Rham Spaces}
Chapter \ref{ch3} studies the motivic classes of open de Rham spaces and is part of a joint project with Tam\'as Hausel and Michael Wong \cite{HWW17}. In order to define an open de Rham space, we fix an effective divisor $D$ on $\BP^1$ and at each point in the support of $D$ some local data called a formal type. Then the open de Rham space $\FM_n$ is defined as the moduli space of connections on the trivial rank $n$ bundle on $\BP^1$ with poles along $D$ with prescribed formal types at the punctures. If the formal types are chosen generically, $\FM_n$ is a smooth affine variety and can be described as a symplectic fusion of coadjoint orbits for the groups $\GL_n(\BC[[z]]/z^k)$ for $k\geq 1$ \cite{Bo01}. 

The same motivic Fourier transform formalism already used in Chapter \ref{ch2} allows us to compute the class of $\FM_n$ in $\mathscr{M}$, at least when all the poles are prescribed to be of order at least $2$. We write $\FP_n$ for the set of partitions of size $n$. For a partition $\lambda = (\lambda_1\geq \dots \geq \lambda_l) \in \FP_n$ we define the numbers $l(\lambda) = l$ and $N(\lambda) = \sum_{j=1}^l \lambda_j^2$ and furthermore we put 
\[\Stab(\lambda) = \prod_{j=1}^l \GL_{\lambda_j}(\BC) \subset \GL_n(\BC).\] 

\begin{thm} If the formal types are chosen generically and all poles are prescribed to be of order at least $2$, the motivic class of $\FM_n$ in $\mathscr{M}$ is given by
\begin{align*}    [\FM_n] &= \\
\nonumber \frac{\BL^{\frac{\mathbf{k}}{2}(n^2-2n)+n(d-n)+1 } }{(\BL-1)^{nd-1}} &\sum_{\lambda \in \FP_n} \frac{(-1)^{l(\lambda)-1}(l(\lambda)-1)! (n!)^d}{\prod_{j} (\lambda_j!)^d \prod_{r \geq 1} m_r(\lambda)!} \BL^{\frac{N(\lambda)}{2}(\mathbf{k}-2d)} [\Stab(\lambda)]^{d-1},\end{align*}
where $\mathbf{k}$ denotes the degree of $D$ and $d$ the number of points in $D$.
\end{thm}

The proof uses motivic Fourier transform and convolution to reduce the computation to a simpler one at each individual pole. The condition that a pole has order at least $2$ ensures then, that the computation localizes from the whole Lie algebra $\mathfrak{gl}_n(\BC)$ to the subspace of semisimple elements.

 By passing to finite fields and using results from \cite{HLV11} the same formula holds under the weaker condition that only one pole has to be of order at least $2$. 

As we will show in \cite{HWW17}, these formulas are in agreement with the predictions of Hausel, Mereb and Wong \cite{HMW16} on the mixed Hodge polynomial of wild character varieties, which was the original motivation for studying the motivic class of $\FM_n$.

\subsubsection*{Push-forward Measures of Moment Maps over
Local Fields}

In Chapter \ref{localpf} we consider the same situation as in Chapter \ref{ch2}, but instead of the motivic measure we consider the Haar measure coming from fixing our base field to be a local field $F$ with ring of integers $\FO$. More precisely we consider as in Section \ref{htv} a hyperplane arrangement $\FA$ of rank $m$ consisting of $n$ hyperplanes and the corresponding hypertoric moment map 
\[ \mu: \FO^n\times \FO^n \rightarrow \FO^m.\]
Then $\mu$ is by definition given by a homogeneous polynomial of degree $2$ and as it turns out, a similar localization philosophy as we used in the previous chapters leads in this situation to a formula for the Igusa zeta function $\FI_\mu(s)$ of $\mu$. Here $s$ can be thought of as a complex variable, and in fact the general theory for those zeta functions will imply, that $\FI_\mu(s)$ is a rational function in $q^{-s}$, where $q$ denotes the cardinality of the the residue field of $F$ \cite{Ig00}. This rational function can be written in terms of the combinatorics of $\FA$ as follows (we refer to Section \ref{hyarr} for details on the notation). 

\begin{thm}\label{comon} For an essential hyperplane arrangement $\FA$ the Igusa zeta function $\FI_\mu(s)$ is given by

\begin{equation*} \FI_\mu(s) =  \frac{q^m-1}{q^m-q^{-s}} + \frac{1-q^{s}}{1-q^{-(s+m)}} \sum_{\infty = I_0 \supsetneq I_1 \supsetneq \dots \supsetneq I_r} q^{-\rk(\FA^{I_r})} \prod_{i=1 }^r \frac{\chi_{\FA_{I_{i-1}}^{I_{i}}}(q)}{q^{s+\delta_{I_i}}- 1},\end{equation*}
where the sum is over all proper chains of flats in $L(\FA)$ of length $r\geq 1$ and for every flat $I \in L(\FA)$ we put $\delta_I = n-|I|+ \rk(\FA_I)$.
\end{thm}

As an immediate consequence of the theorem we see that the real parts of the poles of $\FI_\mu(s)$ are amongst the negative integers $-\delta_I$ for $I \in L(\FA)$. In Section \ref{poles} we derive a sufficient criterion for when $-\delta_I$ is an actual pole of $\FI_\mu(s)$. The interest in such an analysis comes from Igusas's long standing monodromy conjecture \cite{Ig88}, which describes the poles of $I_f(s)$ for any polynomial $f$ in terms of the singularities of $f=0$.

In an other direction consider the finite rings $\FO_\alpha = \FO/ \m^\alpha$, where $\m \subset \FO$ denotes the unique maximal ideal and $\alpha \geq 1$. Then assuming that $\FA$ is coloop-free, Theorem  \ref{comon} implies that the asymptotic number of solutions to $\mu = 0$ in $\FO_\alpha^{2n}$, that is 
\[ B_\mu = \lim_{\alpha \rightarrow \infty} \frac{|\{ x \in \FO_\alpha^{2n} \ |\ \mu(x) = 0\}|}{q^{\alpha(2n-m)}},\]
converges to a rational function in $q$. After multiplying $B_\mu$ by an explicit factor we obtain a polynomial $B'_\mu(q)$. Using the functional equation for Igusa zeta functions of homogeneous polynomials \cite{DM91}, we show that $B'_\mu(q)$ is palindromic and based on numerical evidence we conjecture furthermore, that $B'_\mu(q)$ has positive coefficients, Conjecture \ref{poscof}.

Finally, when $\FA$ is a induced from a quiver $\Gamma$ it is natural to ask for a relation between the number of solutions of the moment map equation $\mu = 0$ and the number of indecomposable representations of $\Gamma$ by a theorem of Crawley-Boevey and Van den Bergh \cite{CV04}. And indeed, the polynomial $B'_\mu(q)$ seems to agree with the asymptotic number of indecomposable representations of $\Gamma$ over $\FO_\alpha$ up to a factor, see Conjecture \ref{lastone}.

\newpage
 
\pagenumbering{arabic} 
\section{Generalities}\label{general}

\subsection{Symplectic geometry}\label{symg}

%This section serves mostly as a motivation for later chapters and only a few results we mention are actually used \todo{better} later. 

A real manifold $M$ is \textit{symplectic} if it admits a non-degenerate, closed $2$-form $\omega$. For such a form to exist it is certainly necessary, that $\dim M= 2n$ is even. Then $\omega$ will induce a natural measure $dL$ on $M$, called the \textit{Liouville measure}, given by integrating the volume form $\frac{\omega^{n}}{n!}$. 

Next consider a compact Lie group $G$ acting on $M$ by symplectomorphisms, that is $g^*\omega = \omega$ for every $g\in G$. Such an action is \textit{Hamiltonian} if there exists a map 
\[ \mu: M \rightarrow \g^* = \Lie(G)^*,\]
satisfying the following two properties:
\begin{itemize}
\item For every $X \in \g$ there is an equality of $1$-forms $\lla d\mu, X \rra = i_{\xi_X}\omega$, where $\xi_X$ denotes the vector field on $M$ generated by $X$.
\item $\mu$ is $G$-equivariant for the coadjoint action of $G$ on $\g^*$.
\end{itemize}
In this case $\mu$ is called a \textit{moment map} for the $G$-action on $M$. From the first condition we deduce in particular, that $\lambda \in \g^*$ is a regular value of $\mu$ if and only if $G$ acts locally freely on $\mu^{-1}(\lambda)$. 
Hence if $\lambda \in (\g^*)^G$ and $G$ acts freely on $\mu^{-1}(\lambda)$ we can consider the quotient manifold
\[ M \sslash_\lambda G = \mu^{-1}(\lambda) /G.\]
By a theorem of Marsden and Weinstein \cite{MW74} the manifold $M \sslash_\lambda G$ is again symplectic and hence admits in particular its own Liouville measure $dL_\lambda$. The dependence of $dL_\lambda$ on $\lambda \in (\g^*)^G$ is described by a theorem if Duistermaat and Heckman \cite{DU82}, which we explain now.

As this section serves mostly as a motivation we restrict ourselves to the case where $G=T$ is a torus, $M$ is compact and the fixed point locus $M^G$ is discrete. Then for any $p \in M^G$ we get a torus action on the tangent space $T_pM$. Equipping $T_pM$ with a compatible complex structure, we write $\lambda^p_1,\dots \lambda^p_n \in \ft^*= \Lie(T)^*$ for the weights of this action. 

\begin{thm}\cite[Theorem 4.1]{DU82}\label{dhthm} For any $t \in \ft$ with $\lla \lambda^p_j, t \rra \neq 0$ for all $p\in M^G, \ 1\leq j\leq n$ the Fourier transform of the push-forward measure $\mu_*(dL)$ is given by
\begin{equation}\label{dhf} \int_{\ft^*} e^{i\lla x,t \rra} \mu_*(dL) =  \int_M e^{i\lla \mu(p),t\rra} dL = (2\pi)^n\sum_{p\in M^G} \frac{e^{i\lla \mu(p),t\rra}}{ \prod_{j=1}^n \lla \lambda^p_j,t \rra}.  \end{equation}
\end{thm} 

The slogan we take away from (\ref{dhf}) is, that the density of $\mu_*(dL)$ varies evenly on $\ft^*$ away from $\mu(M^G) \subset \ft^*$ and hence the only contributions in the Fourier transform come from $M^G$. Their theorem was later reinterpreted by Atiyah and Bott as a localization theorem in equivariant cohomology \cite{AB84}, which explains the term 'localization' in the title of this thesis.

In what follows we will however be interested in a more algebraic setting, where $M$ is a (smooth) algebraic variety over some field $k$. It still makes sense to talk about symplectic forms on $M$ and moment maps of an action of an algebraic group on $M$. In this generality the Liouville measure doesn't makes sense of course, instead one can consider the tautological motivic measure, which we introduce in Section \ref{GCM} below. Also if $k$ is a local field one can in some sense refine this motivic measure using the Haar measure on $k$. 

It turns out that in both situations there is a notion of Fourier transform with respect to these measures. Therefore it makes sense to ask whether there is some sort of localization theorem in this algebraic setting. The results in this thesis can be seen as a series of examples, which indicate that such a theorem might exist, even though there are still too many problems in order to make a precise conjecture. 

%This situation will be considered in Chapter \ref{localpf}.

%In the last Chapter \ref{localpf} we will take $k$ to be a non-archimedean local field, in which case it makes sense again 

%\subsection{Geometric Invariant Theory}

\subsection{Grothendieck rings, exponentials and Fourier transform}\label{GCM}

In this section we start by introducing various Grothendieck rings with exponentials following closely \cite{CLL}. This allows us to define a naive Fourier transform and prove a Fourier inversion formula for motivic functions. We should mention that nothing in this section is new, but rather a special case of the theory developed in \cite{CL}.
%Grothendieck rings with exponentials were first introduced in \cite{CL},\cite{hk09}, 
%In this section we follow closely \cite{CLL}, except that the Fourier transform here is ???\\

 Throughout this section let $k$ be any field. By a variety we will always mean a separated reduced scheme of finite type over $k$. The \textit{Grothendieck ring of varieties}, denoted by $\kvar$, is the quotient of the free abelian group generated by isomorphism classes of varieties modulo the relation
 \[ X - Z - U,\] 
for $X$ a variety, $Z\subset X$ a closed subvariety and $U=X\setminus Z$. The multiplication is given by $[X]\cdot [Y] = [X\times Y]$, where we write $[X]$ for the class of a variety $X$ in $\kvar$. \\
The \textit{Grothendieck ring with exponentials} $\kexp$ is defined similarly. Instead of varieties we consider pairs $(X,f)$, where $X$ is a variety and $f:X \rightarrow \BA^1=\spec(k[T])$ is a morphism. A morphism of pairs $u:(X,f) \rightarrow (Y,g)$ is a morphism $u:X \rightarrow Y$ such that $f = g \circ u$. Then $\kexp$ is defined as the free abelian group generated by isomorphism classes of pairs modulo the following relations.

% We impose three kinds of relations on the free abelian group generated by those pairs. 
\begin{enumerate}
%\item[(i)] For two varieties $X,Y$, a morphism $f:X\rightarrow \BA^1$ and an isomorphism $u:Y\rightarrow X$ the relation
%\[ (X,f) - (Y,f\circ u).\]
\item[(i)] For a variety $X$, a morphism $f:X\rightarrow \BA^1$, a closed subvariety $Z\subset X$ and $U=X\setminus Z$ the relation
\[ (X,f) - (Z,f_{|Z})- (U,f_{|U}).\]
\item[(ii)] For a variety $X$ and $pr_{\BA^1}:X\times \BA^1\rightarrow \BA^1$ the projection onto $\BA^1$ the relation
\[ (X\times \BA^1,pr_{\BA^1}).\]
\end{enumerate} 
The class of $(X,f)$ in $\kexp$ will be denoted by $[X,f]$. We define the product of two generators $[X,f]$ and $[Y,g]$ as 
\[ [X,f]\cdot [Y,g] = [X\times Y, f\circ pr_X + g\circ pr_Y],\]
where $f\circ pr_X + g\circ pr_Y: X\times Y \rightarrow \BA^1$ is the morphism sending $(x,y)$ to $f(x)+g(y)$. This gives $\kexp$ the structure of a commutative ring.

Denote by $\BL$ the class of $\BA^1$ resp. $(\BA^1,0)$ in $\kvar$ resp. $\kexp$. The localizations of $\kvar$ and $\kexp$ with respect to the the multiplicative subset generated by $\BL$ and $\BL^n-1$, where $n\geq 1$ are denoted by $\mathscr{M}$ and $\expm$.

For a variety $S$ there is a straight forward generalization of the above construction to obtain the \textit{relative Grothendieck rings} $\kvar_S,\kexp_S,\mathscr{M}_S$ and $\expm_S$. For example generators of $\kexp_S$ are pairs $(X,f)$ where $X$ is a $S$-variety (i.e. a variety with a morphism $X\rightarrow S$) and $f:X \rightarrow \BA^1$ a morphism. The class of $(X,f)$ in $\kexp_S$ will be denoted by $[X,f]_S$ or simply $[X,f]$ if the base variety $S$ is clear from the context. 

There is a natural map
\begin{align*} \kvar_S &\rightarrow \kexp_S\\
							[X] &\mapsto [X,0]
\end{align*}
and similarly $\mathscr{M}_S \rightarrow \expm_S$, which are both injective ring homomorphisms by \cite[Lemma 1.1.3]{CLL}. Hence we do not need to distinguish between $[X]$ and $[X,0]$ for a $S$-variety $X$.

%\begin{lemma}\cite[Lemma 1.1.3]{CLL} For any variety $S$ the natural ring homomorphisms $\kvar_S \rightarrow \kexp_S$ and $\mathscr{M}_S\rightarrow \expm_S$ are injective.
%\end{lemma}
For a morphism of varieties $u: S \rightarrow T$ we have induced maps
\begin{align*} u_!&:\kexp_S \rightarrow \kexp_T, \ \ \ [X,f]_S \mapsto [X,f]_T\\
u^*&:\kexp_T \rightarrow \kexp_S, \ \ \ [X,f]_T \mapsto [X \times_T S, f \circ pr_X]_S. \end{align*}
In general $u^*$ is a morphism of rings and $u_!$ a morphism of additive groups. However it is straightforward to check that for any $u:S \rightarrow T$ and any $\phi\in \kexp_S$ we have
\begin{eqnarray}\label{genu} u_!(\BL \cdot\phi)=\BL\cdot u_!(\phi),
\end{eqnarray}
where $\BL$ denotes the class of $\BA^1\times S$ and $\BA^1\times T$ in $\kexp_S$ and $\kexp_T$ respectively.

Elements of $\kexp_S$ can be thought of as motivic functions on $S$. The evaluation of $\phi \in \kexp_S$ at a point $s: \spec(k) \rightarrow S$ is simply
\[s^*(\phi)\in \kexp_{\spec(k)}=\kexp.\]
Computations with these motivic functions can sometimes replace finite field computations. More precisely let $\BF_q$ be a finite field and fix a non-trivial additive character $\Psi:\BF_q \rightarrow \BC^\times$. Assume that $S$, $X\rightarrow S$ and $f:X \rightarrow \BA^1$ are also defined over $\BF_q$. Then the class of $(X,f) \in \kexp_S$ corresponds to the function 

\begin{equation}\label{real1} S(\BF_q) \rightarrow \BC, \ \ \ s \mapsto \sum_{x \in X_s(\BF_q)} \Psi(f(x)).\end{equation}
Furthermore for a morphism $u: S \rightarrow T$ the operations $u_!$ and $u^*$ correspond to summation over the fibres of $u$ and composition with $u$ respectively.

There is a slight technical disadvantage to working over finite fields. Namely given a fibration $f:X \rightarrow Y$ where each fiber is isomorphic to some fixed variety $F$, we cannot deduce in general 
\begin{eqnarray}\label{trif} [X]=[F][Y] \end{eqnarray}
in $\kvar$ or $\mathscr{M}$, whereas a similar relation clearly holds over a finite field. However (\ref{trif}) holds if the fibration is Zariski-locally trivial i.e. $Y$ admits an open covering $Y=\cup_{j} U_j$ such that $f^{-1}(U_j) \cong F\times U_j$. Indeed, in this case we have 
\[[X] = \sum_{j} [f^{-1}(U_j)] - \sum_{j_1 < j_2} [f^{-1}(U_{j_1} \cap U_{j_2})] +... = [F][Y].\] 

 Next we discuss an analogue of the crucial identity for computing character sums over finite fields
\begin{eqnarray*} \sum_{v\in V} \Psi(f(v)) = \begin{cases} q^{\dim(V)} &\text{ if } f=0\\
0 &\text{ else,}\end{cases}\end{eqnarray*}
where $V$ is a $\BF_q$ vector space and $f\in V^*$ a linear form.

To establish a similar identity in the motivic setting we let $V$ be a finite dimensional vector space over $k$ and $S$ a variety. We replace the linear form above with a family of affine linear forms i.e. a morphism $g=(g_1,g_2):X \rightarrow V^*\times k$, where $X$ is an $S$-variety. Then we define $f$ to be the morphism
\begin{align*} f:X\times V &\rightarrow k 	\\
								(x,v) &\mapsto \left\langle g_1(x),v\right\rangle + g_2(x).
\end{align*}
Finally, we put $Z= g_1^{-1}(0)$. 

\begin{lemma}\label{orth} With the notation above we have the relation
\[ [X\times V,f] = \BL^{\dim V}[Z,{g_2}_{|Z}] \]
in $\kexp_S$. In particular, if $X = \spec(k)$ and $f \in V^*$, we have $[V,f] = 0$ unless $f = 0$.
\end{lemma}
\begin{proof} By using (\ref{genu}) we may assume $S=X$. Now because of \cite[Lemma 1.1.8]{CLL} it is enough to check for each point $x\in X$ the identity 
\[x^*([X\times V,f]) = x^*(\BL^{\dim V}[Z,{g_2}_{|Z}])\] and this is exactly Lemma $1.1.11$ of \textit{loc. cit.}
\end{proof}

%The next lemma establishes this identity in the motivic setting and is basically lemma $1.1.11$ of \cite{CLL}.
%\begin{lemma}\label{orth} Let $X$ be a $S$-variety and $f:X\times_k \BA^n_k \rightarrow \BA^1_k$ a morphism such that for each $x\in X$ the restriction $f_x:\BA^n_{k(x)} \rightarrow \BA^1_{k(x)}$ is affine linear. Assume that $Z = \{x\in X \ | \ f_x \text{ is constant%}\}$ is a variety (is probably always te case). Then we have a well defined function $f_Z:Z \rightarrow \BA^1_k$ and the relation 
%\[ [X \times_k \BA^n_k,f] = \BL^n[Z,f_Z] \]
%holds in $\kexp_S$.
%\end{lemma}
%\begin{proof} By using \ref{genu} we may assume $S=X$. We have a decompostion $[X \times \BA^n,f]_X = [Z\times \BA^n,f]_X \sqcup [U\times \BA^n,f]_X$ with $U = X\setminus Z$. Now $[U\times \BA^n,f]_X = 0$, because by $1.1.8$ from \cite{CLL} it is enough to check this on each point of $U$ seperatly and there it follows from lemma $1.1.11$ of loc.cit. The equality $[Z\times \BA^n,f]_X = \BL^n[Z,f_Z]_X$ is proven similarly.
%\end{proof}

Now we are ready to define a \textit{naive motivic Fourier transform} for functions on a finite dimensional $k$-vectorspace $V$ and prove an inversion formula. All of this is a special case of \cite[Section 7.1]{CL}.

\begin{defn}\label{nft} Let $p_V:V\times V^* \rightarrow V$ and $p_{V^*}:V\times V^* \rightarrow V^*$ be the obvious projections. 
\textit{The naive Fourier transformation} $\mathcal{F}_V$ is defined as 
\begin{align*} \mathcal{F}_V:\kexp_{V} &\rightarrow \kexp_{V^*} \\
								\phi &\mapsto p_{V^*!}(p_V^*\phi \cdot[V\times V^*,\lla,\rra]). \end{align*}
								Here $\left\langle ,\right\rangle:V\times V^* \rightarrow k$ denotes the natural pairing. We will often write $\mathcal{F}$ instead of $\mathcal{F}_V$ when there's no ambiguity.
\end{defn}

Of course the definition is again inspired by the finite field version, where one defines for any function $\phi: V \rightarrow \BC$ the Fourier transform at $w \in V^*$ by
\[ \FF(\phi)(w) = \sum_{v \in V } \phi(v) \Psi(\lla w,v\rra).\]

Notice that $\mathcal{F}$ is a homomorphism of groups and thus it is worth spelling out the definition in the case when $\phi = [X,f]$ is the class of a generator in $\kexp_V$. Letting $u:X\rightarrow V$ be the structure morphism we simply have 
\begin{eqnarray}\label{ftgen}
\mathcal{F}([X,f]) = [X \times V^*, f\circ pr_X+\left\langle u\circ pr_X,pr_{V^*}\right\rangle].
\end{eqnarray} 

Now we are ready to prove an inversion formula for the naive Fourier transform.
\begin{prop}\label{finv} For every $\phi \in \kexp_{V}$  we have the identity
\[\mathcal{F}(\mathcal{F}(\phi)) = \BL^{\dim(V)} \cdot i^*(\phi),\]
where $i:V\rightarrow V$ is multiplication by $-1$.
\end{prop}

%Before we go to the proof, let's fix some notation. For $Z$-varieties $X \rightarrow Z$ and $Y\stackrel{u}{\rightarrow} Z$ we'll sometimes write $X \times_Z^v Y$ for their fibre product, when there are several possible morphisms $Y \rightarrow Z$. Also we'll always denote by $pr_i$ the porjection onto the $i$-th coordinate in a fibre-product $X_1 \times_{Z_1} X_2 \dots \times_{Z_{n-1}} X_n$. To compute Fourier transforms we will often use the observation, that we have for every $\BA^n_k$-variety $X$ we have an isomorphism of $\BA^n_k$-varieties
%\begin{eqnarray}\label{simp} X \times_{\BA^n}^{p_1} \BA^{2n} \cong X \times \BA^n, \end{eqnarray}
%where the structure morphisms on the lhs and rhs are $p_2 \circ pr_2$ resp. $pr_2$.

\begin{proof} Since $\mathcal{F}$ is a group homomorphism it is enough to prove the lemma for $\phi = [X,f]$ with $X\stackrel{u}{\rightarrow}V $.  Iterating (\ref{ftgen}) we get
\[\mathcal{F}(\mathcal{F}([X,f]))= [X\times V \times V^*, f\circ pr_X + \lla u\circ pr_X+pr_V,pr_{V^*}\rra].\]
Now we can apply Lemma \ref{orth} with $Z=\{ (x,v)\in X \times V\  |\ u(x) +v=0\}$ to obtain
\[[X\times V \times V^*, f\circ pr_X + \lla u\circ pr_X+pr_V,pr_{V^*}\rra] = \BL^{\dim V^*} [Z, f\circ pr_X].\]
Notice that $Z$ is a $V$-variety via projection onto the second factor and hence the projection onto the first factor induces a $V$-isomorphism $Z \cong (X\stackrel{i\circ u}{\rightarrow} V)$, which gives the desired result. 	
\end{proof}

Finally, we introduce a motivic version of convolution. 
\begin{defn}\label{motcon} Let $R:\kexp_V\times \kexp_V \rightarrow \kexp_{V\times V}$ be the natural morphism sending two varieties over $V$ to their product, and $s: V \times V \rightarrow V$ the sum operation. The \textit{convolution product} is the associative and commutative operation
\begin{align*} *: \kexp_V\times \kexp_V &\rightarrow \kexp_V\\
									(\phi_1,\phi_2) &\mapsto \phi_1 * \phi_2 = s_! R(\phi_1,\phi_2).
									\end{align*}
\end{defn}
As expected the Fourier transform interchanges product and convolution product.
\begin{prop} \label{fouconv} For $\phi_1,\phi_2 \in \kexp_V$ we have
\[ \FF(\phi_1 *\phi_2) = \FF(\phi_1)\FF(\phi_2).\]
\end{prop}
\begin{proof} As both $\FF$ and $*$ are bilinear it is enough to consider generators $[X,f],[Y,g] \in \kexp_V$ with structure morphisms $u:X\rightarrow V, v:Y \rightarrow V$ respectively. Using \eqref{ftgen} we can then directly compute
\begin{align*} \FF( [X,f] &*[Y,g] ) = \FF(s_![X\times Y,f\circ pr_X + g\circ pr_Y])\\
																			&= [X\times Y \times V^*,f\circ pr_X + g\circ pr_Y + \lla s \circ (u,v) \circ pr_{X\times Y},pr_{V^*}\rra ]\\
																			&= [X \times V^*,f\circ pr_X + \lla u\circ pr_X,pr_{V^*}\rra][Y \times V^*, g\circ pr_Y+ \lla v\circ pr_Y,pr_{V^*}\rra] \\
																			&= \FF[X,f]\FF[Y,g].\end{align*}
\end{proof}

We will use the convolution product to study equations in a product of varieties i.e. consider $V$-varieties $u_i:X_i \rightarrow V$ say for $i=1,2$. Then it follows from the definition of $*$, that for any $v: \spec(k) \rightarrow V$ the class of $\{ (x_1,x_2) \in X_1 \times X_2 \ |\ u_1(x_1) +u_2(x_2) = v\}$ is given by $v^*([X_1] *[X_2])$. Through Proposition \ref{fouconv} we can compute the latter by understanding the Fourier transforms $\FF(X_1), \FF(X_2)$ separately.

\subsection{Realization morphisms}\label{realize}

As we will see later, the rings $\kvar$ and $\kexp$ and their localizations $\mathscr{M}, \expm$ are quite convenient to compute in. However it is a priori unclear how much geometric information we can extract out of these computations. Typically this is done via morphisms to other rings, which are better understood. We  will not give all the details of the constructions we mention here and refer to \cite{Po11} and also \cite[Appendix]{HR08}.

As a first example, when $k=\BF_q$ is a finite field, (\ref{real1}) defines such a morphism $\#: \kexp \rightarrow \BC$. Notice that $\BL$ and $\BL^n-1$ get mapped to $q$ and $q^n-1$ under $\#$, hence $\#$ even lifts to $\expm$.

Since our focus is more on the topology of complex varieties, we will be interested in realization morphism over a fields of characteristic $0$. Assuming the transcendence degree of $k/\BQ$ is at most the one of $\BC/\BQ$ we can embed $k$ into $\BC$ and consider any variety over $k$ as a variety over $\BC$. This way the $\BC$ points of any variety inherit a topology from $\BC$ and taking the Euler characteristic of that space gives a ring homomorphism
\[  e: \kvar \rightarrow \BZ. \]

Notice that in this case we do not have an extension to $\mathscr{M}$, as for example $e(\BL-1) = 0$. 

The most interesting realization morphism for us relies on two natural filtrations, called the weight and the Hodge filtration, on the compactly supported cohomology $H^*_c(X,\BC)$ of any complex algebraic variety $X$, which were constructed by Deligne \cite{De71, De74}. Taking the dimensions of the graded pieces we obtain the (compactly supported) mixed Hodge numbers of $X$
\[ h_c^{p,q;i}(X) = \dim_\BC \left(Gr^H_p Gr^W_{p+q} H^i_c(X,\BC)\right). \]
Out of these numbers we form the \textit{$E$-polynomial}, a refined Euler characteristic defined by
\begin{equation}\label{epd} E(X;x,y) = \sum_{p,q,i \geq 0} (-1)^i  h_c^{p,q;i}(X) x^py^q.\end{equation}
As it turns out we obtain this way a morphism
\begin{align}\label{reale} \kvar &\rightarrow \BZ[x,y]\\
						\nonumber		[X] &\mapsto E(X;x,y).
\end{align}

\begin{expl} If $X=\BA^1$ is the affine line, then $H_c^*(X,\BC) = H_c^2(X,\BC) \cong \BC$ and the generator is of type $(1,1)$, hence $E(X;x,y) = xy$. Since $E$ descends to a morphism on $\kvar$ we can use the cut and paste relations to compute further
\begin{align*} E(\BA^n) &= E(\BA^1)^n = (xy)^n,\\
 E(\BP^n) &= E(\BP^{n-1}) + E(\BA^n) = (xy)^n+(xy)^{n-1}+ \dots + xy + 1.\end{align*}
\end{expl}

From the example we see in particular that $E$ will extend to a morphism 
\[E:\mathscr{M} \rightarrow \BZ \left[x,y,\frac{1}{xy};\frac{1}{(xy)^n-1},n\geq 1\right].\]
However we do not know how to extend $e$ or $E$ to $\kexp$, which is an interesting question in its own right. 

We finish with a lemma explaining how to extract some simple topological information out of the $E$-polynomial.

\begin{lemma}\cite[Lemma 5.1.2]{HLV13}\label{connoc} Let $X$ be smooth and equidimensional of dimension $d$. Then $E(X;t,t)$ is a polynomial of degree $2d$ and the coefficient of $t^{2d}$ is the number of connected components of $X$.
\end{lemma}

\begin{comment}
 Assume there is a polynomial $f \in \BZ[t]$ such that $[X] = f(\BL) \in \kvar$ (or $\mathscr{M}$). Then
\begin{enumerate}
\item[(i)] The degree of $f$ is $d$
\item[(ii)] $X$ is irreducible if and only if $f$ is monic.
\item[(iii)] If the mixed Hodge structure on $H^*_c(X,\BC)$ is pure i.e. $h_c^{p,q;i}(X) = 0$ unless $p+q=i$, then $f \in \BN[t]$. 
\end{enumerate}
\end{lemma}	
\begin{proof} Since $X$ is smooth, the number of irreducible components is the same as the number of connected components. The latter is given by the dimension of $H^0(X,\BC)$. 
\end{proof}
\end{comment}

\subsection{Hyperplane arrangements}\label{hyarr}

A general reference for the theory of hyperplane arrangements is \cite{St04} and we'll also use conventions from \cite{PW07}.

Let $k$ be a field. Generally speaking a hyperplane arrangement in a finite dimensional $k$-vector space $V$ is a union of finitely many affine hyperplanes in $V$. For us it will be important to consider hyperplanes together with a fixed normal vector, which will always be integral.

More precisely we fix a basis of $V$ and consider it's dual $V^*$ with the dual basis. If we write $m$ for the dimension of $V$, we have a natural morphism $\BZ^m \rightarrow V^*$, which is injective if $\Char(k) = 0$. Now consider non-zero vectors $a_1,\dots,a_n \in \BZ^m$ and denote their images in $V^*$ by the same letters (if $\Char(k)>0$ we'll always assume that the images of the $a_i$'s are non-zero). We further fix elements $r_1,\dots,r_n \in k$ and denote by $\FA_{\mathbf{a},\mathbf{r}}$ the hyperplane arrangement in $V$ consisting of the hyperplanes
\[ H_i = \{ v \in V\ |\ \lla v,a_i \rra = r_i\} \text{ for } 1 \leq i \leq n.\]

Most of the time we will just write $\FA$ for an arrangement and assume implicitly a choice $(\mathbf{a},\mathbf{r})$. 

\begin{defn} A hyperplane arrangment $\FA$ is called
\begin{itemize}
\item \textit{central}, if all hyperplanes are linear subspaces i.e. $r_i = 0$ for all $i$,
\item \textit{essential}, if $\spn\{a_1,\dots,a_n\} = V^*$,
\item \textit{unimodular}, if every collection of $m$ linearly independent vectors $a_{i_1},\dots,a_{i_m}$ spans $\BZ^m$ over the integers,
%\item \textit{simple}, if for every subset $I \subset \{1,\dots,n\}$ we have either $\cap_{i\in I} H_i = \emptyset$ or $\codim\cap_{j\in J} H_j = |J|$, 
%\item \textit{smooth}, if it is both simple and unimodular.
\end{itemize}
\end{defn}

The rank of $\FA$ is defined as $\rk(\FA)= \dim \spn\{a_1,\dots,a_n\}$. Most of the time we consider essential arrangements, in which case $\rk(\FA) = \dim V$. 

For a subset $F \subset \{1,\dots,n\}$ put $H_F = \bigcap_{i\in F} H_i$. The subset $F$ is called a \textit{flat} if $H_F \neq \emptyset$ and $F=\{i\ | \ H_F \subset H_i\}$. The set of all flats $L(\FA)$ is a partially ordered set with the relation $F \leq G$ if $F \subset G$. The unique minimal element in $L(\FA)$ is the empty flat $\emptyset$ with $H_\emptyset = V$ and a unique maximal element exists if and only if $\FA$ is central, in which case we denote it by $\infty$. The rank $\rk(F)$ of a flat $F$ is defined as the dimension of $\spn_{i\in F}\{a_i\}$ and the corank as $\crk(F) = m - \rk(F)$. 

The M\"obius function of the poset $L(\FA)$ 
\begin{equation}\label{mobius} \nu_\FA:L(\FA) \times L(\FA) \rightarrow \BZ\end{equation}
is defined inductively by 	
\begin{align*} \nu_\FA(F,F) &= 1, \text{ for every } F \in L(\FA).\\
								\sum_{F \leq G \leq F'} \nu_\FA(F,G) &= 0, \text{ for all } F < F'.
\end{align*}
We also abbreviate $\nu_\FA(F)=\nu_\FA(\emptyset,F)$ for every $F \in L(\FA)$.

We give now a description of $\nu_\FA$ in terms of chains. A \textit{chain} of flats in $L(\FA)$ is a sequence $F_0 \subsetneq F_1 \subsetneq \dots \subsetneq F_l$ of flats where all the inclusions are strict. The integer $l$ is called the \textit{length} of the chain. For two flats $F \subset F'$ write $N_l(F,F')$ for the number chains $F_0 \subsetneq F_1 \subsetneq \dots \subsetneq F_l$ of length $l$ with $F_0=F$ and $F_l=F'$.

\begin{lemma}\label{nuchain} For any two flats  $F \subset F'$ in $L(\FA)$ we have
\[ \nu_\FA(F,F')  = \sum_{l \geq 0} (-1)^l N_l(F,F').\]
\end{lemma}
\begin{proof} We simply define $\widetilde{\nu}_\FA(F,F') = \sum_{l \geq 0} (-1)^l N_l(F,F')$ and show that it satisfies the same recursion as $\nu_\FA$. First there's only one chain of length $0$ from $F$ to $F$ for any $F \in L(\FA)$ and hence $\widetilde{\nu}_\FA(F,F)=1$. More generally for any $F \subset F'$ we have 
\[ N_l(F,F') = \sum_{F \leq G < F'} N_{l-1}(F,G)\]
and hence 
\begin{align*}\widetilde{\nu}_\FA(F,F') &= \sum_{l \geq 0} (-1)^l N_l(F,F')\\
& = -\sum_{l\geq 0} \sum_{F \leq G < F'} (-1)^{l-1} N_{l-1}(F,G) = -\sum_{F \leq G < F'}\widetilde{\nu}_\FA(F,G). \end{align*}
\end{proof}

An other important invariant of $L(\FA)$ is the \textit{characteristic polynomial} $\chi_\FA(t) \in \BZ[t]$ given by
\begin{equation}\label{chadef} \chi_\FA(t) = \sum_{F \in L(\FA)} \nu_\FA(F) t^{\crk(F)}.\end{equation}

By definition it is monic of degree $m$. A priory $\chi_\FA(t)$ depends on the field $k$, but it turns out that the dependence is rather mild. We define the $k$-variety $c(\FA) = V \setminus \bigcup_{H\in \FA} H$. 

\begin{thm}\cite[Theorem 5.15]{St04}\label{stcha} For any ground field $k$ we have in $\kvar_k$ the equality
\[ [c(\FA)] = \chi_\FA(\BL).\]
In particular, for $k=\BF_q$ a finite field we have
\[ |c(\FA)(\BF_q)| = \chi_\FA(q).\]
Furthermore, away from finitely many positive characteristics $L(\FA)$ and $\chi_\FA(t)$ are independent of the field $k$. 
\end{thm}
\begin{proof} In \cite[Theorem 5.15]{St04} the finite field case is proven, but the same proof works also motivically. The last statement is Proposition 5.13 in \textit{loc. cit.}
\end{proof}

\begin{cor}\label{fq1} The characteristic polynomial of any central hyperplane arrangement is divisible by $(t-1)$.
\end{cor}
\begin{proof} The complement $c(\FA)$ of a central arrangement $\FA$ admits a $k^\times$-action by scaling. Hence if $k=\BF_q$ is a finite field, this implies that $|c(\FA)(\BF_q)|$ is divisible by  $|\BF_q^\times|=q-1$ and then statement now follows from Theorem \ref{stcha}.
\end{proof}

Assume now that $\FA$ is central. For a flat $F$ the \textit{localization} $\FA_F$ is defined as 
\[ \FA_F = \{ H_i/H_F\ |\ i \in F\},\]
which is an arrangement in $V/H_F$ and we have $\rk(\FA_F) = \rk(F)$. The intersection lattice $L(\FA_F)$ can be identified with the sub-lattice of $L(\FA)$ consisting of flats contained in $F$.

Dually, the \textit{restriction} $\FA^F$ is defined as
\[ \FA^F = \{ H_i \cap H_F\ |\ i\notin F\},\]
which is an arrangement in the vector space $H_F$. If $\FA$ is essential we also have $\rk(\FA^F) = \crk(F)$. Similarly $L(\FA^F)$ can be identified with the sub-lattices of $L(\FA)$ consisting of flats containing $F$. For two flats  $F \subset F'$ we write $\FA_{I'}^I= (\FA_{I'})^I \cong (\FA^I)_{I'}$.

\begin{prop}\label{surpos} Let $\FA$ be a central hyperplane arrangement and $F \subset F'$ two flats in $L(\FA)$. Then we have
\[ (-1)^{\rk(\FA^I_{I'})}\nu_{\FA}(I,I') > 0.\]
\end{prop}
\begin{proof} The statement depends only on $L(\FA)$, hence we may assume, that $\FA$ is essential i.e. $\rk(\FA) = m$. Furthermore by replacing $\FA$ with $\FA^I_{I'}$ we can assume $I=\emptyset, I' = \infty$. Then by definition we have $\nu_{\FA}(\emptyset,\infty) = \chi_\FA(0)$. 

For any $1\leq i\leq n$ we denote by $F^i \in L(\FA)$ the flat corresponding to the single hyperplane $H_i$ and by $\FA \setminus F^i$ the arrangement we obtain by removing $F^i$ from $\FA$. Then $\chi_\FA$ satisfies the deletion-restriction relation
\[ \chi_\FA(t) = \chi_{\FA\setminus F^i}(t) - \chi_{\FA^{F^i}}(t). \]
This is proven in \cite[Lemma 2.2]{St04} or can also be seen from Theorem \ref{stcha}. The statement now follows by induction on $\rk(\FA)=m$. If $m=0$ there is nothing to prove. For higher rank the deletion-restriction equation for any $1 \leq i\leq n$ implies 
\[ (-1)^{m} \nu_{\FA}(\emptyset,\infty) = (-1)^{\rk(\FA)}\chi_{\FA\setminus F^i}(0) + (-1)^{m-1}\nu_{\FA^{F^i}}(\emptyset,\infty), \]
as $\FA^{F^i}$ is again essential. Now if $\FA\setminus F^i$ is not essential we have $\chi_{\FA\setminus F^i}(0)=0$ and we are done by the induction hypothesis, as $\rk(\FA^{F^i}) = m-1$. Otherwise we repeat the deletion restriction argument for $\FA\setminus F^i$ until we reach a non-essential arrangement.
\end{proof}

\newpage
\section{Motivic Classes of Symplectic Reductions of Vector Spaces}\label{ch2}

In this chapter we give formulas for the classes of certain hypertoric and Nakajima quiver varieties in the localized Grothendieck ring $\mathscr{M}$ introduced in \ref{GCM}. These classes are given by a polynomial expression in $\BL = [\BA^1]$ and are the motivic analogue of the point count of these varieties obtained by Hausel in \cite{Hau1,Hau2}. 

Both hypertoric and Nakajima quiver varieties arise as symplectic GIT-reduction of an algebraic symplectic vector space and many of the arguments work in this general setting, which is the content of the first section. In particular, we prove in Proposition \ref{prop1} a motivic analogue of \cite[Proposition 1]{Hau1}, which was the main reason to introduce the naive Fourier transform in \ref{GCM}. We also explain, how one can interpret this as an example of 'motivic localization' in the spirit of Theorem \ref{dhthm}.

In the subsequent Sections \ref{htv} and \ref{nqv} we start by defining hypertoric and Nakajima quiver varieties respectively and then compute their motivic classes explicitly. Having established Proposition \ref{prop1} this simply amounts to redo carefully Hausel's finite field computations in the Grothendieck ring of varieties.

Throughout the whole chapter $k$ will denote an algebraically closed field of characteristic $0$

\subsection{Symplectic reductions of vector spaces}\label{genred}
The following constructions are well known and appear for example in \cite{Pro07}. For a detailed account on GIT quotients and the relation with symplectic reduction we refer to \cite{MFK94} and for our purposes also \cite{Kin94}. 

Let $G$ be a reductive algebraic group over $k$ with Lie algebra $\mathfrak{g}$ and $\rho:G \rightarrow\GL(V)$ a representation, where $V$ is some finite dimensional $k$-vector space. We will always assume that $\rho$ is injective and hence the action of $G$ on $V$ effective. The derivative of $\rho$ is the Lie algebra representation $\varrho:\mathfrak{g} \rightarrow \mathfrak{gl}(V)$. The action of $G$ on $V$ induces a Hamiltonian action on $T^*V \cong V \times V^*$ which has a moment map $\mu:V\times V^* \rightarrow \g^*$ given by

\begin{eqnarray}\label{mmap} \left\langle \mu(v,w),X \right\rangle = \left\langle  \varrho(X)(v),w \right\rangle,
\end{eqnarray}
for $(v,w) \in V\times V^*$ and $X \in \mathfrak{g}$.

Now for any character $\chi\in \Hom(G,\BG_m)$ and any $\xi \in (\g^*)^G$ we'll be interested in the GIT quotient
\[ \FM_{\chi,\xi} = \mu^{-1}(\xi) \sslash_\chi G.\]
Here the right hand side is defined as
\[\mu^{-1}(\xi) \sslash_\chi G = \proj \bigoplus_{m=0}^\infty k[\mu^{-1}(\xi)]^{G,\chi^m},\]
with
\[k[\mu^{-1}(\xi)]^{G,\chi^m} = \{ f \in k[\mu^{-1}(\xi)] \ |\ f \circ \rho(g) = \chi(g)^mf \ \ \  \forall g\in G\}.\]

In particular, for $\chi = 0$ the trivial character $\FM_{0,\xi}$ is the affine variety defined by
\[ \FM_{0,\xi} = \spec k[\mu^{-1}(\xi)]^G.\]
 
For general $\chi$ there is always an affinization map 
\begin{equation}\label{affmap}\FM_{\chi,\xi} \rightarrow \FM_{0,\xi},\end{equation}
which is proper.

There is a more geometric way to describe $\FM_{\chi,\xi}$ in terms of $\chi$-(semi)stable points. Instead of giving the definition of (semi)stable points we will directly state 'Mumford's numerical criterion' characterizing them. For the purpose of this thesis Proposition \ref{kinmum} below may thus be taken as a definition. 

A \textit{one-parameter subgroup} of $G$ is an inclusion $\lambda:\BG_m \hookrightarrow G$. For any character $\chi: G \rightarrow \BG_m$ we define the paring $\lla \chi,\lambda \rra \in \BZ$ by $\chi(\lambda(t)) = t^{\lla \chi,\lambda \rra }$.

\begin{defprop}\cite[Proposition 2.5]{Kin94}\label{kinmum} A point $(v,w) \in V \times V^*$ is called $\chi$-semistable and $\chi$-stable respectively if every one-parameter subgroup $\lambda$ of $G$, for which $\lim_{t \rightarrow 0} \lambda(t)(v,w)$ exists, satisfies $\lla \chi,\lambda \rra \geq 0$ and $\lla \chi,\lambda \rra > 0$ respectively.
\end{defprop}

For any subvariety $Z \subset V\times V^*$ we will write $Z^{\chi-ss} \subset Z$ for the open subvariety of semi-stable points. We also define an equivalence relation on $(V\times V^*)^{\chi-ss}$ by
\[(v_1,w_1) \sim (v_2,w_2) \ \Leftrightarrow \ \overline{G(v_1,w_1)} \cap \overline{G(v_2,w_2)} \neq \emptyset \ \text{ inside } \ (V\times V^*)^{\chi-ss}. \]

With this notation we can describe (the points of) $\FM_{\chi,\xi}$ as
\begin{equation}\label{quotrep} \FM_{\chi,\xi} = \mu^{-1}(\xi)^{\chi-ss} / \sim.\end{equation}

The varieties of interest for us will be of the form $\FM_{\chi,0}$. However for computing motivic classes the affine varieties $\FM_{0,\xi}$ are much more tractable. The following proposition connects the two under certain assumptions essentially by an argument of Nakajima \cite[Appendix]{CV04}.

\begin{prop}\label{nakfamily} Let $\xi \in (\g^*)^G$ be non-zero and $l_\xi \subset \g^*$ the linear subspace spanned by $\xi$. If $G$ acts freely on $\mu^{-1}(l_\xi)^{\chi-ss}$ 
%If both the total space and the zero fiber of the family $\mathfrak{F}:\mu^{-1}(l_\xi) \sslash_\chi G \rightarrow l_\xi$ are smooth,
 we have in the localized Grothendieck ring $\mathscr{M}$ the equality 
\begin{equation}\label{family1} [\FM_{\chi,0}] = [ \FM_{\chi,\xi}].\end{equation}
If furthermore $G$ acts freely on $\mu^{-1}(\xi)$ we have
\begin{equation}\label{family2} [\FM_{\chi,0}] = [\FM_{0,\xi}].\end{equation}
\end{prop}
\begin{proof} Since $G$ acts in particular freely on $\mu^{-1}(0)^{\chi-ss}$, we see as in Section \ref{symg} that $\mu^{-1}(0)^{\chi-ss}$ is smooth. Also all $G$-orbits have the same dimension and are thus closed. Hence by (\ref{quotrep}) $\FM_{\chi,0}$ is the quotient space of a smooth manifold by a free action and is thus smooth itself. A similar argument also shows that the total space of the family $\mathfrak{F}:\mu^{-1}(l_\xi) \sslash_\chi G \rightarrow l_\xi$ is smooth. 

Using the bilinearity of $\mu$ we obtain identifications $\mathfrak{F}^{-1}(0) \cong \FM_{\chi,0}$ and $\mathfrak{F}^{-1}(l_\xi \setminus 0) \cong \FM_{\chi,\xi} \times k^\times$, hence in $\kvar$
\[ [\mu^{-1}(l_\xi) \sslash_\chi G] = [\FM_{\chi,0}] + [\FM_{\chi,\xi}](\BL-1).\]

Now there is a natural contracting $k^\times$-action on $\mu^{-1}(l_\xi)$ which descends to the quotient $\mu^{-1}(l_\xi) \sslash_\chi G$ and covers the weight $2$ action on $l_\xi$ via $\mathfrak{F}$. It follows from considering the affinization map \ref{affmap}, that every $k^\times$-orbit under this action has a unique limit point in $(\mu^{-1}(l_\xi) \sslash_\chi G)^{k^\times}$. As we clearly have 
\[(\mu^{-1}(l_\xi) \sslash_\chi G)^{k^\times} = (\mathfrak{F}^{-1}(0))^{k^\times} = (\FM_{\chi,0})^{k^\times},\]
we deduce from the Bialynicki-Birula Theorem \cite[Theorem 4.1]{bi73} 
\[  [\mu^{-1}(l_\xi) \sslash_\chi G] = \BL[\FM_{\chi,0}],\]
where we also used  \cite[Corollary 2]{su74} to guarantee the existence of a $k^\times$-invariant quasi-affine open covering.

By comparing the two expressions for $[\mu^{-1}(l_\xi) \sslash_\chi G]$ we obtain (\ref{family1}) after inverting $\BL-1$.

Finally, \ref{family2} follows simply because the affinization $\FM_{\chi,\xi} \rightarrow \FM_{0,\xi}$ is an isomorphism if $G$ acts freely on $\mu^{-1}(\xi)$.
\end{proof}

In our cases of interest we can compute $[\FM_{0,\xi}]$ from the class of the fiber of the moment map $[\mu^{-1}(\xi)]$ (see Propositions \ref{toricfam} and \ref{qdef}), so we finish this section by explaining how to use the naive Fourier transform \ref{nft} to compute the latter. 

We define
\[a_{\varrho} = \left\{ (v,X) \in V \times \mathfrak{g} \ |\ \varrho(X)v=0 \right\},\]
which is a $\g$-variety via he projection onto the second factor $\pi:a_\varrho \rightarrow \g$. Analogous to \cite[Proposition 1]{Hau1} we have

\begin{prop}\label{prop1} Consider $V\times V^*$ as a $\g^*$-variety via $\mu$. Then we have in $\kexp_\g$ the equality
\begin{equation}\label{mmft1}\mathcal{F}([V\times V^*]) = \BL^{\dim V} [a_\varrho].\end{equation}
In particular, for any $\xi \in \g^*$ the identity
\begin{equation}\label{mmft2} [\mu^{-1}(\xi)] = \BL^{\dim V- \dim\g}[a_\varrho,\lla -\pi,\xi\rra]\end{equation}
holds in $\expm$. 
\end{prop}

\begin{proof} 
By (\ref{ftgen}) the naive Fourier transform of $[V\times V^*]$ is 
\begin{align*} \mathcal{F}([V\times V^*]) = [V \times V^* \times \g, \lla \mu\circ pr_{V\times V^*}, pr_\g\rra]. 
\end{align*}
Now by the definition (\ref{mmap}) of $\mu$ we have
\[ [V \times V^* \times \g, \lla \mu\circ pr_{V\times V^*}, pr_\g\rra] =  [V \times V^* \times \g, \lla(\varrho \circ pr_\g)pr_V,pr_{V^*}\rra].\]
Thus Lemma \ref{orth} with $X= V\times \g$ and $Z = a_\varrho$ gives (\ref{mmft1}).
Next we apply $\mathcal{F}$ again and use the inversion Lemma \ref{finv} to get
\[ \BL^{\dim \g} i^*[V\times V^*] = \mathcal{F}(\BL^{\dim V} [a_\varrho]) = \BL^{\dim V} [a_\varrho \times \g^*,\lla \pi\circ pr_{a_\varrho},pr_{\g^*}\rra]. \]
Finally, passing to $\expm_{\g^*}$ to invert $\BL^{\dim \g}$ and using $(i^*)^2= \Id_{\g^*}$ gives
\[[V\times V^*] = \BL^{\dim V-\dim \g}[a_\varrho \times \g^*,\lla -\pi\circ pr_{a_\varrho},pr_{\g^*}\rra].\]
Then (\ref{mmft2}) follows from pulling back both sides along $\xi:\spec(k) \rightarrow \g^*$.
\end{proof}
\begin{rem}\label{motdh} We like to think of Proposition \ref{prop1} as an instance of a motivic localization formula in the spirit of Theorem \ref{dhthm}. Namely if a motivic version of (\ref{dhf}) were to exist, it should ideally express the Fourier transform $\mathcal{F}([V\times V^*])$ as a sum over the fixed points $(V \times V^*)^G$. As $G$ acts linearly on $V \times V^*$ we further can assume $(V \times V^*)^G = \{0\}$. Then (\ref{mmft1}) indeed reflects such a localization phenomenon by the fact, that $\mathcal{F}([V\times V^*])$ is an honest class in $\kvar_\g$ and not $\kexp_\g$.
\end{rem}

\subsection{Hypertoric varieties}\label{htv}

An algebraic construction of hypertoric varieties was first given in \cite[Section 6]{HS02}, however much of our exposition is taken from \cite{PW07}. 

Consider an inclusion of tori $T^m \hookrightarrow T^n$ which is given in coordinates by an $n\times m$-matrix $A$ with integer entries and denote the rows of $A$ by $a_1,\dots,a_n \in \BZ^m$. Then the diagonal action of $T^n$ on $V=k^n$ gives rise to the representation 
\begin{align} \nonumber \rho:T^m &\rightarrow \GL(V)\\
								\label{toricrho}		t &\mapsto \diag(t^{a_1},\dots,t^{a_n}),
\end{align}
where we use the notation $t^b=t_1^{b_1}\dots t_m^{b_m}$ for any $t \in T^m$ and $b\in \BZ^m$. The derivative $\varrho : \ft^m \rightarrow \mathfrak{gl}(V)$ is for any $X\in \ft^m$ given by
\begin{equation}\label{toricvar} \varrho(X) = \diag(\lla X,a_1\rra,\dots,\lla X,a_n\rra).\end{equation}

From this we can now construct a variety as in \ref{genred}. More precisely $T^m$ will act on $V \times V^*$ in a Hamiltonian way with moment map $\mu: V \times V^* \rightarrow \Lie(T^m)= \ft^m $ given by the explicit formula
\begin{equation}\label{toricmm} \mu(v,w) = \sum_{i=1}^n v_iw_i a_i,
\end{equation}
where $(v,w) \in V \times V^*$ and $a_i \in (\ft^m)^*$ via the natural pairing. For a character $\chi \in \Hom(T^m,\BG_m)$ we define the \textit{hypertoric variety} $\FM_{\chi}(A)$ by 
\[ \FM_{\chi}(A) = \mu^{-1}(0) \sslash_\chi T^m.\]

There is a natural central arrangement $\FA \subset \ft^m$ given by the normal vectors $a_1,\dots,a_n \in (\ft^m)^*$, whose combinatorics are closely related to the geometry of $\FM_{\chi}(A)$. 

\begin{prop}\cite[Proposition 6.2]{HS02}\label{htsm} For a suitable character $\chi$ the hypertoric variety $\FM_{\chi}(A)$ is smooth if and only if $\FA$ is unimodular. In this case $T^m$ acts freely on $(V \times V^*)^{\chi-ss}$.
\end{prop}

\begin{proof} For the only if part we refer to \cite[Proposition 6.2]{HS02}. As in the proof of Proposition \ref{nakfamily}, the if part follows from the second statement, which we prove now.

To define $\chi$ consider the set $Z = \{ F \in L(\FA) \ |\ \rk(F)  =m-1\}$ that is the set of the maximal flats in $L(\FA) \setminus \infty$. For every $F \in Z$ the intersection of all its hyperplanes $H_F$ is a line in $\ft^m$ which is spanned by a (up to a sign) unique primitive $\lambda_F \in \BZ^m$. By changing the sign of $\lambda_F$ if necessary we can then choose $\chi \in \BZ^m$ such that $\lla \chi,\lambda_F\rra < 0$ for all $F \in Z$. 

Let now $(v,w) \in (V \times V^*)^{\chi-ss}$. For every $F \in Z$ the vector $\lambda_F$ defines a one-parameter subgroup of $T^m$ by $t \mapsto t^{\lambda_F}$. By Proposition \ref{kinmum} we therefore see, that $\lim_{t \rightarrow 0}\lambda_F(t)(v,w)$ doesn't exist in $V \times V^*$ for all $F \in Z$. Writing out the definitions we have 
\[ \lambda_F(t)(v,w) = (t^{\lla a_1,\lambda_F \rra}v_1,\dots,t^{\lla a_n,\lambda_F\rra}v_n,t^{-\lla a_1,\lambda_F\rra}w_1,\dots,t^{-\lla a_n,\lambda_F\rra}w_n).\]
From this we deduce that for every $F \in Z$ there must be an $i \in \{1,\dots,n\} \setminus F$ such that, depending on the sign of $\lla a_i,\lambda_F\rra$, either $v_i \neq 0$ or $w_i \neq 0$. In particular, we can construct from this inductively a set $\mathbf{I}=\{i_1,i_2,\dots,i_m\} \subset \{1,\dots,n\}$ such that $\{a_i\}_{i\in \mathbf{I}}$ are linearly independent, hence a $\BZ$-basis of $\BZ^m$ as $\FA$ is unimodular, and for every $i\in \mathbf{I}$ we have either $v_i \neq 0$ or $w_i \neq 0$.

Now let $s=(s_1,\dots,s_m) \in T^m$ be in the stabilizer of $(v,w)$ i.e. $s(v,w)=(v,w)$. Then we see, that whenever $v_i \neq 0$ or $w_i \neq 0$ we must have $s^{a_i} =1$. Hence in particular $s^{a_i} =1$ for all $i\in \mathbf{I}$. Since $\{a_i\}_{i\in \mathbf{I}}$ form a $\BZ$-basis of $\BZ^m$ we deduce that $s=1$ which proves that the stabilizer of $(v,w)$ is trivial.
\end{proof}

\begin{expl}\label{copn}  Consider the diagonal embedding $\BG_m = T^1 \hookrightarrow T^n$. Then $A$ is the $n \times 1$-matrix $A=(1,1,\dots,1)$ and the moment map is $\mu(v,w)=v_1w_1+v_2w_2+\dots +v_nw_n$. Then $\FM_\chi(A)$ will be smooth for any non-trivial character $\chi: \BG_m \rightarrow \BG_m$. If we take for example $\chi = \Id$ we get 
\[ \mu^{-1}(0)^{\chi-ss} = \{ (v,w) \in \mu^{-1}(0) \ |\ v \neq (0,0,\dots,0)\}.\]

Thus the projection onto the $v$-coordinate gives a rank $n-1$ vector bundle $\FM_{\Id}(A) \rightarrow \BP^{n-1}$ and one can check that in fact $\FM_{\Id}(A) \cong T^*\BP^{n-1}$.
\end{expl}

\subsubsection{Motivic classes of smooth hypertoric varieties}

Throughout this section we assume that $\FA$ is unimodular and $\chi$ is chosen as in Proposition \ref{htsm}, so that $\FM_\chi(A)$ is smooth. Under these assumptions we prove the following theorem.

\begin{thm}\label{toricmain} The motivic class of a smooth hypertoric variety $\FM_\chi(A)$ in $\mathscr{M}$ is given by the formula
\begin{equation}\label{htfor} [\FM_\chi(A)] = \frac{\BL^{n-m}}{(\BL-1)^m}\sum_{F\in L(\FA)} \nu_\FA(F,\infty) \BL^{|F|},\end{equation}
where $\nu_\FA:L(\FA)\times L(\FA) \rightarrow \BZ$ denotes the M\"obius function of the central arrangement $\FA$.
\end{thm}

Notice that, the geometry of $\FM_\chi(A)$ will in general depend on the choice of $\chi$ \cite[Section 9]{HS02}, however by Theorem \ref{toricmain} not its motivic class. This can already be seen from the following proposition.

\begin{prop}\label{toricfam} Assume that $\xi \in (t^m)^*$ does not vanish on $H_F$ for any $F\in L(\FA)\setminus \infty$. Then we have in $\mathscr{M}$
\begin{equation}\label{htfirst} [\FM_{\chi}(A)] = \frac{[\mu^{-1}(\xi)]}{(\BL-1)^m}. \end{equation}
\end{prop}
\begin{proof} %This is essentially an application of Proposition \ref{nakfamily} i.e. letting $l_\xi \subset (t^m)^*$ be the linear subspace spanned by $\xi$, we want to show that the total space of the family $\mathfrak{F}:\mu^{-1}(l_\xi) \sslash_\chi T^m \rightarrow l_\xi$ is smooth (the zero fiber $\mathfrak{F}^{-1}(0) \cong \FM_{\chi}(A)$ is smooth by assumption). This is in fact completely analogous to the proof of smoothness for $\FM_{\chi}(A)$, once we know that the generic fiber of $\mathfrak{F}$ is smooth. Thus we'll only explain this part.

First we claim that $T^m$ acts freely on $\mu^{-1}(\xi)$. Indeed, let $(v,w) \in \mu^{-1}(\xi)$ be fixed by some $t \in T^m$ and put $I = \{ i \ |\ v_iw_i \neq 0\} \subset \{1,\dots,n\}$. From the moment map equation (\ref{toricmm})
\[ \sum_i v_iw_ia_i = \xi,\]
together with the choice of $\xi$ we deduce that $\spn_{i\in I}\{a_i\} = (t^m)^*$, and by unimodularity  $\{a_i\}_{i\in I}$ contains a $\BZ$-basis of $\BZ^m$. But now $t\cdot(v,w) = (v,w)$ implies $t^{a_i} = 1$ for all $i\in I$, which finally implies $t = 1$ i.e $T^m$ acts freely on  $\mu^{-1}(\xi)$. Combining this with Propositions \ref{nakfamily} and \ref{htsm} we obtain 
\[ [\FM_{\chi}(A)] = [\mu^{-1}(\xi)\sslash_0 T^m] = [ \spec(k[\mu^{-1}(\xi)]^{T^m})].\]

Now as in \cite[Lemma 6.5]{Re03} one can show that $T^m$ acts scheme-theoretically freely on $\mu^{-1}(\xi)$, which implies by \cite[Proposition 0.9, Amplification 1.3]{MFK94}, that $\mu^{-1}(\xi) \rightarrow \spec(k[\mu^{-1}(\xi)]^{T^m})$ is a $T^m$-principal bundle.
As every $T^m$-principal bundle is Zariski-locally trivial \cite[Section 4.4]{serre58} the proposition follows from (\ref{trif}).
\end{proof}

We are thus left with computing $[\mu^{-1}(\xi)]$ which by (\ref{mmft2}) reduces to understanding $[a_\varrho,\lla -\pi,\xi\rra]\in \kexp$, where $\pi:a_\varrho \rightarrow \ft^m$ denotes the projection. This is done by a familiar cut and paste argument. First we have a stratification
 \begin{equation}\label{toricstra}\ft^m = \coprod_{F \in L(\FA)} \overset{\circ}{H}_F, \text{ where }  \overset{\circ}{H}_F = H_F \setminus \bigcup_{i \notin G} H_i.\end{equation}

By (\ref{toricvar}) we see that $\pi: \pi^{-1}(\overset{\circ}{H}_F) \rightarrow \overset{\circ}{H}_F$ is a vector bundle of rank $|F|$ and hence we have
\[ [a_\varrho,\lla -\pi,\xi\rra] = \sum_{F \in L(\FA)} [\pi^{-1}(\overset{\circ}{H}_F), \lla -\pi,\xi\rra] = \sum_{F\in L(\FA)} [\overset{\circ}{H}_F, \lla \cdot,\xi\rra]\BL^{|F|},\]
where $\lla \cdot,\xi\rra: \overset{\circ}{H}_F \rightarrow \BA^1$ denotes the paring with $\xi$. Theorem \ref{toricmain} now follows from the following

\begin{lemma}\label{incl} The class $[\overset{\circ}{H}_F, \lla \cdot,\xi\rra]$ lies in $\kvar$ and equals $\nu_\FA(F,\infty)[\spec(k)]$.
\end{lemma}
\begin{proof} We can check this directly using the recursive definition of $\nu_\FA$ (\ref{mobius}). For $F= \infty$ we clearly have 
\[ [\overset{\circ}{H}_\infty, \lla \cdot,\xi\rra] = [\spec(k),\lla \cdot,\xi\rra]= [\spec(k)],\]
and for any $F < \infty$ 
\[ \sum_{F \leq F'\leq \infty} [\overset{\circ}{H}_{F'}, \lla \cdot,\xi\rra] = [H_F,\lla \cdot,\xi \rra] = 0,\]
by Lemma \ref{orth}, since by assumption $\xi$ does not vanish on $H_F$.
\end{proof}

\begin{expl} We continue Example \ref{copn}. The corresponding arrangement $\FA$ is given by $n$ times the origin, in particular $L(\FA)= \{\emptyset,\infty\}$. Formula (\ref{htfor}) then reads
\[ [\FM_\chi(A)] = \frac{\BL^{n-1}}{\BL-1}\left(-1 + \BL^n\right) = \BL^{n-1} [\BP^{n-1}],\]
which of course agrees with the class of $T^*\BP^{n-1}$.
\end{expl}

\begin{rem}\label{motdh2} The proof of Theorem \ref{toricmain} relies crucially on Proposition \ref{prop1}, which in turn can be interpreted as a motivic instance of the Duistermaat-Heckman Theorem \ref{dhthm}, see Remark \ref{motdh}. We will briefly explain here, how the varieties $\FM_\chi(A)$ themselves give rise to such localization phenomena, but with more than one fixed point.

First we describe a second arrangement $\FB$, which is commonly associated with $\FM_\chi(A)$. Writing $d= n-m$ and identifying $\BZ^n / \BZ^m$ with $\BZ^{d}$ we obtain a linear map $B:\BZ^n \rightarrow \BZ^d$, where $B$ is a $d\times n$-matrix with integer coefficients. We write $b_1,\dots ,b_n \in \BZ^d$ for the columns of $B$ and $\widetilde{\chi} \in (\BZ^n)^*$ for any lift of $\chi \in \Hom(T^m,\BG_m) \cong (\BZ^m)^* \subset (\ft^m)^*$. Then $\FB \subset (\ft^d)^*$ is the arrangement consisting of the $n$ affine hyperplanes 
\begin{equation}\label{affinarr} H_i = \{ v \in (\ft^d)^* \ |\ \lla b_i,v \rra = \widetilde{\chi}_i\} \ \ \ \ 1\leq i\leq n.\end{equation}

The hypertoric variety $\FM_\chi(A)$ admits a residual torus action of $T^d = T^n/T^m$, which is again Hamiltonian. The moment map is given explicitly by (see \cite{PW07})
\begin{align*} \Phi: \FM_\chi(A) & \rightarrow (\ft^d)^* \\
												[v,w] & \mapsto (v_1w_1,\dots,v_nw_n) \in \ker(A^*) = (\ft^d)^*. \end{align*}

As opposed to the $T^m$-action on $V \times V^*$, the $T^d$-action on $\FM_\chi(A)$ will in general have non-trivial isolated fix points. For a generic point $x:\spec(k) \rightarrow \ft^d$ we thus would expect by our localization philosophy  in $\kexp$ a formula of the form
\[ x^*\FF([\FM_\chi(A)]) = \sum_{p \in \FM_\chi(A)^{T^d}} x^*[\lambda_p][\spec(k), \lla \Phi(p), x \rra],\]
for some $\lambda_p \rightarrow \ft^d$. 

This is indeed the case, as we can combine a inclusion-exclusion argument similar to Lemma \ref{incl} with the following two facts:

\begin{itemize}
\item[(i)] The moment map $\Phi$ induces a bijection between $T^d$-fixed points on $\FM_\chi(A)$  and maximal flats (i.e points) in $L(\FB)$.
\item[(ii)] For any flat $F \in L(\FB)$ and $\overset{\circ}{H}_F \subset (\ft^d)^*$ defined as in (\ref{toricstra}) we have
\[ \Phi^{-1}(\overset{\circ}{H}_F) \cong  \lambda'_F \times \overset{\circ}{H}_F.\]
\end{itemize}

Here (i) can can be found in \cite[Corollary 3.5]{BD00} and (ii) follows from covering $\FM_\chi(A)$ with affine hypertoric varieties as in \cite[Proposition 4.6]{PA16}.

\end{rem}

\subsection{Nakajima quiver varieties}\label{nqv}

We start by recalling the definition of Nakajima quiver varieties. Almost everything can be found in more detail in \cite{Hau2} or in the original sources \cite{nak94}\cite{nak98}.\\
Let $\Gamma = (I,E)$ be a quiver with $I=\{1,2,\dots,n\}$ the set of vertices and $E$ the set of arrows. We denote by $s(e)$ and $t(e)$ the source and target vertex of an arrow $e\in E$. For each $i\in I$ we fix finite dimensional $k$-vector spaces $V_i, W_i$ and write $\bfv=(\dim V_i)_{i\in I}, \bfw=(\dim W_i)_{i\in I} \in \mathbb{N}^I$ for their dimension vectors.

 From this data we construct the vector space
\[ \BV_{\bfv,\bfw}= \bigoplus_{e\in E} \Hom(V_{s(e)},V_{t(e)}) \oplus \bigoplus_{i\in I} \Hom(W_i,V_i),\]
the algebraic group 
\[ G_\bfv = \prod_{i\in I} \GL(V_i)\]
and its Lie algebra
 \[ \mathfrak{g}_\bfv = \bigoplus_{i\in I} \mathfrak{gl}(V_i).\]
We have a natural representation
\[ \rho_{\bfv,\bfw}: G_\bfv \rightarrow \GL(\BV_{\bfv,\bfw})\]
and its derivative 
\[ \varrho_{\bfv,\bfw}: \mathfrak{g}_\bfv \rightarrow \mathfrak{gl}(\BV_{\bfv,\bfw}).\]

For $g=(g_i)_{i\in I}\ ,X=(X_i)_{i\in I} $ and $\phi=(\phi_e,\phi_i)_{e\in E,i\in I}\in \BV_{\bfv,\bfw}$ they are given by the formulas
\begin{align*} 
\rho_{\bfv,\bfw}(g)\phi &= (g_{t(e)}\phi_e g_{s(e)}^{-1}, g_i\phi_i)_{e\in E,i\in I}\\
\varrho_{\bfv,\bfw}(X)\phi &= (X_{t(e)}\phi_e - \phi_e X_{s(e)},X_i \phi_i)_{e\in E,i\in I}.
\end{align*}

We are again in the situation of section \ref{genred} i.e. $G_\bfv$ acts on the vector space $\BV_{\bfv,\bfw}\times	\BV_{\bfv,\bfw}^*$ in a Hamiltonian way with moment map  
\[ \mu_{\bfv,\bfw}: \BV_{\bfv,\bfw} \oplus \BV_{\bfv,\bfw}^* \rightarrow \mathfrak{g}_\bfv^*,\] 
given by (\ref{mmap}).

 Following \cite{nak98} we fix once and for all $\chi$ to be the character of $G_\bfv$ given by $\chi(g) = \prod_{i\in I} \det(g_i)^{-1}$. For $\bfw\neq \bm{0}$ the \textit{Nakajima quiver variety} $\FM(\bfv,\bfw)$ is then defined as
\[ \FM(\bfv,\bfw) = \mu_{\bfv,\bfw}^{-1}(0) \sslash_\chi G_\bfv.\]  

\begin{rem}  For $\bfw =\bm{0}$ the action of $G_\bfv$ will not be faithful, in which case one should choose a different character (see \cite[Remark 3.13]{nak98}), but we do not study this case here.
\end{rem}

\begin{expl}\cite[Proposition 2.8]{Na99}\label{jordan1} Consider the Jordan quiver $\Gamma$ consisting of a single vertex  and a single loop on that vertex with the dimension vectors given by positive integers $\bfv = n$ and $\bfw=1$. In this case we can identify $\mu^{-1}_{n,1}(0)$ with the set of elements $(X,r,Y,s) \in \left(\End(k^n) \oplus k^n\right)^2$ satisfying
\[\mu_{n,1}(X,r,Y,s)= [X,Y] + r s^t =0,\]
where $[X,Y] = XY-YX$ denotes the commutator and $s^t$ the transpose of $s$. 

Now $\chi$-stability will imply $r=0$ and that 
\[\{ f \in k[x,y] \ |\  s^tf(X,Y)v = 0 \text{ for all } v\in k^n\} \subset k[x,y]\]
defines an ideal in $k[x,y]$ of length $n$. This identifies $\FM_{n,1}$ with $\Hilb^n\BA^2$, the Hilbert scheme of $n$ points on $\BA^2$.
\end{expl}

The motivic class of $\FM(\bfv,\bfw)$ can again be determined through Proposition \ref{nakfamily}. Take $\unt_\bfv\in \mathfrak{g}_\bfv^*$ to be the linear functional defined by $\unt_\bfv(X)= \sum_{i\in I} \tr X_i$ for $X\in \mathfrak{g}_\bfv$. 

\begin{prop}\label{qdef} The equality
\[ [\FM(\bfv,\bfw)] = \frac{[\mu_{\bfv,\bfw}^{-1}(\unt_\bfv)]}{[G_\bfv]},\]
holds in $\mathscr{M}$.
\end{prop}
\begin{proof} First by \cite[Lemma 3.10]{nak98} $G_\bfv$ acts freely on $\left(  \BV_{\bfv,\bfw} \oplus \BV_{\bfv,\bfw}^*\right)^{\chi-ss}$ and also on $\mu_{\bfv,\bfw}^{-1}(\unt_\bfv)$ by \cite[Corollary 5]{Hau2}. Hence by Proposition \ref{nakfamily} we get
\[  [\FM(\bfv,\bfw)] = [\mu^{-1}(\unt_\bfv)\sslash_0 G_\bfv].\]

Now again as in Proposition \ref{toricfam}, the quotient $\mu^{-1}(\unt_\bfv) \rightarrow \mu^{-1}(\unt_\bfv)\sslash_0 G_\bfv$ is a $G_\bfv$-principal bundle, hence Zariski-locally trivial \cite[Lemma 5 and 6]{serre58}, this implies the proposition by formula (\ref{trif}).
\end{proof}

\subsubsection{Motivic classes of Nakajima quiver varieties}\label{tmc}

In this section we deduce a formula for the motivic class of $\FM(\bfv,\bfw)$ in terms of the combinatorial data of the quiver $\Gamma$. It will be convenient to consider the  generating series
\begin{eqnarray}\label{defin}\Phi(\bfw) = \sum_{\bfv \in \BN^I} [\FM(\bfv,\bfw)] \BL^{d_{\bfv,\bfw}}T^\bfv \in \mathscr{M}[[T_1,\dots,T_n]], \end{eqnarray}
where we define 
\[d_{\bfv,\bfw} = \dim(\mathfrak{g}_\bfv)-\dim(\BV_{\bfv,\bfw}).\]

Combining Propositions \ref{qdef} and \ref{prop1} we have
\begin{eqnarray}\label{phisim} \Phi(\bfw) = \sum_{\bfv \in \BN^I} \frac{[\FV_1(\bfv,\bfw)]}{[\GL_\bfv]} \BL^{d_{\bfv,\bfw}}T^\bfv =\sum_{\bfv \in \BN^I} \frac{[a_{\varrho_{\bfv,\bfw}},\lla-\pi, \unt_\bfv\rra]}{[\GL_\bfv]}T^\bfv, 
\end{eqnarray}
with the notations 
\[a_{\varrho_{\bfv,\bfw}} = \{ (\phi,X) \in \BV_{\bfv,\bfw} \times \mathfrak{g}_\bfv \ | \ \varrho_{\bfv,\bfw}(X)v=0\}\]
 and $\pi:a_{\varrho_{\bfv,\bfw}}\rightarrow \mathfrak{g}_\bfv$ the natural projection. 

Next we use some basic linear algebra to split up the above generating series into a regular and a nilpotent part. Given a finite dimensional vector space $V$ of dimension $n$ and an endomorphism $X$ of $V$, we can write $V=N(X) \oplus R(X)$, where $N(X) = \ker(X^n)$ and $R(X)=\im(X^n)$. With respect to this decomposition we have $X=X^\n \oplus X^\re$ with $X^\n= X_{|N(X)}$ nilpotent and $X^\re=X_{|R(X)}$ regular. \\
Now let $\bfv'= (v_i')_{i\in I}$ with $\bfv' \leq \bfv $ (i.e the inequality holds for every entry). We define the three varieties
\begin{align*}
a_{\varrho_{\bfv,\bfw}}^{\bfv'} &= \{ (\phi,X) \in a_{\varrho_{\bfv,\bfw}} \ | \ \dim(N(X_i))=v_i' \text{ for }i\in I\} \\
 a_{\varrho_{\bfv,\bfw}}^{\n} &= \{ (\phi,X) \in a_{\varrho_{\bfv,\bfw}} \ | \ X \text{ nilpotent}\}, \\
 a_{\varrho_{\bfv,\bfw}}^\re &= \{ (\phi,X) \in a_{\varrho_{\bfv,\bfw}}\ | \ X \text{ regular}\}.
\end{align*}

\begin{lemma}\label{first} For every $\bfv' \leq \bfv $ we have the following relation in $\expm$
\begin{eqnarray}\label{msp1}  \frac{[a_{\varrho_{\bfv,\bfw}}^{\bfv'},\lla-\pi, \unt_\bfv\rra]}{[G_\bfv]} = \frac{[a_{\varrho_{\bfv',\bfw}}^{\n}]}{[G_{\bfv'}]} \frac{[a_{\varrho_{\bfv-\bfv',0}}^\re, \lla-\pi, \unt_{\bfv-\bfv'}\rra]}{[G_{\bfv-\bfv'}]}.\end{eqnarray}
\end{lemma}

\begin{proof}  Fix for all $i\in I$ a decomposition $V_i=V_i' \oplus V_i''$ with $\dim(V_i')= v'_i$. This induces inclusions
\[\BV_{\bfv',\bfw}\oplus \BV_{\bfv-\bfv',0} \hookrightarrow \BV_{\bfv,\bfw} \text{  and  } \mathfrak{g}_{\bfv'} \oplus \mathfrak{g}_{\bfv-\bfv'} \hookrightarrow \mathfrak{g}_{\bfv}.\]

We will prove that the morphism
\begin{align*} \Delta: a_{\varrho_{\bfv',\bfw}}^{\n} \times a_{\varrho_{\bfv-\bfv',0}}^\re \times G_\bfv &\rightarrow a_{\varrho_{\bfv,\bfw}}^{\bfv'}\\ 
							(\phi',X',\phi'',X'',g) &\mapsto (\rho_{\bfv,\bfw}(g)(\phi'\oplus \phi''), \Ad_g(X' \oplus X''))
\end{align*}
is a Zariski-locally trivial $G_{\bfv'} \times G_{\bfv-\bfv'}$-fibration. Since for every $(\phi',X') \in a_{\varrho_{\bfv',\bfw}}^{\n}$ we have
\[ \lla-\pi, \unt_{\bfv'}\rra(\phi',X')= \sum_{i\in I} \tr X'_i=0,\]
this will imply the lemma using (\ref{trif}).\\
First notice that $\Delta$ is well defined because
\[\varrho_{\bfv,\bfw} \circ \Ad_g = \Ad_{\rho_{\bfv,\bfw}(g)} \circ \varrho_{\bfv,\bfw}.\]
The $G_{\bfv'} \times G_{\bfv-\bfv'}$-action on the domain of $\Delta$ is given as follows. 
%Let $(\phi,X) \in a_{\varrho_{\bfv,\bfw}}^{\bfv'}$. Then for each $i\in I$ we can write $X_i=X_i^\n \oplus X_i^\re$ with the notation from above and up to conjugation in $G_\bfv$ we can assume that this decomposition is compatible with the one we fixed for $V_i$, i.e. $N(V_i)=V_i'$ and $R(V_i)=V_i''$. Now $\varrho_{\bfv,\bfw}(X)\phi=0$ gives for all $i\in I$ and $e\in E$ the equations
%\[ X_{t(e)}\phi_e=\phi_eX_{s(e)}, \ X_i\phi_i=0.\]
%From this we deduce $(({\phi_e}_{|V_{s(e)}'},\phi_i)_{e\in E,i\in I},(X_i^\n)_{i\in I})\in a_{\varrho_{\bfv',\bfw}}^{\n}$ and  $(({\phi_e}_{|V_{s(e)}''})_{e\in E},(X_i^\re)_{i\in I}) \in a_{\varrho_{\bfv-\bfv',0}}^\re$, which proves surjectivity of $\Gamma$.\\
For \linebreak $h=(h',h'')\in G_{\bfv'}\times G_{\bfv-\bfv'}$ and $(\phi',X',\phi'',X'',g)\in a_{\varrho_{\bfv',\bfw}}^{\n} \times a_{\varrho_{\bfv-\bfv',0}}^\re \times G_\bfv$  we set
\[ h\cdot (\phi',X',\phi'',X'',g) = (\rho_{\bfv',\bfw}(h')\phi',\Ad_{h'}X',\rho_{\bfv-\bfv',0}(h'')\phi'',\Ad_{h''}X'',gh^{-1}),\]
where $gh^{-1}$ is understood via the inclusion  $G_{\bfv'} \times G_{\bfv-\bfv'} \hookrightarrow G_{\bfv}$. One checks directly that $\Delta$ is invariant under this action and hence each fiber of $\Delta$ carries a free $G_{\bfv'}\times G_{\bfv-\bfv'}$-action. 

On the other hand, assume $\Delta(\phi_1',X_1',\phi_1'',X_1'',g_1) = \Delta(\phi_2',X_2',\phi_2'',X_2'',g_2)$. This implies
\[ \Ad_{g_2^{-1}g_1}(X_1' \oplus X_1'')=X_2'\oplus X_2''.\]
Since $X_j'$ is nilpotent and $X_j''$ regular for $j=1,2$, the decomposition $V_i=V_i' \oplus V_i''$ is preserved by $g_2^{-1}g_1$ i.e. $g_2^{-1}g_1 \in G_{\bfv'}\times G_{\bfv-\bfv'}$, which shows that each fiber of $\Delta$ is isomorphic to $G_{\bfv'}\times G_{\bfv-\bfv'}$.

Finally, to trivialize $\Delta$ locally we notice, that there is an open covering $a_{\varrho_{\bfv,\bfw}}^{\bfv'}= \cup_j U_j$ and algebraic morphisms $t_j:U_j \rightarrow G_\bfv$ such that for $X\in U_j$ and $i\in I$ the columns of the matrix $t_j(X)_i$ form a basis of $N(X_i)$ and $R(X_i)$.
\end{proof}
%\begin{lemma}\label{linalg} (a) Assume $V=N \oplus R$ and two endomorphisms $X,Y$ of $V$ with $N(X)=N(Y)=N$ and $R(X)=R(Y)=R$. If $g\in \GL(V)$ satisfies $gXg^{-1}=Y$ we have
%\[g(N) \subset N, \ g(R) \subset R.\]
%\end{lemma}
Now we use the stratification $a_{\varrho_{\bfv,\bfw}}= \coprod_{\bfv' \leq \bfv} a_{\varrho_{\bfv,\bfw}}^{\bfv'}$ together with Lemma \ref{first} to get

\begin{align}\label{aster} \nonumber \Phi(\bfw) &\stackrel{(\ref{phisim})}{=} \sum_{\bfv \in \BN^I} \frac{[a_{\varrho_{\bfv,\bfw}},\lla-\pi, \unt_\bfv\rra]}{[\GL_\bfv]}T^\bfv\\
\nonumber &= \sum_{\bfv \in \BN^I}\sum_{\bfv' \leq \bfv} \frac{[a_{\varrho_{\bfv,\bfw}}^{\bfv'},\lla-\pi, \unt_\bfv\rra]}{[\GL_\bfv]}T^\bfv\\ 
\nonumber &\stackrel{(\ref{msp1})}{=}  \sum_{\bfv \in \BN^I}\sum_{\bfv' \leq \bfv}\frac{[a_{\varrho_{\bfv',\bfw}}^{\n}]}{[G_{\bfv'}]} \frac{[a_{\varrho_{\bfv-\bfv',0}}^\re, \lla-\pi,\unt_{\bfv-\bfv'}\rra]}{[G_{\bfv-\bfv'}]}T^\bfv \\
&= \Phi_{\n}(\bfw)\Phi_\re,
\end{align}
where we used the notations 
\[ \Phi_{\n}(\bfw) = \sum_{\bfv \in \BN^I} \frac{[a_{\varrho_{\bfv,\bfw}}^{\n}]}{[G_{\bfv}]}T^\bfv\]
and
\[ \Phi_\re = \sum_{\bfv \in \BN^I} \frac{[a_{\varrho_{\bfv,0}}^\re, \lla-\pi, \unt_{\bfv}\rra]}{[G_{\bfv}]}T^\bfv.\]

Notice that equation (\ref{phisim}) also makes sense for $\bfw = \bm{0}$, in which case \cite[Lemma 3]{Hau2} implies $\Phi(\mathbf{0}) = 1$. Therefore
\begin{eqnarray}\label{fract} \Phi(\bfw) = \frac{\Phi_{\n}(\bfw)}{\Phi_{\n}(\bm{0})},\end{eqnarray}	
which leaves us with computing $\Phi_{\n}(\bfw)$.

We denote by $\FP$ the set of all partitions $\lambda = (\lambda^1,\lambda^2,\dots)$, where $\lambda^1\geq \lambda^2\geq\dots$. The size of $\lambda$ is $|\lambda|=\lambda^1+\lambda^2+\dots$ and $\FP_n$ denotes the set of partitions of size $n$. For $\lambda \in \FP_n$ we write $\FC(\lambda)$ for the nilpotent conjugacy class, whose Jordan normal form is given by $\lambda$. For $\bfl = (\lambda_i)\in \FP^I$ with $\lambda_i\in \FP_{v_i}$ we set 

\[a_{\varrho_{\bfv,\bfw}}^{\n}(\bfl) = \{ (\phi,X) \in a_{\varrho_{\bfv,\bfw}}^{\n} \ |\ X_i \in \FC(\lambda_i) \},\]
which gives the stratification

\begin{eqnarray}\label{strat2}a_{\varrho_{\bfv,\bfw}}^{\n} = \coprod_{\substack{\bfl \in \FP^I \\ \lambda_i\in \FP_{v_i}}} a_{\varrho_{\bfv,\bfw}}^{\n}(\bfl).\end{eqnarray}

To compute $[a_{\varrho_{\bfv,\bfw}}^{\n}(\bfl)]$ we look at the projection 

\begin{eqnarray}\label{vb} \pi: a_{\varrho_{\bfv,\bfw}}^{\n}(\bfl) \rightarrow\FC(\bfl)= \prod_{i\in I} \FC(\lambda_i).\end{eqnarray}
The fiber of $\pi$ over $X\in \FC(\bfl)$ is simply $\ker(\varrho_{\bfv,\bfw}(X))$. Because of $\varrho_{\bfv,\bfw} \circ \Ad_g = \Ad_{\rho_{\bfv,\bfw}(g)} \circ \varrho_{\bfv,\bfw}$ the dimensions of those kernels are constant and hence $\pi$ is a vector bundle of rank, say, $\kappa_{\bfv,\bfw}(\bfl)$. 
\begin{lemma}\label{second} Denote by $Z(\bfl) \subset G_\bfv$ the centralizer of (some element in) $\FC(\bfl)$. We have the following relation in $\mathscr{M}$.
\begin{eqnarray}\label{idkn} \frac{[a_{\varrho_{\bfv,\bfw}}^{\n}(\bfl)]}{[G_\bfv]} = \frac{\BL^{\kappa_{\bfv,\bfw}(\bfl)}}{[Z(\bfl)]}\end{eqnarray}
\end{lemma}
\begin{proof}
The formula (\ref{zent}) below shows in particular that $[Z(\bfl)]$ is invertible in $\mathscr{M}$. Since the projection (\ref{vb}) is a vector bundle, we are left with proving $[G_\bfv]/[Z(\bfl)]= [\FC(\bfl)]$. Since $\FC(\bfl)$ is isomorphic to $G_\bfv / Z(\bfl)$, see for example \cite[Chapter 3.9.1]{bo12}, it is enough to prove that the $Z(\bfl)$-principal bundle $G_\bfv \rightarrow  G_\bfv / Z(\bfl)$ is Zariski locally trivial by (\ref{trif}). In fact, this is true for every $Z(\bfl)$-principal bundle, %as $Z(\bfl)$ is connected and solvable \todo{sure?}, by \cite[Proposition 14]{serre58}.
 which follows from combining Propositions 3.13 and 3.16 of \cite{Mer13}.
\end{proof}

%\begin{lemma}\label{second} The morphism
%\begin{align*}\Xi: \prod_{e\in E} \{\phi\in \Hom(V_{s(e)},V_{t(e)})\ |\ \phi N(\lambda^{s(e)}) = N(\lambda^{t(e)})\phi \} \times \prod_{i\in I} \{\phi\in \Hom(W_i,V_i)\ |\ N(\lambda^i)\phi=0\} \times G_\bfv &\rightarrow a_{\varrho_{\bfv,\bfw}}^{\n}(\bfl)\\
%(\phi_e,\phi_i,g) &\mapsto (\Ad_g\phi_e,g\phi,\Ad_gN(\lambda))
%\end{align*}
%is a Zariski locally trivial $\mathcal{C}_\lambda$-principal bundle.
%\end{lemma}
%\begin{proof} The prove that $\Xi$ is a $\mathcal{C}_\lambda$-principal bundle is straight forward and similar to \ref{first}. Zariski locall triviality follows since $\mathcal{C}_\lambda$ is special by lemma \ref{special}.
%\end{proof}
 
To compute $\kappa_{\bfv,\bfw}(\bfl)$ and $[Z(\bfl)]$, denote by $m_k(\lambda)$ the multiplicity of $k\in \BN$ in a partition $\lambda \in \FP$. Then given any two partitions $\lambda,\lambda' \in \FP$ we define their inner product to be
\[ \lla \lambda,\lambda' \rra = \sum_{i,j\in \BN} \min(i,j)m_i(\lambda)m_j(\lambda').\]
Lemma 3.3 in \cite{Hua} implies now
\begin{eqnarray}\label{kappa} \kappa_{\bfv,\bfw}(\bfl) = \sum_{e\in E} \lla  \lambda_{s(e)},\lambda_{t(e)}\rra + \sum_{i\in I} \lla 1^{w_i},\lambda_i \rra, 
\end{eqnarray}
where $1^{w_i} \in \FP_{w_i}$ denotes the partition $(1,1,\dots,1)$.

For $[Z(\bfl)]$ we can use the formula (1.6) from \cite[Chapter 2.1]{Mac}. There the formula is worked out over a finite field but Lemma 1.7 of \textit{loc. cit.} holds over any field. In our notation this gives (see \cite[Chapter 3]{Hua} for details)
\begin{eqnarray}\label{zent} [Z(\bfl)] =\prod_{i\in I} \BL^{\lla \lambda_i,\lambda_i \rra} \prod_{k\in \BN} \prod_{j=1}^{m_k(\lambda_i)} (1-\BL^{-j}).\end{eqnarray}

Finally, combining (\ref{fract}), (\ref{strat2}), (\ref{idkn}), (\ref{kappa}) and (\ref{zent}) we obtain 

\begin{thm}\label{mainth} For a fixed non-zero dimension vector $\bfw \in \BN^I$  the motivic classes of the Nakajima quiver varieties $\FM(\bfv,\bfw)$ in $\mathscr{M}$ are given by the generating function
	\begin{eqnarray}\label{for1} \sum_{\bfv \in \BN^I} [\FM(\bfv,\bfw)] \BL^{d_{\bfv,\bfw}}T^\bfv = \frac{\sum_{\bfl \in \FP^I} \frac{\prod_{e\in E} \BL^{\lla \lambda_{s(e)},\lambda_{t(e)}\rra} \prod_{i\in I} \BL^{\lla 1^{w_i},\lambda_i \rra} }{\prod_{i\in I} \BL^{\lla \lambda_i,\lambda_i \rra} \prod_k \prod_{j=1}^{m_k(\lambda_i)} (1-\BL^{-j})}T^{|\bfl|}}{\sum_{\bfl \in \FP^I} \frac{\prod_{e\in E} \BL^{\lla \lambda_{s(e)},\lambda_{t(e)}\rra}}{\prod_{i\in I} \BL^{\lla \lambda_i,\lambda_i \rra} \prod_k \prod_{j=1}^{m_k(\lambda_i)} (1-\BL^{-j})}T^{|\bfl|}},\end{eqnarray}
\end{thm}

\begin{expl} Consider again the Jordan quiver as in \ref{jordan1}, where we saw $\FM(n,1) \cong \Hilb^n\BA^2$. Then formula (\ref{for1}) gives a generating series for the classes $[\Hilb^n\BA^2]$:
\[ \sum_{n\geq 0} [\Hilb^n\BA^2] \BL^{-n}T^n = \frac{\sum_{\lambda \in \FP} \prod_k \prod_{j=1}^{m_k(\lambda)} (1-\BL^{-j})^{-1}\BL^{\sum_{k}m_k(\lambda)}T^{|\lambda|}}{\sum_{\lambda \in \FP}  \prod_k \prod_{j=1}^{m_k(\lambda)} (1-\BL^{-j})^{-1}T^{|\lambda|}}. \]

On the other hand there is a well known formula for $[\Hilb^n\BA^2]$ due to G\"ottsche \cite{Go01}, which gives
\[ \sum_{n\geq 0} [\Hilb^n\BA^2] \BL^{-n}T^n = \prod_{k\geq 1} \frac{1}{1-\BL T^k}.\]

Thus we see that (\ref{for1}) gives already in this case a quite non-trivial combinatorial statement. Similar identities appear for example in \cite{Hua}.
\end{expl}

\newpage

\section{Open de Rham spaces}\label{ch3}

In this chapter we study meromorphic connections on the trivial rank $n$ bundle on $\BP^1$. By fixing some local data $\mathbf{C}$ at the poles of the connection one obtains a finite dimensional moduli space $\FM_n(\mathbf{C})$, called the \textit{open de Rham space}. Originally they were introduced by Boalch \cite{Bo01} to study isomonodromic deformations, however for us they arise in a slightly different context. 

Our starting point is the conjecture of Hausel, Mereb and Wong \cite{HMW16} on the mixed Hodge polynomial of wild character varieties $\FM_{Betti}$. These character varieties are the target space for the wild Riemann-Hilbert correspondence, which associates to a meromorphic connection its Stokes data \cite{Bo11}. Even though this correspondence is not algebraic, the purity conjecture \cite{HR08} predicts, that $H^*(\FM_n(\mathbf{C}))$ equals the pure part of $H^*(\FM_{Betti})$. 

In our forthcoming paper \cite{HWW17} we find numerical evidence for both conjectures by computing the $E$-polynomial of $\FM_n(\mathbf{C})$ and proving an agreement with the conjectural pure part of the mixed Hodge polynomial of the corresponding $\FM_{Betti}$. In this thesis we only explain how to compute $E(\FM_n(\mathbf{C});x,y)$, or more precisely $[\FM_n(\mathbf{C})] \in \mathscr{M}$, as this fits into the same motivic Fourier transform setting we already used in Chapter \ref{ch2}.

In the first two sections we introduce the relevant notation and prove some computational lemmas on coadjoint orbits. It might therefore be more interesting to first read Section \ref{mercon}, where we define $\FM_n(\mathbf{C})$ and give a description in terms of coadjoint orbits following $\cite{Bo01}$. More precisely a pole of order $k$ can be modeled locally by a coadjoint orbit $\FO \subset \mathfrak{gl}_n(\BC[[z]]/z^k)^*$, and $\FM_n(\mathbf{C})$ is obtained by a symplectic fusion of the individual poles, see Proposition \ref{odrspa}.

Using the motivic convolution construction \ref{motcon}, the computation of $[\FM_n(\mathbf{C})]$ reduces to understand the Fourier transform of the composition 
\[\FO \hookrightarrow \mathfrak{gl}_n(\BC[[z]]/z^k)^* \rightarrow \mathfrak{gl}_n(\BC)^*.\]

This key computation is carried out in Section \ref{ftp} under the assumption $k\geq 2$, which will ensure that the 'motivic function' $\FF([\FO])\in \kexp_{\mathfrak{gl}_n(\BC)}$ is supported on semi-simple conjugacy classes, which fails for $k=1$. Eventually we would like to explain this phenomenon in a similar way as in Remarks \ref{motdh}, \ref{motdh2} i.e. as an instance of a general motivic localization formula, but we are not able to do so at the moment.

In Section \ref{mcodr} we put everything together and give in Theorem \ref{mainodr} an explicit formula for $[\FM_n(\mathbf{C})]$ as a polynomial in $\BL$, assuming that the order of each pole is $\geq 2$. In the last section we combine our results with the work of \cite{HLV11} on order one poles and sketch how to extend our computations, at least over finite fields, to give a formula for $|\FM_n(\mathbf{C})(\BF_q)|$ under the milder assumption, that at least one pole has to be of order $\geq 2$. 

\subsection{Jets and duals}
Let $n\geq 1$ be an integer. We abbreviate $G = GL_n(\BC)$ and $\g=\mathfrak{gl}_n(\BC)$. Furthermore $T \subset G$ will denote the standard maximal torus consisting of diagonal matrices, $\ft \subset \g$ its Lie algebra and $\ft^{reg} \subset \ft$ the subset of elements with distinct eigenvalues. We also have the jet versions
\begin{align*} G_k &= GL_n(\BC[[z]]/z^k) = \left\{ g_0+zg_1+\dots +z^{k-1}g_{k-1} \ |\ g_0 \in G, \ g_1,\dots,g_{k-1} \in \g \right\},\\
\g_k &= \mathfrak{gl}_n(\BC[[z]]/z^k) = \left\{X_0+zX_1+\dots + z^{k-1}X_{k-1}\ |\ X_i \in \g \right\}, 
\end{align*}
and similarly we define $T_k$ and $\ft_k$.

Finally, we have the unipotent subgroup $B_k \subset G_k$ and its Lie algebra $\fb_k$ defined by
\begin{align*} B_k &= \{ \unt + zb_1 +z^2b_2 +\dots + z^{k-1}b_{k-1}\ |\ b_i \in \g \},\\
 \fb_k &= \{ zX_1 +z^2X_2 +\dots + z^{k-1}X_{k-1} \ |\ X_i \in \g\}.\end{align*}
Note that we have $G_k = B_k \rtimes G$, where $G$ acts on $B_k$ by conjugation, and thus a decomposition $\g_k = \fb_k \oplus \g$.

It will be convenient to identify the dual $\g_k^*$ with 
\[ z^{-k}\g_k = \left\{ z^{-k}Y_k + z^{-(k-1)}Y_{k-1}+\dots + z^{-1}Y_1 \ |\ Y_i \in \g \right\},\]
via the trace residue pairing i.e. for  $X \in \g_k$ and $Y \in z^{-k}\g_k$ we set
\begin{equation} \label{trp} \lla Y,X \rra = 	\Res_0 \tr YX = \sum_{i=1}^{k} \tr Y_{i}X_{i-1}.\end{equation}
Under this identification $\g^*$ corresponds to $z^{-1}\g \subset z^{-k}\g_k$ and $\fb_k^*$ to those elements in $z^{-k}\g_k$ having zero residue term. We write
\begin{equation}\label{pis} \pi_{\res}: \g_k^* \rightarrow \g^*, \ \ \pi_{\irr}: \g_k^* \rightarrow \fb_k^* \end{equation}
for the natural projections. The adjoint and coadjoint actions of $G_k$ on $\g_k$ and $\g_k^*$ will both be denoted by $\Ad$ and are defined by the same formula $\Ad_g X = gXg^{-1}$. Notice that with this convention we have $\lla \Ad_gY ,X\rra = \lla Y,\Ad_{g^{-1}}X\rra$.

\subsection{Coadjoint orbit computations}

We write $\g^{od}$ for matrices with zeros on the diagonal and for $X,Y \in \g$ we write $[X,Y] = XY-YX$ for the commutator.
\begin{lemma} \label{intor} \begin{enumerate} 
\item For $X \in \g$ and $Y \in \ft$ we have $[X,Y] \in \g^{od}$.

\item \label{secnd} Let $Y = z^{-k}Y_k + z^{-(k-1)}Y_{k-1}+\dots + z^{-1}Y_1 \in \ft^*_k$ with $Y_k \in \ft^{reg}$ and $g \in G_k$ such that $\Ad_gY - Y \in \ft^*_{k-1}$. Then $g \in T_k$ and $\Ad_gY = Y$. 

\item \label{intorr} Let $g \in T$ and $h \in B_k$ such that $hgh^{-1} \in T_k$. Then $hgh^{-1} = g$.

\end{enumerate}
\end{lemma}
\begin{proof} The first statement is clear. For part \ref{secnd} write $\Ad_gY - Y = W = z^{-(k-1)}W_{k-1}+\dots + z^{-1}W_1 \in \ft^*_{k-1}$. Then by rewriting we obtain 
\begin{equation}\label{rwrwr} gY = (Y+W)g.\end{equation}
The $z^{-k}$ term of (\ref{rwrwr}) reads $g_0Y_k = Y_kg_0$, hence $g_0 \in T$. We now proceed by induction, assuming $W_{k-1} = \dots = W_r=0$ and $ g_1,g_2\dots g_{k-r} \in \ft$ for some $r\leq k$. The $z^{-(r-1)}$ terms of (\ref{rwrwr}) equal
\[ \sum_{j=0}^{k-r+1} g_jY_{r-1+j} = \sum_{j=0}^{k-r+1} (Y_{r-1+j}g_j + W_{r-1+j}g_j).\]
By the induction hypothesis this simplifies to $g_{k-r+1}Y_k  =W_{r-1} g_0 +  Y_kg_{k-r+1}$. Now by the first part $[Y_k, g_{k-r+1}] \in \g^{od}$ and hence $W_{r-1}=0$. But then $Y_k$ and $g_{k-r+1}$ commute, which implies $g_{k-r+1} \in \ft$ since $Y_k \in \ft^{reg}$.

For part \ref{intorr} write $\bar{h} = hgh^{-1} \in T_k$ and consider $\bar{h}h=hg$ term by term. An argument analogous to the one for part \ref{secnd} gives the desired statement.
\end{proof}

Next we study regular semisimple $G_k$-coadjoint orbits i.e. let $C = z^{-k}C_k + \dots + z^{-1}C_1 \in \ft_k^*$ with $C_k \in \ft^{reg}$ and write $\FO_C = \{ \Ad_gC\ |\ g\in G_k\}$ for the coadjoint orbit through $C$. 

\begin{lemma}\label{codecomp} There is an isomorphism
\begin{align*} \Gamma: (G \times B_k^{od})/T &\rightarrow \FO_C \\
								[g,b] &\mapsto \Ad_{gb}C,
\end{align*}
where we put $B_k^{od} =\{ b \in B_k\ |\ b_1,\dots,b_{k-1} \in \g^{od}\}$. Here the action of $T$ on $G \times B_k^{od}$ is given by $(g_0,b)t_0= (g_0t_0, \Ad_{t_0^{-1}}b)$. In particular, $(G \times B_k^{od}) \rightarrow (G \times B_k^{od})/T$ is a Zariski locally trivial $T$-principal bundle.
\end{lemma}
\begin{proof} It follows from Lemma \ref{intor}.\ref{secnd} that $\Stab_{G_k}(C) = T_k$, hence we have an isomorphism $G_k/T_k \cong \FO_C$ which sends $[g] \in G_k/T_k$ to $\Ad_gC$. Since $G_k = G \ltimes B_k$ and $T_k = T \times (T_k \cap B_k)$ we can further write 
\[G_k/T_k \cong \left(G \times  B_k/(T_k \cap B_k)\right)/T.\]
Finally, given $b \in B_k$ a direct computation shows that there are unique $t \in  T_k \cap B_k$ and $b' \in B_k^{od}$ such that $b=b't$. This gives $B_k/(T_k \cap B_k) \cong B_k^{od}$. It is then straightforward to check that the $T$-action on $G \times B_k^{od}$ is as indicated, and the final statement follows since $G \rightarrow G/T$ is a Zariski locally trivial $T$-principal bundle.
\end{proof}

%The key observation in the proof of Lemma \ref{codecomp} was that every $b \in B_k$ has a unique factorization $b = b' t$ with $t \in  T_k \cap B_k$ and $b' \in B_k^{od}$. This gives a morphism $\Phi:B_k \rightarrow B_k^{od}$ satisfying
%\begin{equation} \label{phiproj}\Ad_b C = \Ad_{\Phi(b)} C \end{equation}
%for every $b \in B_k$ since $C \in T_k$. Hence we obtain the following automorphism, which we record for later:
%\begin{align} \label{adauto} \Psi: G \times B_k^{od} &\rightarrow G \times B_k^{od}\\
	%												\nonumber	(g,b) &\mapsto (g, \Phi(g^{-1} bg)).
%\end{align}

\subsection{Meromorphic connections}\label{mercon}

In this section we introduce irregular connections on $\BP^1$ following \cite[Section 2]{Bo01}.  Fix an effective $\BZ$-divisor $D= k_1a_1+k_2a_2+\dots k_da_d$, where $a_1,\dots,a_d$ are points in $\BP^1$ and $k_1,\dots, k_d \geq 1$. Write $K$ for the canonical divisor on $\BP^1$. A \textit{meromorphic connection with poles along} $D$ on a rank $n$ vector bundle $V \rightarrow \BP^1$ is a $\BC$-linear morphism
 \[ \nabla: V \rightarrow V \otimes K(D), \]
satisfying the Leibniz rule $\nabla(fs) = f\nabla(s) + s \otimes df$, where $f$ is a local holomorphic function and $s$ a local section of $V$.

If $z$ is a local coordinate around $a_i$ we can write, after fixing a trivialization of $V$, $\nabla = d- A$, where $A$ is a meromorphic matrix of $1$-forms. More precisely we can write
\[ A = A_{k_i} \frac{dz}{z^{k_i}} +\dots + A_1\frac{dz}{z} + A_0dz+\dots\]
with $A_i \in \g$. The non-holomorphic part $\sum_{j=1}^{k_i} A_j\frac{dz}{z^j}$ is called the \textit{principal part} of $\nabla$ at $a_i$. Then $\nabla$ is called \textit{regular} if for every $1\leq i \leq d$ the leading coefficient $A_{k_i}$ is diagonalisable with distinct eigenvalues, if $k_i \geq 2$, or with distinct eigenvalues modulo $\BZ$, if $k_i =1$.

\begin{rem}\begin{enumerate} 
\item If $g:U \rightarrow GL_n$ is the transition function for an other trivialization of $V$ on a neighborhood $U$ of $a_i$, then the transformation $A'$ of $A$ is given by (see for example \cite[Lemma III.1.6]{We07})
\[ A'  = g Ag^{-1}+(dg)g^{-1},\]
hence being regular is independent of the choice of trivialization.

\item In \cite[Definition 2.2]{Bo01} the term 'generic' is used instead of 'regular', however 'generic' will have a different meaning for us, see Definition \ref{genericc}.
\end{enumerate}
\end{rem}

In order to obtain finite dimensional moduli spaces we need to fix a \textit{formal type of order }$k_i$ at each pole $a_i$, that is a matrix of meromorphic one forms 
\begin{equation}\label{formalt} C^i = C^i_{k_i} \frac{dz}{z^{k_i}} + \dots + C^i_1\frac{dz}{z} + \dots,\end{equation}
where $C^i_{k_i} \in \ft^{reg}$ and $C^i_{j} \in \ft$ for $j < k_i$. One can think of $d - C^i$ as a meromorphic connection on the trivial rank $n$ bundle over the formal disc $\spec(\BC[[z]])$ around $a_i$.

A meromorphic connection $(V,\nabla)$ with poles along $D$ has \textit{formal type} $C^i$ at $a_i$ if there exists a local trivialization of $V$ around $a_i$ and a formal bundle automorphism $g\in G(\BC[[z]])$ such that we have $\nabla = d - A$ and $g A g^{-1} + (dg) g^{-1} - C^i$ is a diagonal matrix of holomorphic $1$-forms.

From now on the choice of an effective divisor $D = k_1a_1+k_2a_2+\dots k_da_d$ and formal types $C^i$ for $1 \leq i\leq d$ will be abbreviated by $\mathbf{C}$ and the degree of $D$ by $\mathbf{k} = \sum_{i=1}^d k_i$.

\begin{defn}\label{odrsp} The \textit{open de Rham space} $\FM_n(\mathbf{C})$ is the set of isomorphism classes of meromorphic connections $(V,\nabla)$ on $\BP^1$, where $V$ is a trivial bundle of rank $n$ and $\nabla$ has poles along $D$ with prescribed formal types $C^i$ at $a_i$.
\end{defn}
Notice that $\FM_n(\mathbf{C})$ should correspond inside the whole de Rham space (no assumption on the vector bundle) to the locus, where the underlying bundle is semi-stable, as on $\BP^1$ a semi-stable bundle of degree $0$ is trivial. Hence the name open de Rham space.

We will see in Proposition \ref{odrspa} below, that $\FM_n(\mathbf{C})$ admits the structure of an algebraic variety. In order for this variety to be smooth we need to impose a genericity condition on $\mathbf{C}$, more precisely on the residue terms $C^i_1$ of the $C^i$, which will naturally reappear during the computations later, see Lemma \ref{finalc}. We will thus be very explicit about it. Define for $I \subset \{1,2,\dots,n\}$ the matrix $E_I \in \g$ by
 \begin{eqnarray}\label{eij} (E_I)_{ij} = \begin{cases} 1 & \text{ if } i=j\in I \\ 0 & \text{ otherwise.}	\end{cases} \end{eqnarray}
\begin{defn}\label{genericc} We call $\mathbf{C}$ \textit{generic} if $\sum_{i=1}^d \tr C_1^i = 0$ and for every integer $n'<n$ and subsets $I_1,\dots,I_d \subset \{1,\dots,n\}$ of size $n'$ we have 
\begin{eqnarray}\label{gen}\sum_{i=1}^d \lla C_1^i,E_{I_i}\rra \neq 0.\end{eqnarray}
\end{defn}													
In other words there are no invariant subspaces $V_1,\dots,V_d \subset \BC^n$ of the same dimension such that $\sum_i \tr C^i_{1|V_i} = 0$, if $\mathbf{C}$ is generic. It is clear that we can always find such a generic $\mathbf{C}$ and we will see by direct computations, that the invariants we compute do not depend on the choice of $\mathbf{C}$.

We now give an explicit description of $\FM_n(\mathbf{C})$ in terms of $G_k$-coadjoint orbits. First notice that a formal type $C$ as in (\ref{formalt}) naturally defines an element in $\g_k^*$ by taking the principal part and forgetting $dz$. We denote the $G_k$ coadjoint orbit through $C$ by $\FO_C \subset \g_k^*$. The action of $G_k$ on $\FO_C$ is Hamiltonian with respect to the standard symplectic structure on $\FO_C$ and the inclusion $\FO_C \hookrightarrow \g_k^*$ is a moment map. In particular, a moment map for the induced action of $G \subset G_k$ is given by $\pi_{\res}:\FO_C \rightarrow \g^*$, see (\ref{pis}). Consequently, for formal types $C^1,\dots,C^d$ the action of $G$ on $\FO_{C^1} \times \dots \times \FO_{C^d}$ by simultaneous conjugation admits a moment map
\begin{align*} \mu_d: \FO_{C^1} \times \dots \times \FO_{C^d} &\rightarrow \g^*\\
(Y^1,Y^2,\dots,Y^d) &\mapsto  \sum_{i=1}^d \pi_{\res}(Y^i).
\end{align*}

\begin{prop}\label{odrspa}
\begin{enumerate}
\item[(i)]  For any choice of $\mathbf{C}$ there is a  bijection
\[ \FM_n(\mathbf{C}) \cong \mu_d^{-1}(0) /G.\] 
\item[(ii)] If $\mathbf{C}$ is generic, we can identify $\FM_n(\mathbf{C})$ with the points of the smooth affine GIT quotient $\mu_d^{-1}(0) \sslash G = \spec(\BC[\mu_d^{-1}(0)]^G)$. If non-empty, $\FM_n(\mathbf{C})$ is equidimensional of dimension $\mathbf{k}(n^2-n)- 2(n^2-1)$. 
\end{enumerate}
\end{prop}

\begin{proof} The first part is contained in \cite[Proposition 2.1]{Bo01}, but we will reprove the statement for the convenience of the reader. Fix a coordinate $z$ on $\BP^1$ and assume that none of the $a_i$'s are at infinity. Given an element $[V,\nabla] \in \FM_n(\mathbf{C})$ we can write $\nabla = d - A$ (after fixing a trivialization of $V$) with
\[ A = \sum_{i=1}^d A^i_{k_i}\frac{dz}{(a_i-z)^{k_i}} + \dots + A^i_{1}\frac{dz}{z} + \text{ holomorphic terms},\]
where all $A^i_j \in \g$. Thus by looking at the principal parts of $A$ and forgetting $dz$ we obtain for each $1\leq i\leq  d$ an element $A^i \in \g_k^*$. By our assumptions $\nabla$ is regular at each $a_i$ and hence has formal type $C^i$ if and only if $A^i \in \FO_{C^i}$, see for example \cite[Proposition 1]{BJL79}. Furthermore the condition $\mu_d(A^1,\cdots,A^d) = 0$ is equivalent to $\nabla$ not having a pole at infinity. Finally, an isomorphism of trivial bundles over $\BP^1$ is given by an element in $G$, which corresponds to simultaneous conjugation on $\prod_{i=1}^d \FO_{C^i}$.

Assume now $\mathbf{C}$ is generic. We show first, that $\PGL_n = G/\BC^\times$ acts freely on $\mu^{-1}_d(0)$. Let $(A^1,\dots,A^d) \in \mu_d^{-1}(0)$ and $g\in G$ such that $\Ad_gA^i=A^i$ for $1 \leq i \leq d$. We show now, that $g$ is scalar, by looking at some non-zero eigenspace $V$ of $g$. Then clearly $A^i_1$ will preserve $V$ for all $i$ and by the moment map condition we deduce $\sum_i \tr A^i_{1|V} = 0$. The point is now, that for each $i$ there is a subspace $V'_i$ of the same dimension as $V$ such that 
\begin{equation}\label{painf}\tr A^i_{1|V} = \tr C^i_{|V'_i}.\end{equation}
By the genericity of $\mathbf{C}$, \ref{genericc} this implies then $V = V'_i = \BC^n$ and hence $g$ is scalar.

To prove (\ref{painf}) we fix an $i$ and write $A^i =\Ad_hC^i$ for some $h \in G_{k_i}$. By conjugating $A^i$ and $g$ with the constant term $h_0$ of $h$ we can assume without loss of generality $h \in B_{k_i}$ i.e. $h_0=\unt$. Then $A^i_{k_i} = C^i_{k_i} \in \ft^{reg}$ and thus $g\in T$. Next consider $\bar{h} = hgh^{-1}$, which satisfies $\Ad_{\bar{h}} C^i = C^i$. By Lemma \ref{intor}.\ref{secnd} we have $\bar{h} \in T_k$ and then by \ref{intor}.\ref{intorr} $hgh^{-1} = g$. This implies that $h_j$ preserves $V$ for every $0\leq j\leq k_i-1$ and hence we have $\tr A_{|V} = \tr (\Ad_{h}C^i)_{|V} = \tr \Ad_{h_{|V}}C^i_{|V} = \tr C^i_{|V}$. This proves \eqref{painf} and hence $\PGL_n$ acts freely on $\mu_d^{-1}(0)$.

In particular, all the $G$-orbits in $\mu_d^{-1}(0)$ are closed and hence set-theoretic quotient agrees with the points of the GIT quotient $\mu_d^{-1}(0) \sslash G$ \cite[Theorem 6.1]{Do03}. Furthermore as in Section \ref{symg}, freeness of the $\PGL_n$ action implies that $0$ is a regular value of $\mu_d$, which in turn implies smoothness of $\mu_d^{-1}(0)$ and hence of $\FM_n(\mathbf{C})$. By looking at tangent spaces we see that the dimension of $\FM_n(\mathbf{C})$ is given by 
\[  \dim \prod_{i=1}^d \FO_{C^i} - 2 \dim \PGL_n =  \mathbf{k}(n^2-n)- 2(n^2-1).      \]
\end{proof}

\subsection{Fourier transform of a pole}\label{ftp}
In this section we compute the Fourier transform $\FF([\FO_C]) \in \kexp_\g$ of a coadjoint orbit  $\pi_{res}:\FO_C \rightarrow \g^*$ of a formal type $C$, where we use the language of Section \ref{GCM}. Assuming $k\geq 2$ we can give an explicit formula for $\FF([\FO_C])$, but before we need to introduce some more notation.

As in Section \ref{tmc} we denote by $\FP_n$ the set of partitions of $n$. A semi-simple element $X \in \g$ has \textit{type} $\lambda = (\lambda^1,\dots,\lambda^l) \in \FP_n$ if $X$ has $l$ different eigenvalues $a_1,\dots,a_l$ and the multiplicity of $a_i$ is $\lambda_i$ for $1 \leq i \leq l$. We write $\g_\lambda = \{ X \in \g \ |\ X \text{ has type } \lambda \}$ and $i_\lambda: \g_\lambda \hookrightarrow \g$. Finally, we put $N(\lambda) = \sum_i (\lambda^i)^2$.% and $\Stab(\lambda) = \prod_i [\GL_{\lambda^i}] \in \kvar$. The latter is the class of the stabilizer group of any element in $\g_\lambda$.

\begin{thm}\label{ftpuncturet} For any partition $\lambda \in \FP_n$ we have in $\expm_{\g^\lambda}$ the formula

\begin{equation}\label{ftpuncture} i_\lambda^*\FF([\FO_C]) =\frac{ \BL^{n+\frac{1}{2}\left(k(n^2-2n) + (k-2)N(\lambda)\right)}}{ (\BL-1)^{n}}\left[Z_\lambda, \phi^C\right], \end{equation}
where $Z_\lambda = \{ (g,X) \in G \times \g_\lambda \ |\ \Ad_{g^{-1}}X \in \ft\}$ and $\phi^C(g,X) = \lla C_1,\Ad_{g^{-1}} X \rra$.
Furthermore the pullback of $\FF([\FO_C])$ to the complement $\g \setminus \bigsqcup_\lambda \g_\lambda$ equals $0$. 
\end{thm}

\begin{proof}
 By the formula  (\ref{ftgen}) we have
\[ \FF([\FO_C]) = [\FO_C \times \g, \lla \pi_{irr}\circ pr_{\FO_C} , pr_{\g} \rra] = [\FO_C \times \g, \lla pr_{\FO_C} , pr_{\g} \rra],\]
where for the second equality sign we used the definition of $\lla,\rra$ (\ref{trp}).

By Lemma \ref{codecomp} we can rewrite this in $\expm_\g$ as 
%\[ \FF([\FO_C]) = \left[G/T \times B_k^{od} \times \g, \lla \Gamma \circ pr_{G/T \times B_k^{od}}, pr_{\g} \rra \right].\]

%Since the quotient $G \rightarrow G/T$ is Zariski-locally trivial \todo{reference} we can further write in $\expm_\g$
\[\FF([\FO_C]) = (\BL-1)^{-n} \left[G \times B_k^{od} \times \g, \lla \Gamma \circ pr_{G \times B_k^{od}}, pr_{\g} \rra \right].\]

Now notice that for all $(g,b,X) \in G \times B_k^{od} \times \g$ we have 
\[\lla \Gamma(g,b),X \rra = \lla \Ad_{gb}C, X \rra =  \lla \Ad_{b} C, \Ad_{g^{-1}} X\rra.\]
Thus  we finally obtain
\begin{equation}\label{blabla} \FF([\FO_C]) = (\BL-1)^{-n} \left[G \times B_k^{od} \times \g, \lla \Ad_{pr_{B_k^{od}}}C, \Ad_{(pr_G)^{-1}}pr_{\g} \rra \right].\end{equation}

We will simplify this by applying Lemma \ref{orth}. First write $\g^l$ and $\g^u$ for the subspaces of strictly lower and upper triangular matrices in $\g$ respectively, such that $\g^{od} = \g^l \oplus \g^u$. Next consider the decomposition $B_k^{od} = B_+^{od} \oplus B_-^{od}$ with

\[ B_+^{od} = \begin{cases} \{ b \in B_k^{od} \ |\ b_{\frac{k}{2}} = \dots = b_{k-1} = 0 \} \text{ if } k \text{ is even } \\
													\{ b \in B_k^{od} \ |\ b_{\frac{k-1}{2}} \in \g^l,  b_{\frac{k+1}{2}} = \dots = b_{k-1} = 0 \} \text{ if } k \text{ is odd }
\end{cases} \]
 
\[ B_-^{od} = \begin{cases} \{ b \in B_k^{od} \ |\ b_1 = \dots = b_{\frac{k-2}{2}} = 0 \} \text{ if } k \text{ is even } \\
													\{ b \in B_k^{od} \ |\ b_{\frac{k-1}{2}} \in \g^u,  b_{1} = \dots = b_{\frac{k-3}{2}} = 0 \} \text{ if } k \text{ is odd. }
\end{cases} \]
It follows from Lemma \ref{keycomp} below that there are functions 
\[(h_1,h_2): G \times \g \times B_+^{od} \rightarrow (B_-^{od})^* \times \BC\]
such that $\lla \Ad_b C, \Ad_{g^{-1}} X \rra = \lla h_1(g,X,b_+), b_- \rra + h_2(g,X,b_+)$ for all $g \in G, X \in \g$ and $b = (b_+,b_-) \in B_+^{od} \oplus B_-^{od}$. More explicitly $h_1$ and $h_2$ are given by
\begin{align*} h_2(g,X,b_+) &= \lla \Ad_{b_+} C, \Ad_{g^{-1}} X \rra, \\
							\lla h_1(g,X,b_+), b_- \rra &= \lla \Ad_b C- \Ad_{b_+} C, \Ad_{g^{-1}} X\rra.
\end{align*}
Applying now Lemma \ref{orth} to this decomposition, formula (\ref{blabla}) becomes
\[ \FF([\FO_C]) = (\BL-1)^{-n}\BL^{\dim B_-^{od}} [ h_1^{-1}(0), h_2].\]
Again by Lemma \ref{keycomp} we have
 \[h_1^{-1}(0) = \{ (g,X,b_+) \in G \times \g \times B_+^{od} \ |\ \Ad_{g^{-1}}X \in \ft, [b_+,\Ad_{g^{-1}}X]=0\}.\]
The condition $\Ad_{g^{-1}}X \in \ft$ already implies, that $\FF([\FO_C])$ is supported on  $\bigsqcup_\lambda \g_\lambda$. Notice further that for $(g,X,b_+) \in h_1^{-1}(0)$ as $[b_+,\Ad_{g^{-1}}X]=0$, we have 
\[h_2(g,X,b_+) = \lla C, \Ad_{b_+^{-1}} \Ad_{g^{-1}}X\rra = \lla C,\Ad_{g^{-1}}X\rra,\]
in particular $h_2$ is independent of $b_+$. Furthermore for any $\lambda \in \FP_n$ the pullback $i_\lambda^*h_1^{-1}(0) \rightarrow Z_\lambda$ is a vector bundle of rank $\frac{k-2}{2}(N(\lambda) - n)$ and thus 
\[i_\lambda^*[ h_1^{-1}(0), h_2] = \BL^{\frac{k-2}{2}\left(N(\lambda) - n\right)} [Z_\lambda,\phi^C].\]
Together with $\dim B_-^{od} = \frac{k}{2}\left(n^2-n\right)$ the theorem follows.
\end{proof}

We are left with proving Lemma \ref{keycomp}, for which we need the following explicit formula for the inverse of an element $b = \unt + zb_1 +\dots +z^{k-1}b_{k-1} \in B_k$. If we write $b^{-1} = \unt + zw_1 + \dots +z^{k-1} w_{k-1}$, then $w_i$ is given for any $1 \leq i \leq k-1$ by 
\begin{equation}\label{explinv} w_i = \sum_{m=1}^i (-1)^m \sum_{\substack{(j_1,j_2,\dots,j_m) \\ j_1+\dots+j_m = i}} b_{j_1}\cdots b_{j_m}.\end{equation}

Notice that for $\left\lfloor \frac{k+1}{2}\right\rfloor \leq m \leq k-1$, $b_m$ can appear at most once in each summand on the right hand side of (\ref{explinv}). This is the crucial observation in the proof of Lemma \ref{keycomp}.

\begin{lemma}\label{keycomp} For $X \in \g$ the function 
\begin{align*} \phi_X: B_k^{od} &\rightarrow \BC \\
												b & \mapsto \phi_X(b_1,b_2,\dots,b_{k-1}) =  \lla \Ad_bC,X \rra
\end{align*}
is affine linear in $b_{\left\lfloor \frac{k+1}{2}\right\rfloor}, \dots, b_{k-1}$. It is independent of those variables if and only if $X \in \ft$ and $b_1, b_2,\dots,b_{\left\lfloor \frac{k-2}{2}\right\rfloor}$ commute with $X$.

In this case, if $k$ is odd and we decompose $b_{\frac{k-1}{2}} = b^l + b^u$, where $b^l$ and $b^u$ are strictly lower and upper triangular respectively, then $\phi_X$ is affine linear in $b^u$ and independent in of $b^u$ if and only if $b^l$ commutes with $X$.
\end{lemma}
  
\begin{proof} It follows directly from (\ref{explinv}), that $\phi_X$ depends linearly on $b_i$ for $\left\lfloor \frac{k+1}{2}\right\rfloor \leq i \leq k-1$. 

For $b \in B_k^{od}$, using the notation (\ref{explinv}) we have 
\begin{equation} \label{evrytn} \lla \Ad_bC,X\rra = \tr \sum_{i=1}^k \sum_{j=0}^{i-1} b_j C_i w_{i-j-1}X,\end{equation}
where we use the convention $b_0 = w_0 = \unt$. We start by looking at the dependence of $ \lla \Ad_bC,X\rra$ when varying $b_{k-1}$. The terms in (\ref{evrytn}) containing $b_{k-1}$ are given by 
\[\tr(b_{k-1}C_kX - C_kb_{k-1}X) = \tr [b_{k-1},C_k]X.\]
As $C_k \in \ft^{reg}$, the commutator $[b_{k-1},C_k]$ can take any value $\g^{od}$, thus $\tr [b_{k-1},C_k]X$ is independent of $b_{k-1}$ if and only if $X \in (\g^{od})^\perp = \ft$.

Assume from now on $X \in \ft$. We show now inductively that $\phi_X$ is independent of $b_{\left\lfloor \frac{k+1}{2}\right\rfloor  }, \dots, b_{k-2}$ if and only if $b_1, b_2,\dots,b_{\left\lfloor \frac{k-2}{2}\right\rfloor}$ all commute with $X$.

To do so, fix $\left\lfloor \frac{k+1}{2}\right\rfloor \leq m \leq k-2$ and assume that $b_1,\dots,b_{k-2-m}$ commute with $X$. Consider the element 
\[b' = \unt + zb_1+\dots + z^{m-1}b_{m-1} + z^{m} b_{m}X+z^{m+1} b_{m+1} + \dots + z^{k-1} b_{k-1} \in B_k^{od}.\]
The point is now, that the $b_m$-parts of the explicit formulas for $\lla \Ad_bC,X\rra$ and $\Res_0 \tr (b' C b'^{-1}) = \Res_0\tr( C)= \tr (C_1)$ are very similar. Indeed, from (\ref{evrytn}) we see, that all the terms containing $b_m$ in $\lla \Ad_bC,X\rra$ are contained in 
\begin{equation}\label{lala1} \tr \sum_{i=m+1}^k b_mC_i w_{i-m-1}X + \sum_{r=m}^{k-1} \sum_{i=r+1}^k b_{i-r-1} C_i w_r X.\end{equation}

To write a formula for $\Res_0 \tr( b' C b'^{-1})$ we write $b'^{-1} = \unt +z w'_1+\dots + z^{k-1} w'_{k-1}$. Then we can use a similar expression as (\ref{evrytn}) to conclude that all the terms containing $b_m$ in $\Res_0 \tr (b' C b'^{-1})$ are contained in 
\begin{equation}\label{lala2} \tr \sum_{i=m+1}^k b_m X C_i w'_{i-m-1} + \sum_{r=m}^{k-1} \sum_{i=r+1}^k b_{i-r-1} C_i w'_r.\end{equation}

Next we want to study the dependence of the difference (\ref{lala1}) - (\ref{lala2}) on $b_m$. Notice first, that since $[b_i,X]=0$ for $1\leq i \leq k-2-m$ also $[w_i,X]=0$ for $1\leq i \leq k-2-m$ and furthermore $[w_{k-m-1},X] = [X,b_{k-m-1}]$. From (\ref{explinv}) we also see $w_i' = w_i$ for all $1 \leq i < m$. Finally, we remark that $w_rX -w_r'$ is independent of $b_m$ for $m \leq r \leq k-2$ and the terms containing $b_m$ in $w_{k-1}X-w'_{k-1}$ are given by $b_m [b_{k-m-1},X]$. Combining all this we see that the terms containing $b_m$ in (\ref{lala1}) - (\ref{lala2}) are just 
\[ \tr \left(b_m C_k [X,b_{k-m-1}] + C_k b_m [b_{k-m-1},X]\right) = \tr [b_m,C_k][X,b_{k-m-1}].\]
Since $C_k \in \ft^{reg}$, the commutator $[b_m,C_k]$ can take any value in $\g^{od}$ as we vary $b_m \in \g^{od}$. Hence in order for $\tr [b_m,C_k][X,b_{k-m-1}]$ to be constant, we need $[X,b_{k-m-1}] \in \ft$. Since $X \in \ft$ this is only possible if $[X,b_{k-m-1}] = 0$, which finishes the induction step.
		
Finally, we consider the special case $m = \frac{k-1}{2}$, when $k$ is odd. Then by the same argument as before we obtain, that all the terms in $\lla \Ad_bC,X \rra$  which depend on $b_m$ are $\tr [b_m,C_k][X,b_{m}]$. Now using the decomposition $b_m = b^l+b^u$ we have 
\[\tr [b_m,C_k][X,b_{m}] = 2 \tr(b^u[C_k[X,b^l]]).\]
Since the orthogonal complement of strictly upper triangular matrices are the upper triangular matrices we see that $ \tr(b^u[C_k[X,b^l]])$ is independent of $b^u$ if and only if $[X,b^l] = 0$.
\end{proof}

%\begin{lemma} For $X \in \g$ we have $[ B_k^{od}, \phi_X] = 0$ in $\kexp$ unless $X \in \ft$. In this case we have
%\[ \left[ B_k^{od},\phi_X \right] = \BL^{\frac{1}{2}(k(n^2-n) + (k-2) \dim Z^{od}_\g(X))}[\spec(\BC), \lla C_1,X \rra ], \]
%where $Z_\g(X) = \{ Y \in \g^{od} \ |\ [X,Y] =0\} \subset \g^{od}$.
%\end{lemma}

%\begin{proof} This is just an application on the previous Lemma \ref{keycomp}. First if $X \notin \ft$, then $\phi_X$ is not constant and hence $[ B_k^{od}, \phi_X] = 0$ by Lemma \ref{orth}. Now assume $X \in \ft$ and for simplicity $k$ is even. Then we can split up $B_k^{od} = B_{<\frac{k}{2}}^{od} \oplus B_{\geq \frac{k}{2}}^{od}$, where we put 
%\[ B_{<\frac{k}{2}}^{od} = \{ (b_1,\dots,b_{\frac{k-2}{2}}) \ |\ b_i \in \g^{od} \} \text{ and } B_{\geq \frac{k}{2}}^{od} = \{ (b_{\frac{k}{2}},\dots,b_{k-1}) \ |\ b_i \in \g^{od}\}.\]
 %Then $B_{<\frac{k}{2}}^{od}$ is naturally a subspace of $B_k^{od}$ and we put $Z = \{ b \in B_{<\frac{k}{2}}^{od} \ |\ [b,X] = 0\}$. Then once again combining Lemmas \ref{keycomp} and \ref{orth} we obtain 
%\[ \left[ B_k^{od},\phi_X \right] = \BL^{\dim B_{\geq \frac{k}{2}}^{od}} [Z, \phi_{X|Z}].\]
%But for every $b \in Z$ we have $\phi_X(b) = \lla \Ad_b C,X \rra = \lla C, \Ad_{b^{-1}} X\rra = \lla C,X\rra = \lla C_1,X \rra$, hence $[Z, \phi_{X|Z}] = \BL^{\dim Z}[\spec(\BC), \lla C_1,X \rra]$, which gives the desired formula. The case where $k$ is odd is analogous.
%\end{proof}

\subsection{Motivic classes of open de Rham spaces}\label{mcodr}

Let $D=k_1a_1+\dots+k_da_d$ be an effective divisor on $\BP^1$ and $C^1,\dots,C^d$ formal types such that $\mathbf{C}$ is generic. We compute now $[\FM_n(\mathbf{C})] \in \mathscr{M}$, under the assumption that all poles are at least of order two, i.e. $k_i \geq 2$ for $1\leq i \leq d$. For a partition $\lambda = (\lambda_1\geq \dots \geq \lambda_l) \in \FP_n$ we define the numbers $l(\lambda) = l$, $N(\lambda) = \sum_{j=1}^l \lambda_j^2$ and $ m_j(\lambda)$ the multiplicity of $j\in \BN$ in $\lambda$. Furthermore we put 
\[\Stab(\lambda) = \prod_{j=1}^l \GL_{\lambda_j} \subset \GL_n.\] 

\begin{thm}\label{mainodr} The motivic class of $\FM_n(\mathbf{C})$ in $\mathscr{M}$ for a generic $\mathbf{C}$, where all poles are of order at least $2$, is given by 
\begin{align} \label{totalct}   [\FM_n(\mathbf{C})] &= \\
\nonumber \frac{\BL^{\frac{\mathbf{k}}{2}(n^2-2n)+n(d-n)+1 } }{(\BL-1)^{nd-1}} &\sum_{\lambda \in \FP_n} \frac{(-1)^{l(\lambda)-1}(l(\lambda)-1)! (n!)^d}{\prod_{j} (\lambda_j!)^d \prod_{r \geq 1} m_r(\lambda)!} \BL^{\frac{N(\lambda)}{2}(\mathbf{k}-2d)} [\Stab(\lambda)]^{d-1}.\end{align}
\end{thm}

We start by simplifying $[\FM_n(\mathbf{C})]$ in the following standard way.

\begin{lemma} In $\mathscr{M}$ we have 
\begin{equation}\label{firsta} [\FM_n(\mathbf{C})] =  \frac{(\BL-1) [\mu_d^{-1}(0)]}{[G]}. \end{equation}
\end{lemma}

\begin{proof} Similar to Lemma \ref{codecomp} we consider the $T^d_k$-principal bundle $\alpha:G_k^d \rightarrow \prod_{i=1}^d \FO_{C^i}$. Notice that $\alpha$ is $G$-equivariant with respect the free $G$-action on $G_k^d$ given by left-multiplication. By restriction we obtain a $G$-equivariant $T^d_k$-principal bundle $X \rightarrow \mu_d^{-1}(0)$. Taking the (affine GIT-)quotient by $G$, we obtain a $\BC^\times \setminus T^d_k/$-principal bundle $G \setminus X \rightarrow \FM_n(\mathbf{C})$. Also $X \rightarrow G \setminus X$ is a $G$-principal bundle, as it is the restriction of $G_k^d \rightarrow G\setminus G_k^d$. As the groups $T^d_k, G_k$ and $\BC^\times \setminus T^d_k$ are special \cite[Section 4.3]{serre58}, all the principal bundles here are Zariski locally trivial and we get
\[ [\FM_n(\mathbf{C})] = \frac{[G\setminus X]}{[\BC^\times \setminus T^d_k]} = \frac{[X]}{[G][\BC^\times \setminus T^d_k]} = \frac{[\mu_d^{-1}(0)][T^d_k](\BL-1)}{[G][T^d_k]} = \frac{(\BL-1) [\mu_d^{-1}(0)]}{[G]} .\]
\end{proof}

By (\ref{firsta}) it is enough to determine $[\mu_d^{-1}(0)] = 0^*[\FO_{C^1}\times \cdots \times \FO_{C^d}]$, where we consider $0:\spec(\BC) \rightarrow \g^{*}$ as a morphism. Using motivic convolution and in particular Proposition \ref{fouconv} we have an equality of motivic Fourier transforms

\begin{equation}\label{foffs} \FF\left(\left[\prod_{i=1}^d  \FO_{C^i}\right] \right)= \FF\left([\FO_{C^1}]*\dots * [\FO_{C^d}] \right) = \prod_{i=1}^d \FF([\FO_C]) \in \kexp_\g.\end{equation}

Notice that the last product is relative to $\g$, hence we have by Theorem \ref{ftpuncturet} for every $\lambda \in \FP_n$ 

\begin{align*} i_\lambda^*\prod_{i=1}^d \FF([\FO_C]) &= (\BL -1)^{-nd} \BL^{dn + \frac{1}{2}(\mathbf{k}(n^2-2n) +N(\lambda)(\mathbf{k}-2d))}\prod_{i=1}^d\left[Z_\lambda, \phi^{C^i}\right] \\
&= (\BL -1)^{-nd} \BL^{dn + \frac{1}{2}(\mathbf{k}(n^2-2n) +N(\lambda)(\mathbf{k}-2d))} \left[Z^d_\lambda,\sum_{i=1}^d \phi^{C^i} \right], \end{align*}
with the notations $\mathbf{k} = \sum k_i$, $Z_\lambda^d$ for the $d$-fold product $Z_\lambda \times_\g \dots \times_\g Z_\lambda$ and $\sum_i \phi^{C^i}$ for the function taking $(z_1,\dots,z_d)\in Z_\lambda^d$ to $\sum_i \phi^{C^i}(z_i)$.
By Fourier inversion \ref{finv} we thus get

\begin{align} \nonumber [\mu_d^{-1}(0)]& = 0^* \FF\left(\prod_{i=1}^d \FF([\FO_C])\right) \\
	\label{seca}&=  (\BL -1)^{-nd} \BL^{dn + \frac{1}{2}\mathbf{k}(n^2-2n) - n^2} \sum_{\lambda \in \FP_n} \BL^{\frac{1}{2}N(\lambda)(\mathbf{k}-2d)} \left[Z^d_\lambda,\sum_i \phi^{C^i} \right]
\end{align}
This leaves us with understanding $\left[Z^d_\lambda,\sum_i \phi^{C^i} \right]$ as an element of $\kexp$.

We start by taking a closer look at $Z_\lambda = \{ (g,X) \in G \times \g_\lambda \ |\ \Ad_{g^{-1}}X \in \ft\}$. If we put $\ft_\lambda = \ft \cap \g_\lambda$ we have an isomorphism 
\begin{align} \label{decoup} Z_\lambda &\xrightarrow{\sim} \ft_\lambda \times G \\
							\nonumber		(g,X) &\mapsto (\Ad_{g^{-1}} X, g).
\end{align}

Next we need to fix some notation to describe $\ft_\lambda$ combinatorially. To parametrize the eigenvalues of elements in $\ft_\lambda$ define for any $m \in \BN$ the open subvariety $\BA_{\circ}^m \subset \BA^m$ as the complement of $\cup_{i \neq j} \{x_i = x_j\}$.

Furthermore we need some discrete data. A \textit{set partition of $n$} is a partition $I = \{I_1,I_2,\dots ,I_l\}$ of $\{ 1,2,\dots ,n\}$ i.e $I_i \cap I_j = \emptyset$ for $i \neq j$ and $\cup_i I_i = \{ 1,2,\dots ,n\}$. For $\lambda = (\lambda_1 \geq\dots \geq \lambda_l) \in \FP_n$ we write $\FP_\lambda$ for the set of set partitions $I=\{I_1,\dots,I_l\}$ of $n$ such that $\{ |I_1|,\dots,|I_l|\} = \{\lambda_1,\dots,\lambda_l\}$. We stress that the $I_i$'s are not ordered and hence $|\FP_\lambda| = \frac{n!}{\prod_{i = 1}^l \lambda_i!\prod_{j\geq 1} m_j(\lambda)!}$, where $ m_j(\lambda)$ denotes the multiplicity of $j$ in $\lambda$. 

With this notations we get a parametrization
\begin{align} \label{param} p:\FP_\lambda \times \BA_\circ^l &\xrightarrow{\sim} \ft_\lambda \\
									\nonumber	(I, \alpha) &\mapsto \sum_{j=1}^l \alpha_j E_{I_j},
\end{align}
where for any subset $J \subset \{1,\dots,n\}$ $E_J$ is defined as in (\ref{eij}) and we require $|I_j| = \lambda_j$. Notice that $p$ is not uniquely defined this way as we might switch $I_a$ and $I_b$ in a given $I$, if $|I_a|=|I_b|$. As this doesn't matter for us, we will just fix a $p$ once and for all.

For any $I \in \FP_\lambda$ we also define $\ft_\lambda^I$ as the image of the restriction $p_I = p_{|{\{I\} \times \BA_\circ^l}}: \{ I\} \times \BA_\circ^l \xrightarrow{\sim} \ft_\lambda^I$.

\begin{lemma}\label{thirda} The following relation holds in $\kexp$
\[ \left[Z_\lambda^d, \sum_{i=1}^d \phi^{C^i} \right] = \prod_{r\geq 1} (m_\lambda(r)!)^{d-1} [G \times \Stab(\lambda)^{d-1}] \sum_{(I^1,\dots,I^d)\in (\FP_\lambda)^d} \left[\BA_\circ^{l}, \sum_{i=1}^d \lla C_1^i,p_{I^i}\rra\right]. \]
%Here we put $\Stab(\lambda) = \prod_{j=1}^l \GL_{\lambda_j} \subset G$.
\end{lemma}
\begin{proof} 
Combining (\ref{decoup}) and (\ref{param}) we have $Z_\lambda \cong \bigsqcup_{I \in \FP_\lambda} \BA^l_\circ \times G$. One can then  check that $(\BA^l_\circ \times G) \times_{\g_\lambda} (\BA^l_\circ \times G)$ has $\prod_{r\geq 1} m_\lambda(r)!$ connected components, each of which is isomorphic to $\BA^l_\circ \times G \times \Stab(\lambda)$. Applying this reasoning $d-1$ times and keeping track of the isomorphisms gives the desired equality. \end{proof}

%under this isomorphism $\BA^l_\circ \times G \rightarrow \g_\lambda$ is a $\prod_{r\geq 1} m_\lambda(r)!$-fold cover of a $\Stab(\lambda)$-principal bundle, hence we have
%\[(\BA^l_\circ \times G) \times_{\g_\lambda} (\BA^l_\circ \times G) \cong \BA^l_\circ \times G \times \Stab(\lambda).\]
%Here we also use that the stabilizer of any element in $\g_\lambda$ is conjugate to $\Stab(\lambda)$. This gives then an isomorphism $Z_\lambda^d \cong G \times \Stab(\lambda)^{d-1} \bigsqcup_{(I^1,\dots,I^d)\in (\FP_\lambda)^d} \BA_\circ^{l}$ under which $\sum_{i=1}^d \phi^{C^i}$ corresponds to 
%\begin{align*} \sum_{i=1}^d \lla C_1^i,p_{I^i}\rra: \BA^l_\circ &\rightarrow \BC \\
	%																							\alpha &\mapsto \sum_{i=1}^d \lla C_1^i,p_{I^i}(\alpha)\rra, \end{align*}
																								
%which finishes the proof.																							

\begin{proof}[Proof of Theorem \ref{mainodr}]  Combining (\ref{firsta}), (\ref{seca}) and Lemma \ref{thirda} we are left with computing the character sum $\left[\BA_\circ^{l}, \sum_{i=1}^d \lla C_1^i,p_{I^i}\rra \right]$ for a fixed $d$-tuple $(I^1,\dots,I^d)$ of set partitions of $n$. For $\alpha \in \BA_\circ^l$ we can write 
\[ \sum_{i=1}^d \lla C_1^i,p_{I^i}(\alpha)\rra = \sum_{i=1}^d \lla C_1^i, \sum_{j=1}^l \alpha_j E_{I_j^i} \rra = \sum_{j=1}^l \alpha_j \sum_{i=1}^d \lla C_1^i,E_{I^i_j}\rra. \]

Now for a fixed $1\leq j \leq l$ we have by definition $|E_{I_j^1}|=\dots =|E_{I_j^d}|$. Thus by our genericity assumption \eqref{gen} the numbers $\beta_j = \sum_{i=1}^d \lla C_1^i,E_{I^i_j}\rra$ satisfy the assumptions of Lemma \ref{finalc} below, and we deduce

\begin{equation}\label{tadaa} \left[\BA_\circ^{l}, \sum_{i=1}^d \lla C_1^i,p_{I^i}\rra\right] = (-1)^{l-1} (l-1)! \BL, \end{equation}
which proves Theorem \ref{mainodr}.\end{proof}

\begin{lemma} \label{finalc} Let $\beta_1,\dots,\beta_m$ be complex numbers such that $\sum_{j=1}^m \beta_j = 0$ and for $J \subset \{1,2,\dots,m\}$ be a proper subset  $\sum_{j\in J} \beta_j \neq 0$. Then for the function $\lla \cdot, \beta\rra: \BA^m_\circ \rightarrow \BC$, $\alpha \mapsto \sum_{j=1}^m \alpha_j\beta_j$ we have
\[ [\BA^m_{\circ}, \lla \cdot, \beta\rra ] = (-1)^{m-1} (m-1)! \BL \in \kexp.\]
\end{lemma}

\begin{proof} We use induction on $m$. For $m=1$ we have $\beta_1 =0$ and $\BA^1_\circ = \BA^1$, hence the statement is clear. For the induction step consider $\BA^{m}_\circ$ as a subvariety of $\BA^{m-1}_\circ \times \BA^1$. As $\beta_m \neq 0$ we have $[\BA^{m-1}_\circ \times \BA^1, \lla \cdot, \beta\rra] = 0$, hence 
\[ [\BA^m_{\circ}, \lla \cdot, \beta\rra ] = - [\BA^{m-1}_\circ \times \BA^1 \setminus \BA^{m}_\circ,  \lla \cdot, \beta\rra ].\] 
Now notice that the complement $\BA^{m-1}_\circ \times \BA^1 \setminus \BA^{m}_\circ$ has $m-1$ connected components, each of which is isomorphic $\BA^{m-1}_\circ$, which implies the formula.
\end{proof}

\subsection{Remarks on finite fields and purity}

The description of $\FM_n(\mathbf{C})$ in Proposition \ref{odrspa} allows us to consider open de Rham spaces over any field, in particular over a  finite field $\BF_q$.   By taking a spreading out of $\FM_n(\mathbf{C})$ over some finitely generated $\BZ$-algebra we see that if the characteristic of $\BF_q$ is large enough, Theorem \ref{mainodr} also hold when we replace every motivic class with the number of rational points over $\BF_q$  \cite[Appendix]{HR08}. 

We can even say more. Namely the proof of Theorem \ref{ftpuncturet} implies that the Fourier transform of the count function $\#_C: \g^* \rightarrow \BZ$, $Y \mapsto |\pi_{res}^{-1}(Y)(\BF_q)|$ associated to the coadjoint orbit $\pi_{res}:\FO_C \rightarrow \g^*$ is supported on semi-simple elements in $\g$, whose eigenvalues are in the field $\BF_q$. Given such an $X \in \g$ of type $\lambda \in \FP_n$, the $\BF_q$-version of formula (\ref{ftpuncture}) reads 
\begin{equation}\label{ffft} \FF(\#_C)(X) = \frac{q^{n+\frac{1}{2}\left(k(n^2-2n)+(k-2)N(\lambda)\right)}}{(q-1)^n} \left|\Stab(\lambda)(\BF_q)\right| \sum_{t \in \Ad_GX \cap \ft} \Psi(\lla C_1,t\rra),\end{equation}
 where $\Ad_GX \subset \g$ denotes the orbit of $X$ under the adjoint action and $\FF$ and $\Psi$ are defined as in Section \ref{GCM}.

Now even though our argument in Section \ref{ftp} does not work for $k=1$, formula (\ref{ffft}) continues to hold for $X\in \g$ semi-simple with eigenvalues in $\BF_q$. Indeed in this case $\#_C$ is the characteristic function of the coadjoint orbit $\FO_C \subset \g^*$ and the formula (2.5.5) in \cite{HLV11} for its Fourier transform agrees with (\ref{ffft}), when we put $k=1$. 

Hence for a generic $\mathbf{C}$ with at least one pole of order $\geq 2$, we can compute $|\FM_n(\mathbf{C})(\BF_q)|$ as in Section \ref{mcodr}, since we have to evaluate the product in (\ref{foffs}) only on semi-simple elements whose eigenvalues are in $\BF_q$ (if all poles are of order $1$, one also has to consider non semi-simple elements).

\begin{cor}\label{onlyp} For a generic $\mathbf{C}$ with at least one pole of order $\geq 2$, the number of $\BF_q$-rational points of $\FM_{n}(\mathbf{C})$ is given by formula (\ref{totalct}) when we replace $\BL$ by $q$.

In particular, in this case $\FM_n(\mathbf{C})$ is non-empty and connected.
\end{cor}
\begin{proof} The explanation why (\ref{totalct}) continues to hold is given in the previous paragraph. By Katz's theorem \cite[Theorem 6.1.2]{HR08} (or in the motivic case by applying (\ref{reale})) we see that the same formula (\ref{totalct}) also gives the $E$-polynomial $E(\FM_{n}(\mathbf{C});x,y)$, when $\BL$ is replaced everywhere with $xy$. By a direct inspection we then see that $E(\FM_{n}(\mathbf{C});t,t)$ is a monic polynomial of degree $2\mathbf{k}(n^2-n)- 4(n^2-1) = 2 \dim \FM_{n}(\mathbf{C})$, which implies that $\FM_{n}(\mathbf{C})$ is non-empty and by Lemma \ref{connoc} also connected.
\end{proof}

It is somewhat unfortunate that we have to use finite fields to be able to include order $1$ poles in our computations. We plan to come back to this problem in the future and hopefully prove formula (\ref{ffft}) in the motivic setting also for $k=1$.

We finish by looking at some special cases of (\ref{totalct}). For $n=2$ the formula reads  

\[ [\FM_2(\mathbf{C})] = \frac{\BL^{\mathbf{k}-3} (\BL^{\mathbf{k}-d-1}(\BL+1)^{d-1}-2^{d-1})}{\BL-1}.\]
For small values of $d$ $[\FM_2(\mathbf{C})]$ is then given by 
\begin{align*} d=2; \ \ &  \BL^{\mathbf{k}-3}\left( \BL^{	\mathbf{k}-3} + 2\BL^{\mathbf{k}-4}+ 2\BL^{\mathbf{k}-5}+\dots +2\right), \\
d=3; \ \ &  \BL^{\mathbf{k}-3} \left( \BL^{\mathbf{k}-3} + 3\BL^{\mathbf{k}-4} + 4 \BL^{\mathbf{k} - 5}+\dots + 4\right),\\
d=4; \ \ &  \BL^{\mathbf{k}-3} \left( \BL^{\mathbf{k}-3} + 4\BL^{\mathbf{k}-4} + 7 \BL^{\mathbf{k} - 5} + 8\BL^{\mathbf{k}-6} +\dots + 8\right).
\end{align*}

It turns out that in all examples we can compute, the coefficients of $[\FM_{n}(\mathbf{C})]$ as a polynomial in $\BL$ will always be positive, in particular the coefficients of $E(\FM_{n}(\mathbf{C});x,y)$ seem to be positive. By (\ref{epd}) a sufficient condition for this is that the compactly supported cohomology of $\FM_{n}(\mathbf{C})$ is pure i.e. $h_c^{p,q;i} = 0$ unless $p+q=i$. If all poles in $\mathbf{C}$ are of order 1 this is proven in \cite[Theorem 2.2.6]{HLV11} using the description of $\FM_n(\mathbf{C})$ as a quiver variety.  In \cite{HWW17} we obtain a quiver like description of $\FM_n(\mathbf{C})$ for poles of any order, giving more evidence and a possible strategy for the following natural conjecture.

\begin{conj}\label{purcon} For any generic $\mathbf{C}$ the (compactly supported) mixed Hodge structure of $\FM_n(\mathbf{C})$ is pure.
\end{conj}

\newpage

\section{Push-forward Measures of Moment Maps over Local Fields}\label{localpf}

In this final Chapter we consider the same situation as in Section \ref{genred}, but with base field a local field $F$ and instead of the motivic measure we consider the natural Haar measure $\h$ on $F$. These two measures are not unrelated. Namely $\h$ will induce a canonical measure $d_X$ on any variety $X$ which is smooth over the ring of integers $\FO \subset F$. By a theorem of Weil \cite{WE12} the volume of $X(\FO)$ is up to a factor equal to $|X(\BF_q)|$, where $\BF_q$ is the residue field of $F$, which is always assumed to be finite. As counting over finite fields is essentially a realization of the motivic measure, we see that $d_X$ is in some sense a refinement of the motivic measure.

As a natural question we will thus study the analogue of Proposition \ref{prop1}, i.e. the push forward of the Haar measure along the moment map
\begin{equation}\label{mmag}\mu: V \times V^* \rightarrow \g.\end{equation}
By Weil's theorem we do not expect to see anything new at the $\FO$-smooth fibers of $\mu$, which is why we focus our attention to $\mu^{-1}(0)$. More precisely we are interested in the 'relative volume' $B_\mu$ of $\mu^{-1}(0)$ i.e. the value of the density  $\mu_*(\h_{V \times V^*})/\h_\g$ at $0$ and the geometric information it contains. We are not able to give a precise answer to this question here, but we hope that the computations and conjectures we present will be a starting point for interesting future research directions.

We now explain the structure of this chapter in more detail. The first section contains the necessary background on local fields, we introduce in particular the Fourier transform operator and the corresponding inversion formula.

In the second section we define the local Igusa zeta function $\FI_f(s)$ of a polynomial map $f: \BA^n \rightarrow \BA^m$ between affine spaces, our main example being $f=\mu$ a moment map as in (\ref{mmag}). Not only is $\FI_\mu(s)$ a strictly finer invariant than $B_\mu$, but while $B_\mu$ can be infinite, $\FI_\mu(s)$ is always a rational function in $q^{-s}$ and if $B_\mu$ is finite we can recover it as a residue of $\FI_\mu(s)$ at a simple pole. Another reason to consider $\FI_f(s)$ is the functional equation it satisfies, which will explain certain symmetries of $B_\mu$.

In Section \ref{pfmmm} we use our localization philosophy, or more precisely a $p$-adic analogue of Proposition \ref{prop1}, to give a formula for computing $\FI_\mu(s)$. In general it seems quite hard to evaluate this formula, but if we restrict ourselves to the moment maps that appear in the construction of hypertoric varieties (\ref{toricmm}) we can be very explicit. In this case we can express $\FI_\mu(s) = \FI_\FA(s)$ in terms of the combinatorics of the associated hyperplane arrangement $\FA$, which is the content of Section \ref{hzf}.

A direct consequence of this explicit formula is that the (real parts of the) poles of $\FI_\FA(s)$ are contained in a finite set of negative integers $\mathscr{P}_\FA$. In Section \ref{poles} we give a criterion, for when an integer in $\mathscr{P}_\FA$ is an actual pole	 of $\FI_\FA(s)$ and deduce that the two largest and the smallest number in $\mathscr{P}_\mu(s)$ will always be poles of $\FI_\FA(s)$. The interest in the poles of $\FI_\FA(s)$, or more generally 
$\FI_f(s)$, comes from Igusa's long standing monodromy conjecture \cite{Ig88}. One version of the conjecture states that the poles of $\FI_f(s)$ should 
agree with the roots of the so-called Bernstein-Sato polynomial $b_f$ of $f$. The conjecture has been checked in many cases when $f$ is a single polynomial, but in general only a few examples are known \cite{HMY07}. It would thus be interesting to see, if one can use analogous localization ideas to compute the roots of $b_\mu$, which is a question we will try to answer in the future.

In the last two sections we come back to the relative volume $B_\mu$. In \ref{relap} we use the description of $B_\mu$ as a residue of $\FI_\FA(s)$ to prove that the 'numerator' $B'_\mu$ of $B_\mu$ is palindromic as a polynomial in $q$. Based on numerical evidence we furthermore conjecture that $B'_\mu$ has positive coefficients. 

Finally, in Section \ref{irhd} we consider hyperplane arrangements which come from a quiver $\Gamma$. Here our motivation comes from a result of Crawley-Boevey and Van den Bergh \cite{CV04}, which says that the number of indecomposable representations of $\Gamma$ over $\BF_q$ for an indivisible dimension vector is up to a factor equal to the number of stable points on $\mu^{-1}(0)(\BF_q)$. In our case, we consider idecomposable $1$-dimensional representations of $\Gamma$ over the finite quotient rings $\FO \rightarrow \FO/\m^\alpha$, where $\m \subset \FO$ denotes the maximal ideal. Using a formula of Mellit \cite{Me162} we show that the asymptotic number of such representations as $\alpha \rightarrow \infty$ is given by a rational function  $A_\Gamma(q)$. We finish by giving some numerical evidence for the conjecture that the numerator $A'_\Gamma(q)$ of $A_\Gamma(q)$ equals $B'_\mu(q)$.

\begin{comment}
Our leading theme, that push-forward measures under moment maps should exhibit some localization phenomena \todo{phrase this in the introduction}, has so far only be applied to the tautological motivic measure, namely in \ref{prop1} and \ref{ftpuncturet}. As a natural next step we consider in this last part of the thesis push-forwards of the $p$-adic Haar measure. 

In the first section we recall some basic features of local fields, in particular the notion of Fourier transform. As opposed to the motivic setting, we are naturally led to consider Fourier transforms on spaces with infinite volume, which leads to some convergence issues \todo{reference to example}.  

\end{comment}

\subsection{Local fields and some harmonic analysis}

In this section we recall some basic facts about local fields. Everything we say here can be found in various places, for example \cite{Se13,Ta75}.

By a \textit{local field} $F$ we will always mean a locally compact, non-discrete, totally disconnected field, where locally compact means that both abelian groups $(F,+)$ and $F^\times$ are locally compact. With this definition there are two kinds of local fields:
\begin{enumerate}
\item[] \textbf{$\mathbf{\Char(F) = 0:}$} $F$ is a finite extension of a $p$-adic field $\BQ_p$ for some prime number $p$.
\item[] \textbf{$\mathbf{\Char(F) >0:}$} $F$ is the field of rational functions over a finite field $\BF_q$ i.e. $F = \BF_q((X))$.
\end{enumerate}

We fix now a local field $F$ once and for all. Write $\nu:F \rightarrow \BZ \cup \{\infty\}$ for the valuation and $|\cdot |: F \rightarrow \BQ$ for the norm. %given by $|x| = q^{-\nu(x)}$ for $x \in F$
The latter is multiplicative and satisfies the non-archimedean triangle inequality
\begin{equation}\label{nontri}
|x+y| \leq \max\{|x|,|y|\} \text{ for all } x,y \in F.
\end{equation}
Moreover, we have an equality in \eqref{nontri} whenever $|x| \neq |y|$.

Hence the unit ball $\FO = \{ x \in F \ |\ |x| \leq 1 \}$ is a sub-ring of $F$ called \textit{the ring of integers}. It is a regular local ring of dimension $1$ with maximal ideal $\m = \{x\in F \ |\ |x| <1 \}$. In particular, $\m$ is principal and we fix for convenience a generator $\pi$ i.e. $(\pi) = \m$. The quotient $\FO/\m$ is called the \textit{residue field} and is isomorphic to a finite field $\BF_q$. The characteristic of $\BF_q$ is called the \textit{residue characteristic of $F$}. More generally we have for any $\alpha \in \BZ_{\geq 0}$ a finite ring $\FO_\alpha = \FO / \m^\alpha$.

 The units $\FO^\times$ of $\FO$ are the complement of $\m$ i.e $\FO^\times = \{ x \in \FO \ |\ |x|=1\}$. Then every $x\in F^\times$ can be written uniquely as $x = u\pi^\alpha$ for some $u \in \FO^\times$ and $\alpha \in \BZ$. As $|\pi| = q^{-1}$ we see that the image of the norm map $|\cdot |: F \rightarrow \BQ$ is $q^\BZ \cup \{0\}$.

%Next we discuss a version of \todo{find a reference or prove it!} Hensel's lemma.

%\begin{prop}\label{hensel} Let $X$ be a smooth variety over $\spec(\FO)$ of relative dimension $d$. Then for every $\alpha \geq 1$ the projection $X(\FO/\pi^\alpha \FO) \rightarrow X(\BF_q)$ is surjective and every fiber has cardinality $q^{d(\alpha-1)}$.
%\end{prop}

We will need three more pieces of data naturally associated with $F$. The first is a natural section of the projection $\FO \rightarrow \BF_q$ called the \textit{Teichm\"uller lift} $\sigma: \BF_q \rightarrow \FO$, which is characterized by $\sigma(0) =0$ and $\sigma_{|\BF_q^\times}$ being multiplicative. Every $x \in \FO$ then has a unique presentation as convergent power series
\[ x = \sum_{\alpha=0}^\infty  \sigma(x_\alpha) \pi^\alpha, \ \ \ x_\alpha \in \BF_q.\] 

Hence we can speak of the coefficient of $x$ at $\pi^\alpha$ as an element in $\BF_q$. In particular, we have decompositions 
\begin{equation}\label{ofdecomp}\FO = \bigsqcup_{\zeta \in \BF_q} \sigma(\zeta) + \m, \ \ \ \ \FO^\times = \bigsqcup_{\zeta \in \BF_q^\times} \sigma(\zeta) + \m. \end{equation}

Next, since $(F,+)$ is locally compact, we can consider the Haar measure $\mathscr{H}$ on $F$ and more generally $\h_n$ on $F^n$ for any $n\geq 1$, normalized by $\h_n(\FO^n) =1$ . For any measurable subset $A \subset F^n, x \in F$ and $y \in F^n$ we have 
\begin{equation}\label{scaling} \h_n(xA) = |x|^n \h_n(A), \ \ \ \ \h_n(y + A) = \h_n(A).\end{equation}
%From (\ref{ofdecomp}) we can deduce for example $\h(\m) = \frac{1}{q}\h\left(  \bigsqcup_{\zeta \in \BF_q} \sigma(\zeta) + \m\right) = \frac{1}{q}$.

All integrals we consider will be with respect to this Haar measure, and we will in general only indicate the variable, over which we integrate. The following lemma is an easy consequence of (\ref{scaling}) and we will use it many times without explicitly mentioning.
\begin{lemma}\label{everyp} For any $\alpha \in \BZ$ we have 
\[ \h\left(\pi^\alpha\FO \right) = \h\left(\{ x \in \FO \ |\ |x| \leq q^{-\alpha}\} \right) = q^{-\alpha},\]
\[\h\left(\pi^\alpha\FO^\times\right) = \h\left(\{ x \in \FO \ |\ |x| = q^{-\alpha}\} \right) = q^{-\alpha}\left(1-q^{-1}\right).\]
\end{lemma}
\begin{proof} Both equations follow from (\ref{scaling}), for the second one we also use $ \FO^\times = \FO \setminus \m$.
\end{proof}

Finally, we fix a non-trivial additive character $\Psi:F \rightarrow \BC^\times$ normalized by the condition $\ker(\Psi) = \m$. As in the finite field case, the integrals we compute will not depend on the actual choice of $\Psi$.

\begin{comment}
\begin{lemma} For any $\alpha \in \BZ$ we have 
\[\h\left(\pi^\alpha\FO^\times\right) = \h\left(\{ x \in \FO \ |\ |x| = q^{-\alpha}\} \right) = q^{-\alpha}\left(1-q^{-1}\right)\]
 and \todo{more precise about lifts} furthermore
\[ \int_{\pi^\alpha\FO^\times} \Psi(x)dx =  \begin{cases} q^{-\alpha}(1-q^{-1}) &\text{if } \alpha > 0\\
																					-q^{-1} &\text{if } \alpha=0\\
																					0 &\text{if } \alpha < 0.
\end{cases}\]
\end{lemma}
\begin{proof} By the scaling property (\ref{scaling}) we have $\h\left(\pi^\alpha\FO^\times\right) = q^{-\alpha} \h\left(\FO^\times\right)$ and also $\h(\FO^\times) = \h(\FO) - \h(\m) = 1-q^{-1}$, which proves the first formula. This also proves the $\alpha > 0$ case for the second formula, as $\ker(\Psi) = \m$. For $\alpha \leq 0$ we simply notice that $\Psi$ descends to a non-trivial additive character $\widetilde{\Psi}$ on the finite abelian group $\pi^\alpha\FO / \m$ and hence 
\[ \int_{\pi^\alpha\FO} \Psi(x)dx =q^{-1} \sum_{\tilde{x} \in \pi^\alpha\FO / \m} \widetilde{\Psi}(\tilde{x}) = 0.\] 
The lemma now follows by writing $\pi^\alpha\FO^\times = \pi^\alpha\FO \setminus \pi^{\alpha+1}\FO$.
\end{proof}
\end{comment}

We finish this section by introducing the Fourier transform on $F^n$. Write $\mathcal{S}(F^n)$ for the $\BC$-vector space of locally constant, complex valued functions with compact support on $F^n$. The \textit{Fourier transform} $\FF(f)$ of a function $f \in \mathcal{S}(F^n)$ is a function on $F^n$ defined by
\[ \FF(f)(y) = \int_{F^n} f(x) \Psi(\lla y,x\rra)dx \text{ for all } y \in F^n,\]
where $\lla,\rra:F^n \times F^n \rightarrow F$ denotes the standard inner product.

\begin{lemma}\cite[Lemma 8.1.3]{Ig00}\label{localfin} The Fourier transform defines a linear isomorphism $\FF: \mathcal{S}(F^n) \rightarrow \mathcal{S}(F^n)$ and satisfies 
\[ \FF(\FF(f))(x) = q^{-n}f(-x)\]
 for all $f\in\mathcal{S}$ and $x\in F^n$.
\end{lemma}

For any subset  $A \subset F^n$ we write  $\chi_A$ for the characteristic function of $A$. In practice all the Fourier transforms we need can be computed from the 
 following lemma using the linearity of $\FF$.

\begin{lemma}\label{lmbl} For every $\alpha \in \BZ$ and  $y \in F^n$ we have 
\[ \FF(\chi_{\pi^\alpha \FO^n})(y) =  \int_{\pi^\alpha \mathcal{O}^n} \Psi(\langle y,x\rangle) dx = q^{-\alpha n}\chi_{\pi^{-\alpha+1}\FO^n}(y).\]
\end{lemma}
\begin{proof} Since $\Psi$ is a character we can write
\[ \int_{\pi^\alpha\mathcal{O}^n} \Psi(\langle y,x\rangle)dx = \prod_{i=1}^n \int_{\pi^\alpha\FO} \Psi(y_ix_i)dx_i.\] 
If for some $1\leq i\leq n$ we have $|y_i| = q^{\beta}$ with $\beta \geq \alpha$, then $x_i \mapsto \Psi(y_ix_i)$ will descend to a non-trivial character on the finite abelian group $\pi^\alpha \FO/ \pi^{\beta +1} \FO$ and hence $\int_{\FO} \Psi(x_iy_i)dy_i=0$. This implies the lemma.
\end{proof}

%Notice that, as opposed to Definition \ref{nft}, we fixed here a basis of the base vector space $F^n$. This way have a distinguished free $\FO$-module of rank $n$ inside $F^n$ i.e. $\FO^n \subset F^n$. In what follows we will sometime consider $\FO$-points of an affine space over $\spec(\FO)$, without explicitly giving a basis, nevertheless 

\subsection{Local Igusa zeta funcions}\label{liz}

In this section we fix a polynomial map $f:\BA^n \rightarrow \BA^m$ given by $f_1,\dots,f_m \in \FO[x_1,\dots,x_n]$. As mentioned already, we are interested in the push-forward measure $f_*(\h_n)$ and how it compares to $\h_m$. As it turns out this is described by some interesting arithmetics of $f$. For $\alpha \geq 0$ denote by $B_{f,\alpha}$ the number of solutions to $ f \equiv 0$ modulo $\pi^\alpha$ i.e. 
\[B_{f,\alpha} = |\{ x \in \FO^n_\alpha \ |\ f(x) = 0\}|.\]
Looking at the projection $\FO^n \rightarrow \FO^n_\alpha$ for every $\alpha \geq 0$, we see 
\begin{equation}\label{projct} f_*(\h_n)\left(\pi^\alpha \FO^m\right)= \h_n\left(\{x \in  \FO^n \ |\ ||f(x)||\leq q^{-\alpha}\}\right) = B_{f,\alpha}q^{-n\alpha}. \end{equation}
In particular, the "quotient" $f_*(\h_n)/\h_m$ at the origin is given by 

\[ B_f  = \lim_{\alpha\rightarrow \infty} \frac{f_*(\h_n)\left(\pi^\alpha \FO^m \right)}{\h_m(\pi^\alpha \FO^m)}  = \lim_{\alpha \rightarrow \infty} q^{-\alpha(n-m)}B_{f,\alpha}\]

We think of $B_f$ as the \textit{relative volume} of $f^{-1}(0)$, which is in general a singular variety. If $f^{-1}(0)$ is smooth then it follows from a theorem of Weil \cite[Theorem 2.2.5]{WE12}, that $B_f = q^{-\dim f^{-1}(0)} |f^{-1}(0)(\BF_q)|$. In the singular case $B_f$ seems to have some interesting properties, as we try to illustrate with the following example.

\begin{expl}\label{m1m} Take $f:\BA^{2m+2} \rightarrow \BA^m$ to be given for all $x,y \in \FO^{m+1}$ by 
\[ f(x,y) = (x_1y_1-x_2y_2, x_2y_2 -x_3y_3,\dots,x_my_m - x_{m+1}y_{m+1}). \]
Then it is not hard to see, that we have
\begin{align}\label{bfas} B_{f,\alpha} = \sum_{\lambda \in \FO_\alpha} |\{(z,w) \in \FO^2_\alpha  \ |\ zw=\lambda\}|^{m+1}. \end{align}
A direct computation then shows
\begin{equation}\label{fooo} |\{(z,w) \in \FO^2_\alpha \ |\ zw=\lambda\}| = \begin{cases} (\beta+1)(q-1)q^{\alpha-1} \ \text{ if } |\lambda| = q^{-\beta} \neq 0 \\
             (\alpha+1)q^\alpha - \alpha q^{\alpha-1} \ \text{ if } \lambda = 0. \end{cases} \end{equation}
As it turns out the formula for $B_{f,\alpha}$ will not be particularly nice, however the limit $B_f = \lim_{\alpha \rightarrow \infty} q^{-\alpha(m+2)}B_{f,\alpha}$ seems to be much better behaved. First the $\lambda = 0$ term in (\ref{bfas}) goes to zero and we get
\begin{align*} B_{f} &= \lim_{\alpha \rightarrow \infty} q^{-\alpha(m+2)} \sum_{\beta = 0}^{\alpha-1} (q-1)q^{\alpha-\beta-1} (\beta+1)^{m+1}(q-1)^{m+1} q^{(\alpha-1)(m+1)}\\
&= (1-q^{-1})^{m+2} \sum_{\beta=0}^{\infty} (\beta+1)^{m+1} q^{-\beta} = q^{-m} E_{m+1}(q), \end{align*}
where $E_n$ denotes the $n$-th Eulerian polynomial \cite{Pe15}. These polynomials appear in many places, notably as the Poincar\'e polynomials of toric varieties associated with the permutahedra. In particular, they are palindromic and have positive coefficients. The first few are given by 
\begin{align*} &E_1(t) = 1, \ \ \ E_2(t) = t+1, \ \ \ E_3(t) = t^2 + 4t + 1, \\
&E_4(t) = t^3 + 11t^2+ 11t+ 1, \ \ \ E_5(t) = t^4+26 t^3 + 66 t^2 + 26t+1. \end{align*}						
\end{expl}

From the definition of $B_f$ it is not at all clear, that $B_f < \infty$ and indeed in general it will not be. In fact this already fails in the following simple example.

\begin{expl}\label{divex} Consider the multiplication map $f:\BA^2 \rightarrow \BA^1$ given by $f(x,y) = xy$. From (\ref{fooo}) we see that the number $B_{f,\alpha}$ of solution to $xy=0$ in $\FO^2_\alpha$ is given by $(\alpha+1)q^\alpha - \alpha q^{\alpha-1}$ and hence
\[ B_f = \lim_{\alpha \rightarrow \infty} q^{-\alpha}B_{f,\alpha} = \lim_{\alpha \rightarrow \infty} (\alpha+1) - q^{-1}\alpha = \infty.\]
\end{expl}

A both interesting and convenient way to deal with this problems is to introduce an extra complex variable. The resulting object is called the \textit{local Igusa zeta function} associated with $f$ and is defined as
\[ \FI_f(s) = \int_{\FO^n} ||f(x)||^s dx,\]
where $|| (y_1,\dots,y_m) ||= \max(|y_1|,\dots,|y_m|)$ and $s\in \BC$ with real part greater than $0$. Because of the non-Archimedian norm $\FI_f(s)$ depends only on the ideal $(f_1,\dots,f_m) \subset \FO[x_1,\dots,x_n]$ hence one can think of $\FI_f(s)$ as being associated to the variety $f^{-1}(0)$. 

The relation with the push-forward measure comes from the almost tautological formula \cite[Section 3.6]{Bo07}
\begin{equation}\label{tautp}      \int_{\FO^n} ||f(x)||^s dx   = \int_{\FO^m} ||y||^s df_*(\h_n)   \end{equation}

Maybe not so surprisingly $\FI_f(s)$ is also closely related to the $B_{f,\alpha}$'s. Namely we define the \textit{Poincar\'e series} of $f$ by
\[ P_f(t) = \sum_{\alpha \geq 0} B_{f,\alpha}q^{-n\alpha} t^\alpha.\]
Then using again (\ref{projct}) we see

\begin{align} \nonumber \FI_f(s) &= \int_{\FO^n} ||f(x)||^s dx = \sum_{\alpha\geq 0} q^{-\alpha s} \h_n\left(\{x\ |\ ||f(x)||=q^{-\alpha}\}\right)\\
																			\nonumber		&= \sum_{\alpha\geq 0} q^{-\alpha s}\left( \h_n\left(\{x\ |\ ||f(x)||\leq q^{-\alpha}\}\right)- \h_n\left(\{x\ |\ ||f(x)||\leq q^{-\alpha-1}\}\right)\right)\\
																				\label{idk}	&= \sum_{\alpha\geq 0} q^{-\alpha s}\left( B_{f,\alpha} q^{-\alpha n} - B_{f,\alpha+1} q^{-(\alpha+1)n}\right)\\
																			\label{igupo}		&= (P_f(q^{-s})-1)(1-q^s) + 1.
\end{align}
%In particular rationality of $\FI_f(s)$ implies rationality of $P_f(t)$, which was conjectured by \todo{ref} and proved by Igusa if $\Char(F) = 0$ through Theorem \ref{igrat}.

The real advantage in introducing $\FI_f(s)$ comes from its definition as an integral, which makes it possible to use analytic and geometric methods to study it. We quickly explain some basic structure results on $\FI_f(s)$, assuming for the rest of this section that $\Char(F)=0$.

The key Idea of Igusa \cite{Ig74} was to use an embedded resolution of $f$ to prove the following theorem, originally in the case $m=1$ (for the multivariate case see \cite{Lo89}).

\begin{thm}\label{igrat}  The Igusa zeta function $\FI_f(s)$ is a rational function in $q^{-s}$. The real parts of its poles in $\BC$, as a function of $s$, are negative rational numbers.
\end{thm}

The numerical data of a resolution of $f$ will give a set of possible poles for $\FI_f(s)$ which is in general larger that the actual set of poles. The description of the actual poles is an intriguing open problem and the content of various monodromy conjectures (see \cite{De90} for a survey). This is the reason we will spend some time on the description of the poles of $\FI_f(s)$ in the cases we can compute, see Section \ref{poles}.

The same numerical data also describe the asymptotic behavior of $B_{f,\alpha}$ as $\alpha \rightarrow \infty$ \cite{VZ08}. For us however the following much simpler criterion will do.

\begin{lemma}\label{fresh} Assume that the largest poles of $\FI_f(s)$ is at $s=-m$. Then $B_f$ is finite if and only if $s=-m$ is a simple pole, in which case we have 
\[ B_f = \frac{q^m}{q^m-1} \Res_{s=-m}\FI_f(s)=  \frac{q^m(q^{s+m}-1)}{q^m-1}\FI_f(s)_{|s=-m}.\]
\end{lemma} 

\begin{proof} Consider the generating series 
\begin{align*} (q^{s+m}-1) P_f(q^{-s})& = \sum_{\alpha \geq 0} (q^{s+m}-1) B_{f,\alpha} q^{-\alpha(s+n)} \\
&= q^{s+m} + \sum_{\alpha \geq 0} q^{-\alpha(s+n)} \left( q^{m-n}B_{f,\alpha+1}- B_{f,\alpha}\right).
\end{align*}
If $s=-m$ is a simple pole, then it follows from our assumptions and (\ref{igupo}), that $(q^{s+m}-1) P_f(q^{-s})$ converges for $|s| \leq m$. 
We can then compute the value at $s=-m$ as follows:
\begin{align*} (q^{s+m}-1) P_f(q^{-s})_{s=-m} &= \lim_{N\rightarrow \infty } 1 + \sum_{\alpha =  0}^N q^{-\alpha(n-m)} \left( q^{m-n}B_{f,\alpha+1}- B_{f,\alpha}\right) \\
&= \lim_{N \rightarrow \infty} q^{-(N+1)(n-m)}B_{f,N+1} = B_f.\end{align*}
If $s=-m$ is a higher order pole, a similar argument shows that $B_f$ diverges.
\end{proof}

\begin{expl} Continuing Example \ref{divex} we can use the formula $B_{f,\alpha} = (\alpha+1)q^\alpha - \alpha q^{\alpha-1}$ to compute $\FI_f(s)$ via (\ref{idk}):
\begin{align*} \FI_f(s) &= \sum_{\alpha \geq 0} q^{-\alpha s} \left( \left((\alpha+1) - \alpha q^{-1} \right) q^{- \alpha} - \left((\alpha+2) - (\alpha+1)q^{-1}\right)q^{-(\alpha+1)} \right)\\
&= (1-q^{-1})^2 \sum_{\alpha \geq 0} (\alpha+1)q^{-\alpha(s+1)} = \frac{(1-q^{-1})^2}{(1-q^{-(s+1)})^2}.
\end{align*}
In particular, $\FI_f(s)$ has a poles of order $2$ at $s=-1$ which by Lemma \ref{fresh} explains the divergence of $B_f$.
\end{expl}

Finally, we discuss an analogue of the functional equation satisfied by the Weil zeta function. For this consider a homogeneous polynomial $f \in \FO[x_1,\dots,x_n]$. We denote for every $e \geq 1$ by $F^{(e)}$ the unique unramified extension of $F$ of degree $e$ and by $I^{(e)}_f(s)$ the Igusa zeta function of $f$ computed over $F^{(e)}$. We call $\FI_f(s)$ \textit{universal over $F$} if there exists $\mathcal{I}_f(u,v) \in \BQ(u,v)$ such that for every $e \geq 1$ we have $\FI^{(e)}_f(s) = \mathcal{I}_f(q^{-es},q^{-e})$. Under these assumptions the following is a simplified version of a theorem of Denef and Meuser.

\begin{thm}\cite{DM91}\label{feqi} For almost all residue characteristics, if $\FI_f(s)$ is universal over $F$ it satisfies the functional equation 
\[ \FI_f(u^{-1},v^{-1}) = u^{\deg f} \FI_f(u,v).\]
\end{thm}

\begin{rem}\label{mulv} The proof uses an embedded resolution of the projective hypersurface $\{f=0\} \subset \BP^{n-1}$ and the functional equation of the Weil zeta function of the exceptional divisors, which are now projective as well. The same argument should hence also prove a version of Theorem \ref{feqi} when $f=(f_1,\dots,f_m)$ is a collection of homogeneous polynomials of the same degree.
\end{rem}

\subsection{Push forward measures of moment maps}\label{pfmmm}

In this section we derive an analogue of Proposition \ref{prop1} over local fields in the following general set up. Let $\varrho: \FO^m \rightarrow \mathfrak{gl}_n(\FO)$ be an $\FO$-linear map, where $\mathfrak{gl}_n(\FO)$ denotes the $\FO$-module of $n\times n$-matrices with entries in $\FO$. Define a 'moment map' $\mu: \FO^n \times \FO^n \rightarrow \FO^m$ by the equation
\[ \lla \mu(x,y), z \rra = \lla \varrho(z)x,y \rra,\]
for all $x,y \in \FO^n$ and $z \in \FO^m$. Notice that the standard paring $\lla,\rra$ induces an isomorphism $\Hom_\FO(\FO^n,\FO) \cong \FO^n$ and hence $\mu$ is uniquely defined this way.

Finally, we define a function $a_\varrho: F^m \rightarrow \BR_{\geq 0}$ by 
\begin{equation}\label{aro} a_\varrho(z) = \h_n(\varrho(z)\FO^n \cap \m^n) = \int_{\FO^n} \chi_{\m^n}\left(\varrho(z)x\right) dx, \end{equation}
for $z \in F^n$. Here we also use the notation $\chi_A$ for the characteristic function of $A \subset F^n$.

\begin{prop}\label{ppf} The Fourier transform of the push-forward measure $\mu_*(\h_{2n})$ is given by $a_\varrho$. More precisely we have for every compact measurable $A \subset \FO^m$ 
\[ \mu_*(\h_{2n})(A) = q^m\int_{F^m} \FF(\chi_A)(z) a_\varrho(z) dz.\]
\end{prop}

\begin{proof}  By definition of the push-forward measure we can write 
\begin{equation}\label{moom} \mu_*(\h_{2n})(A) = \h_{2n}(\mu^{-1}(A)) = \int_{\FO^n \times \FO^n} \chi_A(\mu(x,y)) dxdy.\end{equation}
For fixed $x,y \in \FO^n$ we have by the Fourier inversion Lemma \ref{localfin} 
\begin{align*} \chi_A(\mu(x,y)) &= q^m\int_{F^m} \FF(\chi_A)(z)\Psi(- \lla \mu(x,y),z \rra)dz \\
&= q^m\int_{F^m} \FF(\chi_A)(z)\Psi(- \lla \varrho(z)x,y \rra)dz.\end{align*}
By integrating along $y \in \FO^n$ we thus have by lemma \ref{lmbl} 
\begin{align*} \int_{\FO^n} \chi_A(\mu(x,y)) dy &=  q^m\int_{F^m} \FF(\chi_A)(z) \left(\int_{\FO^n}\Psi(- \lla \varrho(z)x,y \rra) dy\right) dz \\ 
&=q^m \int_{F^m}\FF(\chi_A)(z) \chi_{\m^n}(\varrho(z)x) dz.\end{align*}
Plugging this into (\ref{moom}) we obtain the proposition:
\begin{align*} \mu_*(\h_{2n})(A) &= q^m\int_{\FO^n} \int_{F^m} \FF(\chi_A)(z) \chi_{\m^n}(\varrho(z)x) dzdx \\
&= q^m\int_{F^m} \FF(\chi_A)(z) a_\varrho(z) dz. %= q^m\int_{A} \left(\int_{F^m} a_\varrho(z) \Psi(\lla a,z \rra)dz\right) da
\end{align*}
\end{proof}

\begin{rem} Ideally we would like to phrase Proposition \ref{ppf} differently by writing
	\begin{align*} \mu_*(\h_{2n})(A)& = q^m\int_{F^m} \FF(\chi_A)(z) a_\varrho(z) dz \\
 &= q^m\int_{F^m} \int_A \Psi(\lla a,z\rra)a_\varrho(z) dadz = q^m \int_A \FF(a_\varrho)(a) da. \end{align*} 
Then we could say, that the density of $\mu_*(\h_{2n})$ with respect to $\h_m$ is given by the Fourier transform of $a_\varrho$, which would be the exact analogue of Proposition \ref{prop1}. However $a_\varrho$ will in general not be integrable (Example \ref{divex} gives such a case), and hence we cannot interchange the integration over $F^m$ and $A$ in general. 
\end{rem}

%Of course the set up we consider here will arise for us as in Section \ref{genred} i.e. we start with a representation of a reductive group $\rho: G \rightarrow \GL_n$ over $\spec(\FO)$ and $\varrho: \g \rightarrow \mathfrak{gl}_n$ is its derivative. Then $\mu$ will be an honest moment map and Proposition \ref{ppf} can be seen as an analogue of Proposition \ref{prop1}.

\begin{expl}
An interesting special case of Proposition \ref{ppf} comes from taking $A= \sigma(\zeta) + \m^m$ where $\zeta = (\zeta_1,\zeta_2,\dots,\zeta_m) \in \BF_q^m$. Then by definition of the push-forward measure and similar as in (\ref{projct}) we have (the $\widetilde{\cdot}$ will always denote the reduction to the residue field $\BF_q$).
\[ \mu_*(\h_{2n})(A) = \h_{2n}(\mu^{-1}(A)) = q^{-2n}|\widetilde{\mu}^{-1}(\zeta)(\BF_q)|.\]
On the other hand we have essentially by Lemma \ref{lmbl} 
\[\FF(\chi_A)(z) = q^{-m}\chi_{\FO^m}(z) \Psi(\lla z,\sigma(\zeta)\rra)\]
 and thus Proposition \ref{ppf} reads 

\[ \mu_*(\h_{2n})(A) = \int_{\FO^m} \Psi(\lla z,\sigma(\zeta)\rra) a_\varrho(z) \]

%\[ \int_{A} \left(\int_{F^m} a_\varrho(z) \Psi(\lla a,z \rra)dz\right) da = \int_{F^m} a_\varrho(z) \Psi(\lla \sigma(\zeta),z \rra)\int_{\m^m} \Psi(\lla a,z \rra) dadz. \]

Notice now, that the integrand is invariant on cosets of $\m^m$ and by (\ref{projct}) we have the formula $a_\varrho(z) = q^{-n}|\ker(\widetilde{\varrho})(\widetilde{z})(\BF_q)|$. Hence we recover \cite[Proposition 1]{Hau1}:
\[   |\widetilde{\mu}^{-1}(\zeta)(\BF_q)| = q^{n-m}\sum_{\widetilde{z} \in \BF_q^m} |\ker(\widetilde{\varrho})(\widetilde{z})(\BF_q)| \widetilde{\Psi}(\lla \zeta,\widetilde{z}\rra).\]
\end{expl}

As we already mentioned in (\ref{tautp}), the Igusa zeta function $\FI_\mu(s)$ is closely related to the push-forward measure $\mu_*(\h_{2n})$ and hence Proposition \ref{ppf} implies the following general formula.

\begin{cor}\label{crucro}
\[\FI_{\mu}(s) = \frac{q^{(s+m)}-q^s}{q^{(s+m)}-1} + \frac{q^{m}-q^{s+m}}{q^{s+m}-1} \int_{F^m\setminus \m^m} a_\varrho(z) ||z||^{-(s+m)}dz\]
\end{cor}
\begin{proof} By the tautological equation (\ref{tautp}) we have
\[\FI_\mu(s) =  \int_{\FO^m} ||w||^s \mu_*(\h_{2n})= \sum_{\alpha \geq 0} q^{-\alpha s}\mu_*(\h_{2n})\left(\pi^\alpha \FO^m \setminus \pi^{\alpha +1}\FO^m \right) .\]
With Proposition \ref{ppf} and Lemma \ref{lmbl} this becomes
\begin{align*} \FI_\mu(s) &= q^m \sum_{\alpha \geq 0} q^{-\alpha s} \left(\int_{F^m} \FF(\chi_{\pi^\alpha \FO^m})(z) a_\varrho(z) dz - \int_{F^m} \FF(\chi_{\pi^{\alpha+1} \FO^m})(z) a_\varrho(z) dz\right)  \\
&= q^m \sum_{\alpha \geq 0} q^{-\alpha (s+m)}\left( \int_{\pi^{-\alpha+1}\FO^m} a_\varrho(z)dz  - q^{-m} \int_{\pi^{-\alpha} \FO^m} a_\varrho(z)dz \right) \\
&= 1 + (q^{-s}-1) \sum_{\alpha \geq 0} q^{-\alpha (s+m)} \int_{\pi^{-\alpha}\FO^m} a_\varrho(z) dz.
    \end{align*}
Now the domain of each integral in each summand contains $\m^m$. Using  $a_{\varrho|\m^n} \equiv 1$ we can evaluate the integrals over $\m^m$ first and get
\[ 	\sum_{\alpha \geq 0} q^{-\alpha (s+m)} \int_{\m^m} a_\varrho(z) dz = q^{-m}\sum_{\alpha \geq 0} q^{-\alpha (s+m)} = \frac{q^{-m}}{1-q^{-(s+m)}} \]
and then 
\[ \FI_{\mu}(s) = \frac{q^{s+m}-q^s }{q^{s+m}-1}+ (q^{-s}-1) \sum_{\alpha \geq 0} q^{-\alpha (s+m)} \int_{\pi^{-\alpha}\FO^m \setminus \m^m} a_\varrho(z) dz.\]
Finally, we can reorder the domains of integration according to the norm of $z$ and obtain
\begin{align*} \sum_{\alpha \geq 0} q^{-\alpha (s+m)} &\int_{\pi^{-\alpha}\FO^m \setminus \m^m} a_\varrho(z) dz = \sum_{\beta \geq 0} \left(\int_{\pi^{-\beta}\FO^m \setminus \pi^{-\beta+1} \FO^m } a_\varrho(z)dz\right) \sum_{\alpha \geq \beta} q^{-\alpha(s+m)}\\
&= \frac{q^{s+m}}{q^{s+m}-1}  \sum_{\beta \geq 0} \left(\int_{\pi^{-\beta}\FO^m \setminus \pi^{-\beta+1} \FO^m } a_\varrho(z)dz\right) q^{-\beta(s+m)}\\
&= \frac{q^{s+m}}{q^{s+m}-1} \int_{F^m \setminus \m^m} a_\varrho(z) ||z||^{-(s+m)}dz.
\end{align*}
\end{proof}

\begin{expl}\label{npts} Consider the case $m=1$ and $\varrho: \FO \rightarrow \mathfrak{gl}_n(\FO)$ given by $\varrho(z) = z \unt_n$. Then the moment map $\mu$ is given by $\mu(x,y) = \sum_{i=1}^n x_iy_i$ for all $x,y \in \FO^n$ and $\FI_\mu(s)$ is now straightforward to compute:

First we have for any $z \in F\setminus \m$ and $x\in \FO^n$ the equivalence $\varrho(z)x \in \m^n \Leftrightarrow |x_i| \leq (q|z|)^{-1} \ \forall i$ and hence $a_\varrho(z) = (q|z|)^{-n}$. Then 
\begin{align*} \int_{F \setminus \m} a_\varrho(z)|z|^{-(s+1)}dz &= q^{-n} \int_{F \setminus \m} |z|^{-(s+n+1)}dz \\
&= q^{-n} \sum_{\alpha \geq 0} q^{-\alpha(s+n+1)} q^{\alpha}(1-q^{-1}) = \frac{q^{s}(1-q^{-1})}{q^{s+n}-1}.
\end{align*}	
Plugging this into Corollary \ref{crucro} we obtain the formula
\[ \FI_\mu(s) = \frac{(q-1)(q^n-1)q^{2s}}{(q^{s+1}-1)(q^{s+n}-1)}.\]

This example appears already in \cite{Ig00} and was part of the initial motivation for looking at zeta functions of moment maps. 
\end{expl}

\subsection{Hypertoric zeta functions}\label{hzf}

In this section we determine the Igusa zeta function for the moment maps that appear in the construction of hypertoric varieties \ref{htv}. Let $\FA$ be a central hyperplane arrangement of rank $m$, where the hyperplanes $H_1,\dots,H_n$ are given by normal vectors $a_1,\dots,a_n \in \BZ^m$. From now on we will work under the following assumption:

\begin{ass}\label{assum} If $A \in M_{n\times m}(\BZ)$ denotes matrix whose rows are the $a_i$, we will always assume that the residue characteristic of $F$ is larger than any of the minors of $A$.
\end{ass}

As a first consequence of this assumption we notice that the intersection lattice $L(\FA)$ does not depend on whether we consider $\FA$ over $F$ or over $\BF_q$.

Associated to $\FA$ we have a moment map $\mu: \FO^n \times \FO^n \rightarrow \FO^m$ defined by (\ref{toricmm}) i.e.
\[ \mu(x,y) =  \sum_{i=1}^n x_iy_ia_i \text{ for all } x,y\in \FO^n. \]

The Igusa zeta function associated with $\FA$ is then defined as 
\[\FI_\FA(s) = \FI_\mu(s) = \int_{\FO^{2n}} ||\mu(x,y)||^sdxdy.\]

This is really an invariant of $\FA$ and does not depend on the choice of normal vectors as long as \ref{assum} is satisfied. Also the assumption that $\FA$ is essential is not necessary here, as $||\mu||$ depends only on the span of the $a_i$.

In this case $a_\varrho:F^m \rightarrow \BR_{\geq 0}$ defined in (\ref{aro}) takes a rather simple form. Namely using (\ref{toricvar}) we have for $z \in F^m$

\[ a_\varrho(z) = \int_{\mathcal{O}^n}\chi_{\m^n}(\varrho(z)v)dv 
														= \prod_{i=1}^n \int_{\FO} \chi_{\m}(\lla z,a_i\rra v_i)dv_i
														= \prod_{i=1}^n \min\left\{1,\frac{1}{q|\lla z,a_i\rra|}\right\}.\]

The determination of $\FI_\FA(s)$ thus reduces by Corollary \ref{crucro} to computing
\begin{equation}\label{jms} \FJ_\FA(s) = \int_{F^m \setminus \m^m} \prod_{i=1}^n \min\left\{1,\frac{1}{q|\lla z,a_i\rra|}\right\} ||z||^{-(s+m)}dz.\end{equation}

The point is now, that the integrand of $\FJ_\FA(s)$ clearly takes only countably many different values and we can partition $F^m \setminus \m^m$ accordingly. 
Let $\mathbf{c}=(c_1,\dots,c_n) \in \BZ_{\geq 0}^n\setminus 0$ and define
\begin{eqnarray}\label{Z_c} Z_{\FA,\mathbf{c}} = \left\{ z \in F^m\ \middle|\ \ \ |\lla z,a_i\rra| \begin{cases} = q^{c_i-1} &\text{ if } c_i>0\\
																																	\leq q^{-1} &\text{ if } c_i = 0.\end{cases}\right\}.\end{eqnarray}	
																	
With this we can write $F^m \setminus \m^m = \bigsqcup_{\bfc \in \BZ_{\geq 0}^n\setminus 0} Z_{\FA,\bfc}$. Now in general for a matrix $B \in M_{k\times l}(\FO)$, which has full rank when reduced over $\BF_q$, we always have
\begin{equation}\label{fullrk} ||Bv||= ||v|| \text{ for all } v \in \FO^l \end{equation}
by the non-archimedean triangle inequality (\ref{nontri}). In particular, when we apply this to the matrix $A$ from \ref{assum} we get 

\begin{equation}\label{gmax} ||z|| = \max_{1 \leq i \leq n} |\lla z,a_i\rra|.\end{equation}
Then (\ref{jms}) becomes

\begin{eqnarray}\label{bisum} 	\FJ_{\FA}(s)	= \sum_{ \bfc \in \BZ_{\geq 0}^n\setminus 0} q^{-\sum_i c_i -(\ol{\bfc}-1)(s+m)} \h_m(Z_{\FA,\bfc}), \end{eqnarray}			
where we use the notation $\bar{\bfc} = \max \{c_1,\dots,c_n\}$. We will determine $\FJ_{\FA}(s)$ using a recursion for which it will be more convenient to consider 
\[ \FJ'_{\FA}(s) = q^{-(s+m)} \FJ_{\FA}(s) =  \sum_{ \bfc \in \BZ_{\geq 0}^n\setminus 0} q^{-\sum_i c_i -\ol{\bfc}(s+m)} \h_m(Z_{\FA,\bfc}).\]

Now in general there are linear relations between the $a_i$ and hence $Z_{\FA,\mathbf{c}}$ will be empty for many choices of $\mathbf{c}$. Again because of (\ref{nontri}) we need for example 
\[I = \{ i \in \{1,\dots,n\}\ |\ c_i < \overline{\bfc}\}\]
 to be a flat in $\FA$ for $Z_{\FA,\mathbf{c}}$ to be non-empty. Assuming $I \neq \emptyset$ we put $\bfc_I = (c_i)_{i\in I} \in \BZ_{\geq 0}^I$ and $\overline{\bfc}_I= \max_{i\in I}(c_i)$. 

\begin{lemma}\label{volz} The volume of $ Z_{\FA,\bfc}$ is given by
\[  \h_m(Z_{\FA,\bfc}) = \chi_{\FA^I}(q)q^{\rk(\FA^I)(\ol{\bfc}-2)}  \h_{\rk(\FA_I)}(Z_{\FA_I,\bfc_I}),\]
where we set $\h_{\rk(\FA_I)}(Z_{\FA_I,\bfc_I}) = q^{-\rk(\FA_I)}$ if $I = \emptyset$ or $\ol{\bfc}_I = 0$.
\end{lemma}

\begin{proof}  As we already remarked in (\ref{gmax}) we have $||z|| = \max_{1 \leq i \leq n} |\lla z,a_i\rra| = q^{\ol{\bfc}-1}$ for every $z \in Z_{\FA,\mathbf{c}}$, hence in particular $Z_{\FA,\mathbf{c}} \subset \pi^{-\ol{\bfc}+1} \FO^m$. Furthermore for $z \in Z_{\FA,\mathbf{c}}$ and $w \in \m^m$ we have by definition $z+w \in Z_{\FA,\mathbf{c}}$. In other words, if we write
\[ \phi^m: F^m \rightarrow \left(F / \m\right)^m,\]
then every non-empty fiber of $\phi^m_{| Z_{\FA,\mathbf{c}}}$ is a translation of $\m^m$ and thus 
\begin{equation}\label{voldis} \h_m(Z_{\FA,\bfc}) = q^{-m} | Z'_{\FA,\mathbf{c}}|,\end{equation}
where we abbreviate $ Z'_{\FA,\mathbf{c}}=\phi^m(  Z_{\FA,\mathbf{c}})$.
Now assume first $I= \emptyset$ i.e. $c_1=\dots=c_n= \ol{\bfc}$. Then the image of $Z'_{\FA,\mathbf{c}}$ under the projection 
\[(\pi^{-\ol{\bfc}+1} \FO /\m)^m \rightarrow (\pi^{-\ol{\bfc}+1} \FO / \pi^{-\ol{\bfc}+2} \FO)^m \cong \BF_q^m\]
are exactly the points which do not lie in any hyperplane of $\FA \subset \BF_q^m$. Thus by Theorem \ref{stcha} we obtain  $|Z'_{\FA,\mathbf{c}}| = \chi_\FA(q) q^{m(\ol{\bfc}-1)}$ and  $\h_m(Z_{\FA,\bfc}) = \chi_\FA(q) q^{m(\ol{\bfc}-2)}$.

Next assume $I \neq \emptyset$ but $\ol{\bfc}_I = 0$. Then every $z \in  Z'_{\FA,\mathbf{c}}$ satisfies 
\begin{equation}\label{fatvs}\lla z,a_i \rra = 0 \text{ for all } i\in I.\end{equation}

Furthermore by \ref{assum} we can write the set of all $z \in (\pi^{-\ol{\bfc}+1} \FO /\m)^m$ satisfying (\ref{fatvs}) as $(\pi^{-\ol{\bfc}+1} \FO /\m)^{\rk(\FA^I)}$ in a suitable basis. Now the same argument as for $I= \emptyset$ shows 
\[ \h_m(Z_{\FA,\bfc}) = q^{-m}|Z'_{\FA,\mathbf{c}}| = \chi_{\FA^I}(q) q^{\rk(\FA^I)(\ol{\bfc}-1)-m} =\chi_{\FA^I}(q) q^{\rk(\FA^I)(\ol{\bfc}-2)} q^{-\rk(\FA_I)} .\]

The general case  works the same way by considering the image of $Z'_{\FA,\mathbf{c}}$ under the projection $\left(\pi^{-\ol{\bfc}+1} \FO / \m\right)^m \rightarrow \left(\pi^{-\ol{\bfc}+1} \FO / \pi^{-\ol{\bfc}_I+1} \FO\right)^m$.

\end{proof}

\begin{prop}\label{gorec} $\FJ'_{\FA}(s)$ satisfies the recursion
\[ \FJ'_{\FA}(s) = q^{-m}\sum_{\substack{I \in L(\FA)\\ I \neq \infty}} \frac{q^{-\rk(\FA^I)} \chi_{\FA^I}(q)}{\left(q^{s+\delta_I} - 1\right) } \left(q^{\rk(\FA_I)}\FJ'_{\FA_I}\left(s+\delta_I-\rk(\FA_I)\right)  + 1  \right),\]
where we put $\delta_I = n-|I|+ \rk(\FA_I)$ and $\FJ'_{\FA_\emptyset} =0$.
\end{prop}

\begin{proof} Using Lemma \ref{volz} this is now a straightforward computation. Fix a flat $I \in L(\FA)$ with $I \neq \infty = \{1,\dots,n\}$  and set $\BZ^n_I = \left\{ \bfc \in \BZ_{\geq 0}^n \ |\ c_i  < \ol{\bfc} \Leftrightarrow i \in I\right\}$. Then we have
\begin{align*}  \sum_{ \bfc \in \BZ_I^n} &q^{-\sum_i c_i -\ol{\bfc}(s+m)} \h_m(Z_{\FA,\bfc}) = \sum_{\bfc_I \in \BZ^I_{\geq 0}} q^{-\sum_{i\in I} c_i} \sum_{\substack{\ol{\bfc}\in \BZ_\geq 0 \\ \ol{\bfc} > \ol{\bfc}_I}} q^{-\ol{\bfc}(s+n-|I|+m)} \h_m(Z_{\FA,\bfc})\\
&= \chi_{\FA^I}(q) q^{-2\rk(\FA^I)} \sum_{\bfc_I \in \BZ^I_{\geq 0}} q^{-\sum_{i\in I} c_i} \h_{\rk(\FA_I)}(Z_{\FA_I,\bfc_I}) \sum_{\ol{\bfc} > \ol{\bfc}_I} q^{-\ol{\bfc}\left(s+\delta_I \right)}\\
&= \frac{ \chi_{\FA^I}(q)q^{-2\rk(\FA^I)}}{q^{s+\delta_I} -1 } \sum_{\bfc_I \in \BZ^I_{\geq 0}} q^{-\sum_{i\in I} c_i  -\ol{\bfc}_I\left(s+\delta_I \right)} \h_{\rk(\FA_I)}(Z_{\FA_I,\bfc_I})\\
&= \frac{ \chi_{\FA^I}(q)q^{-m-\rk(\FA^I)}}{q^{s+\delta_I} -1 }  \left(q^{\rk(\FA_I)}\FJ'_{\FA_I}\left(s+\delta_I-\rk(\FA_I)\right)  + 1  \right),
\end{align*} 
where we used in the last line our convention $ \h_{\rk(\FA_I)}(Z_{\FA_I,\bfc_I}) = q^{-\rk(\FA_I)}$ for $\bfc_I = 0$. Now the proposition follows by summing over all flats $I\neq \infty$.
\end{proof}

Of course a recursion for $\FJ'_{\FA}(s)$ implies one for $\FI_{\FA}(s)$, which turns out to be a bit cumbersome however. Instead we give an explicit formula by iterating the recursion from Proposition \ref{gorec}. Recall that for any flat $I \in L(\FA)$ we can identify $L(\FA_I)$ with the sublattice of $L(\FA)$ consisting of flats contained in $I$.

\begin{thm}\label{iguthm} The Igusa zeta function of an essential hyperplane arrangement $\FA$ of rank $m$ is given by

\begin{equation}\label{igusaf} \FI_\FA(s) =  \frac{q^m-1}{q^m-q^{-s}} + \frac{1-q^{s}}{1-q^{-(s+m)}} \sum_{\infty = I_0 \supsetneq I_1 \supsetneq \dots \supsetneq I_r} q^{-\rk(\FA^{I_r})} \prod_{i=1 }^r \frac{\chi_{\FA_{I_{i-1}}^{I_{i}}}(q)}{q^{s+\delta_{I_i}}- 1},\end{equation}
where the sum is over all proper chains of flats in $L(\FA)$ of length $r\geq 1$.
\end{thm}

\begin{expl}\label{ncent} When $\FA$ is the central arrangement in $F$ consisting of $n$ times the origin i.e. $a_1=\dots=a_n=1$ we recover Example \ref{npts}. In this case $L(\FA) = \{\emptyset,\infty\}$, and the sum in (\ref{igusaf}) has only one term corresponding to the chain $\infty \supsetneq \emptyset$. 

From this example we also see that $\FI_\FA(s)$ is really an invariant of $\FA$ and not just of $L(\FA)$.
\end{expl}

\begin{expl}\label{m2m} Next we consider the arrangement in $F^2$ defined by the three normal vectors $a_1=(1,0), a_2=(-1,1), a_3=(0,-1)$. The moment map is then given by 
\[      \mu(x,y) = \begin{pmatrix} x_1y_1-x_2y_2 \\ x_2y_2 - x_3y_3  \end{pmatrix}. \]
We have $L(\FA) = \{\emptyset,\{1\},\{2\},\{3\},\infty\}$, and hence there are $4$ chains of length $1$ and $3$ of length $2$. The contribution of each of them to the sum in (\ref{igusaf}) is given by
\begin{align*}
\infty \supsetneq \emptyset&: \ \ \ \frac{(q-1)(q-2)}{q^2(q^{s+2}-1)}\\ 
\infty \supsetneq \{i\}&: \ \ \ \frac{(q-1)}{q(q^{s+2}-1)} \\
\infty \supsetneq \{i\} \supsetneq \emptyset &: \ \ \ \frac{(q-1)^2}{q^2(q^{s+2}-1)^2}.
\end{align*}
Putting this together we get 
\[ \FI_\FA(s) =   \frac{(q - 1)^2q^{s}\left(q^{s}(q^6 + 2q^5 + 2q^4 - 2q^3 ) + 2q^3 - 2q^2 - 2q - 1\right)}{(q^{s+2} -1) (q^{s+3} -1 )^2}.\] 
\end{expl}

We give a final example to illustrate that in particular the numerators of $\FI_\FA(s)$ get complicated very quickly and do not seem to have any interesting structure apart from the symmetry predicted by Theorem \ref{feqi}.

\begin{expl}\label{hugm} Consider the arrangement in $F^3$ given by the six normal vectors
\[ a_1=\begin{pmatrix} 1 \\1 \\ 0 \end{pmatrix},a_2=\begin{pmatrix} 0\\1\\1 \end{pmatrix} ,a_3=\begin{pmatrix} 1\\0\\1 \end{pmatrix} ,a_4=\begin{pmatrix} 1\\-1\\0 \end{pmatrix} ,a_5=\begin{pmatrix} 0\\1\\-1 \end{pmatrix} ,a_6=\begin{pmatrix} -1\\0\\1 \end{pmatrix} .\]
With the help of Sage we find that the denominator of $\FI_\FA(s)$ is given by
\[ ( q^{s+3} -1) ( q^{s+5} -1)( q^{s+6} -1)^{3}.  \]
The numerator is the following:
\begin{align*}(q-1)^2q^{2s} \Big[ q^{3s}( q^{24} + 2 q^{23} + 3 q^{22} + 3 q^{21} + 3 q^{20}- q^{19} - 11 q^{18}+6 q^{17})\\
+ q^{2s}( 3 q^{19}   + 9 q^{18}   - 12 q^{17}   - 9 q^{16}  - 9 q^{15}  - 9 q^{14}  - 3 q^{13} + 9 q^{12} + 3 q^{11}  )\\
+q^s( - 3 q^{13}  - 9 q^{12}   + 3 q^{11}  + 9 q^{10}  + 9 q^{9}  + 9 q^{8}  + 12 q^{7}  - 9 q^{6}  - 3 q^{5} )\\
- 6 q^{7}  + 11 q^{6}  +  q^{5}  - 3 q^{4}  - 3 q^{3}  - 3 q^{2}  - 2 q  - 1\Big]\end{align*}
\end{expl}

\subsection{Poles}\label{poles}

Formula (\ref{igusaf}) shows that $I_\FA(s)$ is a rational function in $q^{-s}$. We start by studying the poles of $I_\FA(s)$ (as a function in $s$), more precisely the real parts of those poles. 

Then if follows from our formula that the poles of $I_\FA(s)$ are amongst the negative integers $\mathscr{P}_\FA=\left\{ -\delta_I \ | \ I \in L(\FA) \right\} $, as opposed to just negative rational numbers as in Theorem \ref{igrat}. 

Of course for a given $\epsilon \in \p_\FA$ the different summands in (\ref{igusaf}) with a $(q^{s-\epsilon}-1)^{-1}$-term might cancel. We explain now a criterion which will show in many cases, that there is no such cancellation. Consider the subset
\[L(\FA)_\epsilon = \{ I \in L(\FA) \ |\ -\delta_I=\epsilon\},\]
with the induced ordering i.e. $I \geq I'$ iff $I \supseteq I'$. The following observation gives some control over $L(\FA)_\epsilon$. 

\begin{lemma}\label{alstructure}\begin{enumerate} \item[(i)] For any two flats $I \supset 	I'$ in $L(\FA)$ we have $\delta_I \leq \delta_{I'}$ with equality if and only if $\{ a_i \ |\ i \in I \setminus I'\}$ are linearly independent in $F^m / H_I$.

In particular, if $I,	I' \in L(\FA)_\epsilon$, any $J\in L(\FA)$ with $I\supset J \supset I'$ is in $L(\FA)_\epsilon$
% \textit{interval} $[I,I'] = \{ J \in L(\FA)\ |\ I \supset J \supset I'\}$ is contained in $L(\FA)_\epsilon$. 
\item[(ii)] There are unique flats $J_1,\dots,J_d \in L(\FA)_\epsilon$ such that every $I\in L(\FA)_\epsilon$ contains exactly one of the $J_i$'s

 %As a poset $L(\FA)_\epsilon$ is a disjoint union of intervals i.e.
%\[ L(\FA)_\epsilon = \bigsqcup_{j} [I_j,I'_j],\]
% where $I_j \supset I'_j$ are flats in $L(\FA)_\epsilon$.
\end{enumerate}
\end{lemma} 
\begin{proof} For $(i)$ we can write
\[ \delta_{I'} - \delta_{I} = |I|-|I'| + \rk(\FA_{I'}) - \rk(\FA_{I}) = |I\setminus I'| - \rk(\FA_{I}^{I'}).\]

Now the normal vectors defining $\FA_{I'}^{I}$ are exactly the images of $\{ a_i \ |\ i \in I \setminus I'\}$ in $F^m /H_{I}$ and $(i)$ follows.

Statement $(ii)$ follows from the observation that given $I,J_1,J_2 \in L(\FA)_\epsilon$ such that $I \supset J_i$ for $i=1,2$, we have $J_1 \cap J_2 \in L(\FA)_\epsilon$. Indeed, first we can reduce this to the case where $J_1 \cap J_2 = \emptyset$ by restricting to $\FA^{J_1 \cap J_2}$. Then using (i) we can write $I = J_i \sqcup I'_i$, where $\{ a_j \ |\ j\in I'_i\}$ are linearly independent in $F^m/ H_{J_i}$. As $J_1\cap J_2 = \emptyset$ we conclude that $J_2 \subset I'_1, J_1 \subset I'_2$ and hence $\{ a_i\ |\ i\in I\}$ are linearly independent in $F^m$ and $\delta_I = \delta_\emptyset$, which finishes the proof. 
\end{proof}

Next we write $l(\epsilon)$ for the length of a maximal chain in $L(\FA)_\epsilon$. Each maximal chain is of the form $I_0 \supsetneq I_1 \supsetneq \dots \supsetneq I_{l(\epsilon)} = J_i$ for some $1\leq i\leq d$, where $J_i$ is as in Lemma \ref{alstructure}. For a fixed $1 \leq i\leq d$ we write $J_{i,1},J_{i,2},\dots J_{i,k_i}$ for the different $I_0$'s that appear in a maximal chain with smallest flat $J_i$. Notice that again by Lemma \ref{alstructure}  $\{ a_e \ |\ e \in J_{i,j} \setminus J_i\}$ are $l(\epsilon)$ linearly independent vectors, hence any maximal chain from $J_{i,j}$ to $J_i$ is constructed by adding the $a_e$'s one by one. In particular, there are $l(\epsilon)!$ different maximal chains from $J_{i,j}$ to $J_i$.

\begin{prop}\label{pocrit} A negative integer $\epsilon \in \p_\FA \setminus \{-m\}$ is a pole of $\FI_\FA(s)$ of order $l(\epsilon) + 1$ if $\sum_{i,j} \nu_\FA(J_{i,j},\infty) \neq 0$.

%In particular $\epsilon = -n$ is always a pole of $\FI_\FA(s)$ of order $l(n)$.
\end{prop}
\begin{proof} We see directly from (\ref{igusaf}), that the order of the pole $\epsilon$ is less than or equal to $l(\epsilon)$ as $(q^{s-\epsilon}-1)^{-1}$ can appear at most $l(\epsilon)$ times in one summand. For $\epsilon$ to be exactly of order $l(\epsilon)$ is equivalent to 
\[R(q) = (q^{s-\epsilon}-1)^{l(\epsilon)+1}I_\FA(s)_{|s=\epsilon} \neq 0.\] 

We now study the contribution of a single summand corresponding to the chain $ \mathbf{I} = (\infty = I_0 \supsetneq I_1 \supsetneq \dots \supsetneq I_r)$ (notation as in (\ref{igusaf})). Clearly $\mathbf{I}$ has to contain a maximal chain from $L(\FA)_\epsilon$ in order to have a non-zero contribution to $R(q)$, so we assume this from now on. Then we can look at the contribution of $\mathbf{I}$ to the limit $\lim_{q \rightarrow \infty}(1-q^{-(\epsilon+m)}) R(q)$, where we include the factor $(1-q^{-(\epsilon+m)})$ to cancel the factor in front of the sum in (\ref{igusaf}) . The polynomial $\prod_{i=1 }^r \chi_{\FA_{I_{i-1}}^{I_{i}}}(q)$ is of degree $\rk(\FA^{I_r})$, which can be seen directly from the definition (\ref{chadef}). Furthermore if $I_r \notin L(\FA)_\epsilon$ the summand will contain a factor $(q^{\delta_{I_r}+\epsilon}-1)^{-1}$, which will tend to zero as $q \rightarrow \infty$. Combining these two facts and the description of maximal chains in $L(\FA)_\epsilon$ above leads to 
\begin{align*} \lim_{q \rightarrow \infty} (1-q^{-(\epsilon+m)})R(q) &= l(\epsilon)!\sum_{i,j} \sum_{\infty= I_0 \supsetneq I_1 \supsetneq \dots \supsetneq I_{l-1} \supsetneq J_{i,j}} (-1)^{l-1} \\
&= -l(\epsilon)! \sum_{i,j} \nu_\FA(J_{i,j},\infty).\end{align*}
Here we also used Lemma \ref{nuchain} for the last equality. By our assumption we then have $\lim_{q \rightarrow \infty} (1-q^{-(\epsilon+m)})R(q) \neq 0$, and hence also $R(q) \neq 0$. 
\end{proof}

\begin{rem} The proof of Proposition \ref{pocrit} considers a limit for $q \rightarrow \infty$ and it is natural to look also at a similar limit when $q \rightarrow 0$. It turns out however, that this will produce exactly the same criterion.
\end{rem}

In the next theorem we summarize all the information we have about the poles of $I_\FA(s)$.

\begin{thm}\label{ipoles} Let $\FA$ be a central arrangement of rank $m$ consisting of $n$ hyperplanes and $\epsilon$ one of the three numbers $-m,-n, \max \p_\FA \setminus \{-m\}$.  Then $\epsilon$ is a pole of $I_\FA(s)$ of order $l(\epsilon)+1$.
\end{thm}
\begin{proof} Assume first $\epsilon = -m$. In this case we can directly look at
\[\lim_{s \rightarrow -m} (q^{s+m}-1)^{l(m)+1}I_\FA(s),\]
where we consider the limit from above i.e. $s > -m$. As the characteristic polynomial of any hyperplane arrangement is monic we see that every summand in (\ref{igusaf}) tends either to $0$ or $+\infty$ individually (assuming $q$ is large enough), hence $-m$ is a pole of order $l(m)+1$.

For $\epsilon \in \left\{ -n, \max \p_\FA \setminus \{-m\} \right\}$ we use the criterion from Proposition \ref{pocrit}. Together with Proposition \ref{surpos} the statement will then follow from the fact that all $J_{i,j}$ in $L(\FA)_\epsilon$ have the same rank.

For $\epsilon = -n = \delta_\emptyset$ we can see this since $L(\FA)_{-n}$ consists exactly of those flats $I$ for which $\{a_i\}_{i \in I}$ are linearly independent. Hence the rank of each $J_{i,j}$ is simply $l(\epsilon)$.  

For $\epsilon = \max \p_\FA \setminus \{-m\}$ we first notice that there is a unique minimal flat $J \in L(\FA)_{-m}$. Indeed, $J$ is defined by the property $a_i \notin \spn_{k \neq i} \{a_k\}$ for all  $i\notin J$. %In other words $J$ is the smallest flat such that $\FA^J$ is not 
We can then see that $\rk(J) = \rk(\FA_J) = m - \rk(\FA^J) = m-l(-m)$ and since there cannot be any flat between $J$ and any $J_{i,j} \in L(\FA)_{\epsilon}$ we have $\rk(J_{i,j}) = m-l(-m)-1$.
\end{proof}

\subsection{The residue at the largest pole}\label{relap}

By Theorem \ref{ipoles} we know that $\FI_\FA(s)$ has its largest pole at $s=-m$. Furthermore we know that $-m$ is a simple pole if and only if $l(-m)=0$. In this case $\FA$ is called \textit{coloop-free} (see for example \cite[Remark 2.3]{PW07}) meaning there is no $a_i$ such that $ a_i \notin \spn_{j\neq i} \{a_j\}$. 

Recall that under these assumptions the residue $\Res_{s=-m}\FI_\FA(s)$ has an interesting interpretation. Namely if $\mu: \FO^n\times \FO^n \rightarrow \FO^m$ denotes the moment map associated with $\FA$ and $B_\mu$ the limit 
\[ B_\mu = \lim_{\alpha \rightarrow \infty} q^{-\alpha(2n-m)}  \left|\left\{ x \in \FO^{2n}_\alpha \ |\ \mu(x) = 0\right\}\right|,\]
then Lemma \ref{fresh} implies $ B_\mu = \frac{q^m}{q^m-1} \Res_{s=-m}\FI_\mu(s)$. Combining this with Theorem \ref{iguthm} gives 

\begin{cor} $B_\mu$ is finite if and only if $\FA$ is essential and coloop-free. In this case we have
\begin{equation}\label{bmu} B_\mu = \sum_{\infty = I_0 \supsetneq I_1 \supsetneq \dots \supsetneq I_r} q^{-\rk(\FA^{I_r})} \prod_{i=1 }^r \frac{\chi_{\FA_{I_{i-1}}^{I_{i}}}(q)}{q^{\delta_{I_i}-m}- 1},\end{equation}
where the sum is over all proper chains of flats in $L(\FA)$ of length $r\geq 0$.% and $\epsilon_I = \delta_I - m = n-|I| - \rk(\FA^I)$ for a flat $I \in L(\FA)$.
\end{cor}

The formula shows in particular that $B_\mu(q) \in \BZ(q)$ is a rational function in $q$ for $q$ large enough, see \ref{assum}. It is not hard to see that $B_\mu(q) $ has degree $0$, where the degree of a rational function is defined as the degree of the numerator minus the degree of the denominator. Indeed the contribution of $r=0$ in (\ref{bmu}) equals $1$ and all other summands have negative degree since $\delta_I-m = n-|I|+ \rk{A_I}-m= n-|I|-\rk(A^I) >0$ for any flat $I \neq \infty$, as $\FA$ is coloop-free.

\begin{expl} As in Example \ref{ncent} we start with the arrangement consisting of $n$-times the origin in $F$. Then (\ref{bmu}) has two terms, one for $\{\infty\}$ and one for $\infty \supsetneq \emptyset$ and we get
\[ B_\mu = 1 + \frac{q-1}{q(q^{n-1}-1)} = \frac{(q-1)(q^{n-1}+\dots+1)}{q(q^{n-1}-1)}.\]
\end{expl}

\begin{expl}\label{siml} The case of $m+1$ hyperplanes in $F^m$ in general position appeared already in Example \ref{m1m} and for $m=2$ in \ref{m2m}. It is also interesting to consider these arrangements with some hyperplanes doubled e.g. take $\FA$ as in \ref{m2m} and add the normal vector $a_4=a_3$. In this case we obtain 
\[ B_\mu = \frac{(q-1)^2(q^4+3q^3+6q^2+3q+1)}{q^2(q^2-1)^2}.\]
\end{expl}

\begin{expl} Finally, we consider again Example \ref{hugm}, in which case $B_\mu$ looks considerably nicer than $\FI_\FA(s)$:
\[ \frac{(q-1)^4(q^{10} + 4 q^{9} + 13 q^{8} + 35 q^{7} + 50 q^{6} + 58 q^{5} + 50 q^{4} + 35 q^{3} + 13 q^{2} + 4 q + 1)}{q^{3} (q^2 - 1) (q^{3}-1)^{3}}  \]
\end{expl}

These examples suggest that the "numerator" of $B_\mu$ satisfies some remarkable properties. In order to make this more precise define the numerator $B'_\mu(q)$ by multiplying $B_\mu(q)$ with what we expect to be its denominator i.e. 
 
\[ B'_\mu(q) = q^m\prod_{\epsilon \in \p_\FA \setminus \{-m\}} \left(\frac{q^{-\epsilon-m} -1}{q-1}\right)^{l(\epsilon)+1} B_\mu(q) \in \BZ[q].   \]

Notice that $B'_\mu(q)$ is indeed a polynomial since the characteristic polynomial of any central arrangement is divisible by $q-1$, see Corollary \ref{fq1}.

One thing we would expect from the examples, is that the numerator $B'_\mu(q)$ is palindromic, which is indeed the case. This is a consequence of the functional equation for the Igusa zeta function of a homogeneous polynomial\ref{feqi}. Since $\mu$ is in general not given by a single polynomial, we also rely on Remark \ref{mulv}.

\begin{prop} $B'_\mu(q)$ is a polynomial of degree $d= m + \sum_{ \epsilon \in \p_\FA \setminus \{-m\}} -(\epsilon+m)(l(\epsilon)+1)$ which is palindromic i.e.
\[ B'_\mu(q) = q^d B'_\mu(q^{-1}).\]
\end{prop}
 \begin{proof} We saw already, that $B_\mu(q)$ has degree $0$, hence the formula for $d$ follows directly from the definition of $B'_\mu(q)$. From the equation 
\[B_\mu(q) = \frac{q^m(q^{s+m}-1)}{q^m-1}\FI_\mu(s)_{|s=-m}\]
and the functional equation \ref{feqi} we deduce $B_\mu(q) = q^{-m} B_\mu(q^{-1})$, which now implies palindromicity of $B'_\mu(q)$.
\end{proof}

The examples above and some further computer evidence also suggest that the coefficients of $B'_\mu(q)$ are positive integers. Unfortunately we are unable to prove this at the moment and can only record it as a conjecture.

\begin{conj}\label{poscof} $B'_\mu(q)$ is a polynomial with positive coefficients.
\end{conj}

\subsection{Indecomposable quiver representations in higher depth}\label{irhd}

In this section let $\Gamma = (I,E)$ be a quiver as in \ref{nqv}. Most of what we say in this section will be independent of the orientation of $\Gamma$, hence we use the words quiver and graph interchangeably. By a subquiver $\Gamma'=(I',E')$ of $\Gamma$ we will always mean a quiver with $I'=I$ and $E' \subset E$.

For a dimension vector $\bfv \in \BN^I$ and $\alpha \geq 1$ we write $\Rep_{\alpha,\bfv}(\Gamma)$ for the free $\FO_\alpha$-module of representations of $\Gamma$ over $\FO_\alpha$ with dimension $\bfv$ i.e.
\[ \Rep_{\alpha,\bfv}(\Gamma) = \bigoplus_{e \in E} \Hom_{\FO_\alpha}(\FO_\alpha^{v_{s(e)}}, \FO_\alpha^{v_{t(e)}})\]

We are only considering $1$-dimensional representations here, that is $\bfv=(1,1,\dots,1)$, in which case we abbreviate $\Rep_{\alpha}(\Gamma) = \Rep_{\alpha,\bfv}(\Gamma) \cong \FO^n_\alpha$.

A representation $x \in \Rep_\alpha(\Gamma)$ is \textit{indecomposable} if it is not isomorphic to the sum of two representations with strictly smaller dimension vectors. We denote the subset of indecomposable representations by $\Rep_\alpha^{ind}(\Gamma)$. Since we are considering only $(1,\dots,1)$-dimensional representations, $x \in\Rep_\alpha(\Gamma)$ will be indecomposable if and only if the subquiver $\Gamma_{\alpha,x} = (I,E_{\alpha,x})$, where $E_{\alpha,x} =\{ e\in E \ |\ x_e \neq 0\}$, is connected.  

As in \ref{nqv} the group
\[G_\alpha = \prod_{i\in I} \GL_1(\FO_\alpha) \cong \left(\FO_\alpha^\times\right)^I\]
will act on $\Rep_\alpha(\Gamma)$ and two representations are isomorphic if and only if they lie in the same orbit under this action. The number $A_{\Gamma,\alpha}$ of indecomposable representations up to isomorphism is then given by 
\[A_{\Gamma,\alpha} = \left| \Rep_\alpha^{ind}(\Gamma)/G_\alpha \right|.\]

We now explain a formula for $A_{\Gamma,\alpha}$ which we learned from unpublished notes of Anton Mellit \cite{Me162}. First we have by Burnside's Lemma
\begin{equation}\label{burn} A_{\Gamma,\alpha} = |G_\alpha|^{-1} \sum_{x \in \Rep_\alpha^{ind}(\Gamma)} |\Aut(x)|.\end{equation}

To evaluate this sum we define for every $x \in \Rep_\alpha(\Gamma)$ a sequence 
\[\Gamma_{1,x}\subseteq \Gamma_{2,x} \subseteq \dots \subseteq \Gamma_{\alpha,x}\]
 of subquivers of $\Gamma$ by putting $\Gamma_{k,x} = (I,E_{k,x})$ and 
\[E_{k,x} = \left\{ e\in E\ |\ |x_e| > q^{-k}\right\}\] 
for $1 \leq k \leq \alpha$. Here the norm of an element in $\FO_\alpha$ is defined as the norm of any lift to $\FO$.

\begin{lemma} The number of automorphism of $x \in \Rep_\alpha(\Gamma)$ is given by 
\[ |\Aut(x)| = (q-1)^{c(\Gamma_{\alpha,x})} q^{\sum_{k=1}^{\alpha-1} c(\Gamma_{k,x})},\]
where we write $c(\Gamma')$ for the number of connected components of a graph $\Gamma'$.
\end{lemma}
\begin{proof} An element $g = (g_i) \in G_\alpha$ is an automorphism of $x$ if and only if for every $e \in E$ we have 
\begin{equation}\label{auteq} (g_{t(e)}-g_{s(e)})x_e=0.\end{equation}
In particular, we have $g_{t(e)}=g_{s(e)}$ whenever $x_e \in \FO_\alpha^\times \Leftrightarrow e \in E_{1,x}$. This proves the lemma in the case $\alpha=1$. For $\alpha >1$ denote by $\widetilde{\Gamma}$ the graph obtained from $\Gamma$ by contracting all the edges in $E_{1,x}$. Then $x$ and $g$ descend to a representation $\widetilde{x} \in \Rep_{\alpha}(\widetilde{\Gamma})$ and an automorphism $\widetilde{g} \in \Aut(\widetilde{x})$. Furthermore we have by construction $\widetilde{x} = \pi y$ for some $y \in \Rep_{\alpha-1}(\widetilde{\Gamma})$. This shows that the coefficient of $\pi^{\alpha-1}$ in $\widetilde{g}$ does not actually appear in (\ref{auteq}). As $\widetilde{\Gamma}$ is a graph on $c(\Gamma_{1,k})$ vertices we thus obtain the recursion
\[ |\Aut(x)| = q^{c(\Gamma_{1,k})} |\Aut(y)|\]
and the lemma follows by induction on $\alpha$.
\end{proof}

\begin{prop}\cite{Me162} 	The number of indecomposable representations of $\Gamma$ up to isomorphism over $\FO_\alpha$ is given by 

\begin{equation} \label{absind} A_{\Gamma,\alpha} = \sum_{\substack{\Gamma_1 \subseteq \dots \subseteq \Gamma_\alpha \subseteq \Gamma \\ c(\Gamma_\alpha)=1}} (q-1)^{b(\Gamma_\alpha)} q^{\sum_{k=1}^{\alpha-1} b(\Gamma_k)}, \end{equation}
where we write $b(\Gamma') = c(\Gamma') - V(\Gamma') + E(\Gamma')$ for the first Betti number of a graph $\Gamma'$.% In particular $A_{\Gamma,\alpha}$ is a polynomial of degree $\alpha b(\Gamma)$.
\end{prop}
\begin{proof} For a given sequence $\Gamma_1 \subseteq \dots \subseteq \Gamma_\alpha \subseteq \Gamma$ the number of $x \in \Rep_\alpha(\Gamma)$ with $\Gamma_{k,x} = \Gamma_k$ for all $1\leq k\leq \alpha$ is given by
\begin{align*} (q-1)^{E(\Gamma_1)}& q^{E(\Gamma_1)(\alpha-1)} (q-1)^{E(\Gamma_2)-E(\Gamma_1)}q^{\left(E(\Gamma_2)-E(\Gamma_1)\right)(\alpha-2)} \cdots\\ 
& \cdots (q-1)^{E(\Gamma_{\alpha-1})-E(\Gamma_{\alpha-2})}q^{E(\Gamma_{\alpha-1})-E(\Gamma_{\alpha-2})} (q-1)^{E(\Gamma_\alpha)-E(\Gamma_{\alpha-1})}\\
&= (q-1)^{E(\Gamma_\alpha)} q^{\sum_{k=1}^{\alpha-1} E(\Gamma_k)},
	\end{align*}
as the norm $|x_e|$ is prescribed for each $e \in E$ by the sequence of graphs. Since $|G_\alpha| = \left((q-1)q^\alpha\right)^{V(\Gamma)}$ and $V(\Gamma_k) = V(\Gamma)$ for all $1 \leq k \leq \alpha$ by definition, the formula follows from Burnside's Lemma (\ref{burn}).
\end{proof}

As in Section \ref{relap} we consider now the limit of $A_{\Gamma,\alpha}(q)$ as $\alpha \rightarrow \infty$ appropriately normalized. In order to do this we have to make sure, that the coefficient of the highest $q$-power in $A_{\Gamma,\alpha}$ doesn't grow as we vary $\alpha$, which will exactly be the case when $\Gamma$ is \textit{$2$-(edge)connected} i.e. when $\Gamma$ stays connected after removing any of its edges.

\begin{cor} $A_{\Gamma,\alpha}(q)$ is a polynomial in $q$ of degree $\alpha b(\Gamma)$.  The limit $A_\Gamma(q) = \lim_{\alpha \rightarrow \infty} q^{-\alpha b(\Gamma)} A_{\Gamma,\alpha}(q)$ converges if and only if $\Gamma$ is $2$-connected, in which case it is given by
\begin{equation}\label{cumb} A_\Gamma(q) = \left(1-q^{-1}\right)^{b(\Gamma)}\sum_{\Gamma'_1 \subsetneq \dots \subsetneq \Gamma'_{\beta-1} \subsetneq \Gamma'_{\beta} = \Gamma }  \prod_{j=1}^{\beta-1} \frac{1}{q^{b(\Gamma)-b(\Gamma'_{j-1})}-1}, \end{equation}
where the sum is over all strict chains of subgraphs of length $\beta \geq 1$.
\end{cor}

\begin{proof} To get the degree of $A_{\Gamma,\alpha}(q)$ we notice in general, that given two graphs $\Gamma' \subseteq \Gamma''$ on the same set of vertices we have $b(\Gamma') \leq b(\Gamma'')$ since adding an edge to a graph can at most connect two components. The leading coefficient of $A_{\Gamma,\alpha}$ then comes from summing up over all sequences with $b(\Gamma_1)=\dots =b(\Gamma_\alpha) = b(\Gamma)$. If there is any subgraph $\Gamma' \subsetneq \Gamma$ with $b(\Gamma') = b(\Gamma)$, then the number of such sequences will go to infinity as $\alpha \rightarrow \infty$. Hence $ \lim_{\alpha \rightarrow \infty} q^{-\alpha b(\Gamma)} A_{\Gamma,\alpha}(q)$ converges if and only if $b(\Gamma') < b(\Gamma)$ for any subgraph of $\Gamma$, which in turn is equivalent to $\Gamma$ being $2$-connected.

To compute $A_\Gamma(q)$ when $\Gamma$ is $2$-connected it will be convenient to rewrite (\ref{absind}) as a sum over strict chains $\Gamma'_1\subsetneq \Gamma'_2 \subsetneq \dots \subsetneq \Gamma'_\beta \subseteq \Gamma$. Explicitly given any chain $\Gamma_1 \subseteq \dots \subseteq \Gamma_\alpha \subseteq \Gamma $ there exists an integer $1 \leq \beta \leq \alpha$ and integers $0=l_0< l_1 < \dots < l_{\beta-1} < l_\beta = \alpha$ such that $\Gamma_{l_j+1} = \dots = \Gamma_{l_{j+1}}$ for $0 \leq j \leq \beta-1$. If we then put $\Gamma'_j = \Gamma_{l_j}$ for $1\leq j \leq \beta$ we can rewrite (\ref{absind}) as
\begin{align*} A_{\Gamma,\alpha}(q) &= \sum_{0<l_1<\dots <l_{\beta-1} <l_\beta=\alpha} \sum_{\substack{\Gamma'_1 \subsetneq \dots \subsetneq \Gamma'_{\beta} \subseteq \Gamma \\ c(\Gamma'_\beta)=1}} (1-q^{-1})^{b(\Gamma'_\beta)} q^{\sum_{j=1}^{\beta}(l_j - l_{j-1}) b(\Gamma'_j)}\\
&= \sum_{\substack{\Gamma'_1 \subsetneq \dots \subsetneq \Gamma'_{\beta} \subseteq \Gamma \\ c(\Gamma'_\beta)=1}}(1-q^{-1})^{b(\Gamma'_\beta)} \sum_{0<l_1<\dots <l_{\beta-1} <\alpha} q^{\alpha b(\Gamma'_\beta)} \prod_{j=1}^{\beta-1} q^{-l_j\left(b(\Gamma'_{j+1}) - b(\Gamma'_j)\right)}.
\end{align*}
From this we see that the only sequences $\Gamma'_1 \subsetneq \dots \subsetneq \Gamma'_{\beta} \subseteq \Gamma $ giving a non-zero contribution  to the limit $A_\Gamma(q) = \lim_{\alpha \rightarrow \infty} q^{-\alpha b(\Gamma)} A_{\Gamma,\alpha}(q)$ are the ones where $b(\Gamma'_\beta) = b(\Gamma)$ i.e. $\Gamma'_\beta = \Gamma$. The corollary then follows from the iterated geometric series summation
\[  \sum_{0<l_1<\dots <l_{\beta-1} <\infty} \prod_{j=1}^{\beta-1} q^{-l_j\left(b(\Gamma'_{j+1}) - b(\Gamma'_j)\right)} = \prod_{j=1}^{\beta-1} \frac{1}{q^{b(\Gamma)-b(\Gamma'_{j-1})}-1}.\]
\end{proof}

In order to compare $A_\Gamma(q)$  with $B_\mu(q)$ from the last section, we define a hyperplane arrangement $\FA = \FA(\Gamma)$ out of $\Gamma$. For every edge $e \in E$ define the hyperplane $H_e$ in $F^I$ by the equation $x_{s(e)} - x_{t(e)}$. This way, if $\Gamma$ is connected, we obtain an arrangement of rank $|I|-1$, since the $1$-dimensional subspace spanned by $(1,1, \dots,1)$ will be contained in all hyperplanes. Furthermore $\FA$ will be coloop-free if and only if $\Gamma$ is $2$-connected. 

Since formula \eqref{cumb} is essentially a sum over all possible chains of subgraphs of $\Gamma$ it is quite cumbersome to evaluate it even for small graphs. That is why we used Sage to compute the following examples.

\begin{expl} For $\Gamma$  the affine $\widetilde{A}_2$  i.e. $I=\{1,2,3\}$ and $E=\{(1,2),(2,3),(3,1)\}$ we get
\[  A_\Gamma(q) =  \frac{q^2+4q+1}{(q-1)^2}. \]
Similarly for $\Gamma = \widetilde{A}_3$ we have
\[ A_\Gamma(q) = \frac{q^3+11q^2+11q+1}{(q-1)^3}.\]
The values $B_\mu(q)$ for the corresponding arrangements were studied in Example \ref{m1m}.

One can also consider non simply-laced quivers, for example $\widetilde{A}_2$ with one edge doubled, which corresponds to Example \ref{siml}. In this case we we have 
\[ A_\Gamma(q) = \frac{q^4+3q^3+6q^2+3q+1}{(q^2-1)^2}.\] 
\end{expl}

The pattern in all the examples is of course that $A_\Gamma$ and $B_\FA$ seem to have very similar numerators, but slightly different denominators. We record this observation in the following conjecture.

\begin{conj}\label{lastone} For any $2$-connected graph $\Gamma =(I,E)$ with associated arrangement $\FA$ we have 
 \[ (q^{b(\Gamma)}-1)^{|I|-1}A_\Gamma(q) = B'_\FA(q). \]
\end{conj}

\newpage

\bibliographystyle{amsalpha}
\bibliography{reference}

\providecommand{\bysame}{\leavevmode\hbox to3em{\hrulefill}\thinspace}
\providecommand{\MR}{\relax\ifhmode\unskip\space\fi MR }
\providecommand{\MRhref}[2]{%
  \href{http://www.ams.org/mathscinet-getitem?mr=#1}{#2}
}
\providecommand{\href}[2]{#2}
\begin{thebibliography}{CBVdB04}

\bibitem[AB84]{AB84}
Michael~F Atiyah and Raoul Bott, \emph{The moment map and equivariant
  cohomology}, Topology \textbf{23} (1984), no.~1, 1--28.

\bibitem[AP16]{PA16}
Matthew Arbo and Nicholas Proudfoot, \emph{Hypertoric varieties and zonotopal
  tilings}, International Mathematics Research Notices \textbf{2016} (2016),
  no.~23, 7268--7301.

\bibitem[BB73]{bi73}
Andrzej Bialynicki-Birula, \emph{Some theorems on actions of algebraic groups},
  Annals of mathematics (1973), 480--497.

\bibitem[BD00]{BD00}
Roger Bielawski and A~Dancer, \emph{The geometry and topology of toric
  hyperkahler manifolds}, Communications in Analysis and Geometry \textbf{8}
  (2000), no.~4, 727--760.

\bibitem[BJL79]{BJL79}
Werner Balser, Wolfgang~B Jurkat, and Donald~A Lutz, \emph{{Birkhoff invariants
  and Stokes' multipliers for meromorphic linear differential equations}},
  Journal of Mathematical Analysis and Applications \textbf{71} (1979), no.~1,
  48--94.

\bibitem[Boa01]{Bo01}
Philip Boalch, \emph{Symplectic manifolds and isomonodromic deformations},
  Advances in Mathematics \textbf{163} (2001), no.~2, 137--205.

\bibitem[Boa11]{Bo11}
Philip~Paul Boalch, \emph{{Geometry and braiding of Stokes data; fission and
  wild character varieties}}, arXiv preprint arXiv:1111.6228 (2011).

\bibitem[Bog07]{Bo07}
Vladimir~I Bogachev, \emph{Measure theory}, vol.~1, Springer Science \&
  Business Media, 2007.

\bibitem[Bor12]{bo12}
Armand Borel, \emph{Linear algebraic groups}, vol. 126, Springer Science \&
  Business Media, 2012.

\bibitem[CBVdB04]{CV04}
William Crawley-Boevey and Michel Van~den Bergh, \emph{{Absolutely
  indecomposable representations and Kac-Moody Lie algebras}}, Inventiones
  mathematicae \textbf{155} (2004), no.~3, 537--559.

\bibitem[CL10]{CL}
Raf Cluckers and Fran{\c{c}}ois Loeser, \emph{Constructible exponential
  functions, motivic fourier transform and transfer principle}, Annals of
  mathematics \textbf{171} (2010), no.~2, 1011--1065.

\bibitem[CLL13]{CLL}
Antoine Chambert-Loir and Fran{\c{c}}ois Loeser, \emph{Motivic height zeta
  functions}, arXiv preprint arXiv:1302.2077 (2013).

\bibitem[Del71]{De71}
Pierre Deligne, \emph{{Th{\'e}orie de Hodge: II}}, Publications
  Math{\'e}matiques de l'Institut des Hautes {\'E}tudes Scientifiques
  \textbf{40} (1971), no.~1, 5--57.

\bibitem[Del74]{De74}
\bysame, \emph{{Th{\'e}orie de Hodge: III}}, Publications math{\'e}matiques de
  l'IH{\'E}S \textbf{44} (1974), 5--77.

\bibitem[Den90]{De90}
Jan Denef, \emph{{Report on Igusa's local zeta function}}, S{\'e}minaire
  Bourbaki \textbf{33} (1990), 359--386.

\bibitem[DH82]{DU82}
Johannes~J Duistermaat and Gerrit~J Heckman, \emph{On the variation in the
  cohomology of the symplectic form of the reduced phase space}, Inventiones
  mathematicae \textbf{69} (1982), no.~2, 259--268.

\bibitem[DM91]{DM91}
Jan Denef and Diane Meuser, \emph{{A functional equation of Igusa's local zeta
  function}}, American Journal of Mathematics \textbf{113} (1991), no.~6,
  1135--1152.

\bibitem[Dol03]{Do03}
Igor Dolgachev, \emph{Lectures on invariant theory}, no. 296, Cambridge
  University Press, 2003.

\bibitem[G\"01]{Go01}
Lothar G\"ottsche, \emph{{On the motive of the Hilbert scheme of points on a
  surface}}, Mathematical Research Letters (2001).

\bibitem[Hau06]{Hau1}
Tam{\'a}s Hausel, \emph{{Betti Numbers of holomorphic symplectic quotients via
  aritmetic Fourier transform}}, Proceedings of the National Academy of
  Sciences of the United States of America \textbf{103} (2006), no.~16,
  6120--6124.

\bibitem[Hau10]{Hau2}
\bysame, \emph{{Kac's conjecture from Nakajima quiver varieties}}, Inventiones
  mathematicae \textbf{181} (2010), no.~1, 21--37.

\bibitem[HK09]{hk09}
Ehud Hrushovski and David Kazhdan, \emph{{Motivic Poisson summation}}, Mosc.
  Math. J \textbf{9} (2009), no.~3, 569--623.

\bibitem[HLRV11]{HLV11}
Tam{\'a}s Hausel, Emmanuel Letellier, and Fernando Rodriguez-Villegas,
  \emph{Arithmetic harmonic analysis on character and quiver varieties}, Duke
  Math. J \textbf{160} (2011), no.~2, 323--400.

\bibitem[HLRV13]{HLV13}
\bysame, \emph{Arithmetic harmonic analysis on character and quiver varieties
  ii}, Advances in Mathematics \textbf{234} (2013), 85--128.

\bibitem[HMW16]{HMW16}
Tamas Hausel, Martin Mereb, and Michael~Lennox Wong, \emph{Arithmetic and
  representation theory of wild character varieties}, arXiv preprint
  arXiv:1604.03382 (2016).

\bibitem[HMY07]{HMY07}
Jason Howald, Mircea Musta{\c{t}}a, and Cornelia Yuen, \emph{{On Igusa zeta
  functions of monomial ideals}}, Proceedings of the American Mathematical
  Society \textbf{135} (2007), no.~11, 3425--3433.

\bibitem[HRV08]{HR08}
Tam{\'a}s Hausel and Fernando Rodriguez-Villegas, \emph{{Mixed Hodge
  polynomials of character varieties}}, Inventiones mathematicae \textbf{174}
  (2008), no.~3, 555--624.

\bibitem[HS02]{HS02}
Tam{\'a}s Hausel and Bernd Sturmfels, \emph{Toric hyperk{\"a}hler varieties},
  Doc. Math \textbf{7} (2002), 495--534.

\bibitem[Hua00]{Hua}
J.~Hua, \emph{Counting representations of quivers over finite fields}, J.
  Algebra \textbf{226} (2000), 1011--1033.

\bibitem[HWW]{HWW17}
Tamas Hausel, Michael Lennox~Wong Wong, and Dimitri Wyss, \emph{{Arithmetic and
  metric aspects of open de Rham spaces}}, In preparation.

\bibitem[Igu74]{Ig74}
Jun-ichi Igusa, \emph{Complex powers and asymptotic expansions. i. functions of
  certain types.}, Journal f{\"u}r die reine und angewandte Mathematik
  \textbf{268} (1974), 110--130.

\bibitem[Igu88]{Ig88}
J.-i. Igusa, \emph{B-functions and $p$-adic integrals}, Algebraic analysis,
  Academic press (1988).

\bibitem[Igu00]{Ig00}
Jun-ichi Igusa, \emph{{An introduction to the theory of local zeta functions}},
  Cambridge Univ Press, 2000.

\bibitem[Kin94]{Kin94}
Alastair~D King, \emph{Moduli of representations of finite dimensional
  algebras}, The Quarterly Journal of Mathematics \textbf{45} (1994), no.~4,
  515--530.

\bibitem[Loe89]{Lo89}
Fran{\c{c}}ois Loeser, \emph{{Fonctions z{\^e}ta locales d'Igusa {\`a}
  plusieurs variables, int{\'e}gration dans les fibres, et discriminants}},
  Annales scientifiques de l'{\'E}cole Normale Sup{\'e}rieure, vol.~22, 1989,
  pp.~435--471.

\bibitem[Mac98]{Mac}
Ian~Grant Macdonald, \emph{{Symmetric functions and Hall polynomials}}, 1998.

\bibitem[Mel16]{Me162}
Anton Mellit, \emph{Counting 1-dimensional representations}.

\bibitem[Mer13]{Mer13}
A~Merkurjev, \emph{Essential dimension: a survey}, Tansformation groups
  \textbf{18} (2013), no.~2, 415--481.

\bibitem[MFK94]{MFK94}
David Mumford, John Fogarty, and Frances~Clare Kirwan, \emph{Geometric
  invariant theory}, vol.~34, Springer Science \& Business Media, 1994.

\bibitem[MW74]{MW74}
Jerrold Marsden and Alan Weinstein, \emph{Reduction of symplectic manifolds
  with symmetry}, Reports on mathematical physics \textbf{5} (1974), no.~1,
  121--130.

\bibitem[Nak94]{nak94}
Hiraku Nakajima, \emph{{Instantons on ALE spaces, quiver varieties, and
  Kac-Moody algebras}}, Duke Mathematical Journal \textbf{76} (1994), no.~2,
  365--416.

\bibitem[Nak98]{nak98}
\bysame, \emph{{Quiver varieties and Kac-Moody algebras}}, Duke Mathematical
  Journal \textbf{91} (1998), no.~3, 515--560.

\bibitem[Nak99]{Na99}
\bysame, \emph{{Lectures on Hilbert schemes of points on surfaces}}, vol.~18,
  American Mathematical Society Providence, RI, 1999.

\bibitem[Pet15]{Pe15}
T~Kyle Petersen, \emph{{Eulerian numbers}}, Eulerian Numbers, Springer, 2015,
  pp.~3--18.

\bibitem[Pop]{Po11}
Mohnea Popa, \emph{{Modern aspects of the cohomological study of varieties,
  Chapter 6}}, http://www.math.northwestern.edu/~mpopa/571/.

\bibitem[Pro07]{Pro07}
Nicholas~J Proudfoot, \emph{A survey of hypertoric geometry and topology},
  arXiv preprint arXiv:0705.4236 (2007).

\bibitem[PW07]{PW07}
Nicholas~J Proudfoot, , and Ben Webster, \emph{Arithmetic and topology of
  hypertoric varieties}, J. Alg. Geom. (2007).

\bibitem[Rei03]{Re03}
Markus Reineke, \emph{{The Harder-Narasimhan system in quantum groups and
  cohomology of quiver moduli}}, Inventiones mathematicae \textbf{152} (2003),
  no.~2, 349--368.

\bibitem[Ser58]{serre58}
J-P Serre, \emph{Espaces fibr{\'e}s alg{\'e}briques}, S{\'e}minaire Claude
  Chevalley \textbf{3} (1958), 1--37.

\bibitem[Ser13]{Se13}
Jean-Pierre Serre, \emph{Local fields}, vol.~67, Springer Science \& Business
  Media, 2013.

\bibitem[Sta04]{St04}
Richard~P Stanley, \emph{An introduction to hyperplane arrangements}, Geometric
  combinatorics \textbf{13} (2004), 389--496.

\bibitem[Sum74]{su74}
Hideyasu Sumihiro, \emph{Equivariant completion}, J. Math. Kyoto Univ.
  \textbf{14} (1974), no.~1, 1--28.

\bibitem[Tai75]{Ta75}
Mitchell~H Taibleson, \emph{{Fourier Analysis on Local Fields.}}, Princeton
  University Press, 1975.

\bibitem[VZG08]{VZ08}
Willem Veys and W~Zuniga-Galindo, \emph{{Zeta functions for analytic mappings,
  log-principalization of ideals, and Newton polyhedra}}, Transactions of the
  American Mathematical Society \textbf{360} (2008), no.~4, 2205--2227.

\bibitem[Wei12]{WE12}
Andr{\'e} Weil, \emph{Adeles and algebraic groups}, vol.~23, Springer Science
  \& Business Media, 2012.

\bibitem[Wel07]{We07}
Raymond~O Wells, \emph{Differential analysis on complex manifolds}, vol.~65,
  Springer Science \& Business Media, 2007.

\end{thebibliography}
\end{document}